\documentclass[11pt]{amsart}
\usepackage{amsmath}
\usepackage{amssymb, verbatim}
\usepackage{pstricks,pst-node,pst-plot}
\usepackage{comment}
\usepackage{color}
\usepackage{extarrows}
\usepackage{slashed}

 \title[]{{\Large Spinor flows with flux, I} \\
 short-time existence and smoothing estimates}
\author[T. C. Collins]{Tristan C. Collins}
  \email{tristanc@mit.edu}
  \address{Department of Mathematics, Massachusetts Institute of Technology, 77 Massachusetts Avenue, Cambridge, MA 02139}
\thanks{T.C.C. is supported in part by NSF CAREER grant DMS-1944952 and an Alfred P. Sloan Fellowship. } 

 \author[D. H. Phong]{Duong H. Phong}
  \email{phong@math.columbia.edu}
 \address{Department of Mathematics, Columbia University,
   New York, NY 10027}
  \thanks{D.H.P. is supported in part by NSF grant DMS-18455.}

\theoremstyle{plain}
\newtheorem{thm}{Theorem}[section]
\newtheorem{prop}[thm]{Proposition}

\newtheorem{lem}[thm]{Lemma}
\newtheorem{cor}[thm]{Corollary}

\theoremstyle{definition}

\newtheorem{rk}[thm]{Remark}

\numberwithin{equation}{section}

\newcommand{\del}{\partial}

\newcommand{\cE}{\mathcal{E}}

\newcommand{\bR}{\mathbb{R}}

\newcommand{\na}{\nabla}

\newcommand{\p}{\partial}
\newcommand{\be}{\begin{equation}}
\newcommand{\bea}{\begin{eqnarray}}
\newcommand{\eea}{\end{eqnarray}} 
\newcommand{\ee}{\end{equation}}

\renewcommand{\leq}{\leqslant}
\renewcommand{\geq}{\geqslant}

\renewcommand{\epsilon}{\varepsilon}
\renewcommand{\phi}{\varphi}

\renewcommand{\Re}{{\rm Re}}

\begin{document}
\maketitle

\begin{abstract}

{\footnotesize 
Spinor fields which are covariantly constant with respect to a connection with flux are of particular interest in unified string theories and supergravity theories,
as their existence is required by supersymmetry. In this paper, flows 
of spinor fields are introduced which are parabolic and whose stationary points are such covariantly constant spinors.}

\end{abstract}

\section{Introduction}
\setcounter{equation}{0}

The physical laws governing the four known forces of nature have long been a powerful engine for the development of the theory of geometric partial differential equations. It is now well-known that electromagnetism, weak and strong interactions are described by gauge theories and the Maxwell and Yang-Mills equations, while general relativity is described by Riemannian geometry and Einstein's equation. Since the mid 1980's, the search for a unified theory of all interactions has received a particular impetus, and with the emergence of string theories and subsequently M Theory in the mid 1990's \cite{BBS}, ever more sophisticated geometric structures have been introduced. Of these, the earliest and most basic may be Calabi-Yau manifolds \cite{CHSW}, with subsequent extensions required by the heterotic string and Hull-Strominger system \cite{S, Hu, FY, PPZ, PPZ2}, the Type IIA string \cite{FPPZ}, the Type II B string \cite{FPPZb, CJY, CY, CY1}, and $11$-dimensional supergravity \cite{CJS, FGP} (see e.g. \cite{DH, Ga1, Ga2, Gu, P} for surveys from different viewpoints). While the different string theories and $11$-dimensional supergravity are different limits of the same M Theory, their formulations are quite different from each other, with settings ranging from real to complex, to almost-complex, and to symplectic geometry. Nevertheless, these different geometric structures all originate from the same requirement, namely that space-time be supersymmetric.

\smallskip
The infinitesimal generator of supersymmetry in a theory incorporating gravity is a spinor field, and unbroken supersymmmetry requires the existence of a covariantly constant spinor. However, the notion of covariant derivative is not necessarily the Levi-Civita spin connection, but rather a modification of it by a Clifford algebra valued closed form which is another field in the theory and is known as {\it flux}. Covariantly constant spinors and their consequences for holonomy are of course very familiar in differential geometry (see e.g. 
\cite{LM, Harv, F,McKW, McKW2}). What is novel in our situation, and is required by the physics of unified interactions, is this possible presence of flux.

\smallskip
The main goal of the present paper is to develop a flow approach to the problem of finding a metric and spin structure admitting a non-trivial covariantly constant spinor with flux. It should be stressed that this problem is quite different from, and in a sense complementary to, the more familiar one of finding solutions to the field equations of a given physical theory. For example, a flow approach to the field equations of $11$-dimensional supergravity had been considered in \cite{FGP}, but the supersymmetry of the solution, and in particular whether the solution admitted covariantly constant spinors, was not addressed there. Solutions of the field equations can of course be supersymmetric or not supersymmetric, with explicit examples of both given in e.g.  \cite{FGP0}. Returning to the problem of covariantly constant spinors, in the simpler case of no flux, a natural proposal by Ammann, Weiss, and Witt \cite{AWW} was to consider the gradient flow of the energy functional
\bea
I(g_{ij},\psi)=\int_M|\na \psi|^2\sqrt g dx
\eea
where $(M,g_{ij})$ is a compact Riemannian manifold with a spin structure,
and $\psi$ a spinor field subject to the normalization $|\psi|^2=1$. It is shown in \cite{AWW} that this flow admits a short-time solution, and that its stationary points are exactly covariantly constant spinors of norm $1$. Shi-type estimates for this flow have been established by He and Wang \cite{HeWang}.
However, in presence of flux, the spin connection $\na$ gets replaced by the covariant derivative $\na^H$
\bea
\label{nablaH}
\na^H_a\psi&:=&\na_a\psi+(\lambda_1H_{ab_1\cdots b_{k-1}}\gamma^{b_1\cdots b_{k-1}}+\lambda_2H_{b_1\cdots b_k}\gamma^{ab_1\cdots b_k})\psi
\nonumber\\
&:=&
\na_a\psi+(\lambda|H)_a\psi
\eea
where $H_{b_1\cdots b_k}$ is a $k$-form and $\lambda=(\lambda_1,\lambda_2)$ are given coupling constants. The gradient flow of the $L^2$ norm of $\na^H\psi$ is degenerate in $H$, and its short-time existence does not appear guaranteed by any known method. A related, but additional difficulty is that the connection with flux is usually not unitary, and there is no indicated normalization that would prevent the spinor field $\psi$ from simply converging to the zero section as time evolves.

\smallskip
In this paper, we examine both kinds of scenarios, one where the normalization of $\psi$ is dynamical, and the other where it is given in advance. The second one is analytically easier, and can be viewed as a simpler model for the first one. In both cases, we still need to impose conditions of the flows of the fluxes in order to arrive at parabolic flows. In supergravity theories, the fluxes of interest will satisfy the field equations. Thus it is natural to consider flows where the stationary points would satisfy such equations.
The main result of the present paper is that these considerations identify classes of spinor flows whose stationary points are covariantly constant spinors with flux
which are weakly parabolic, admit short-time existence, and satisfy Shi-type estimates. They are in fact particularly well-behaved, as bounds on $\psi$ and $H$ can be shown to imply bounds on the Riemann curvature. The precise statements are as follows.

\smallskip
Let $M$ be a compact $n$-dimensional spin manifold, $S$ be the bundle of spinors on $M$. Fix a positive integer $k$. We consider $4$-tuples of the form $(g_{ab},\psi,H,\varphi)$
where $g_{ab}$ is a metric on $M$, $\psi$ is a section of $S$, $H$ is a $k$-form and $\phi$ is a function satisfying the normalization condition
\bea
\label{normalization}
e^{-2\varphi}|\psi|^2=1.
\eea
Although the discussion below applies in more general situations, to be specific and with the field equations of $5$ and $11$-dimensional supergravities in mind, we assume that $3k=n+1$. As discussed above, we consider both flows where the normalization $\varphi$ is dynamical and where it is given and fixed. In the first case, we consider flows of $4$-tuples $(g_{ij},\psi, H,\varphi)$ of the form
\begin{equation}
\begin{aligned}
\label{flow}
\dot{g} &= -\frac{1}{4}e^{-2\phi}{\rm Re}\left(Q_{\{p\ell\}}^H\right) +\frac{1}{2} \langle \nabla_{\{p}^{H}\psi, \nabla_{\ell \}}^{H}\psi\rangle - \frac{1}{4} g_{p\ell}|\nabla^{H}\psi|^2\\
\dot{\psi} &= -(\nabla_a^{H})^{\dagger}\nabla^{H}_{a} \psi  + c(t) \psi\\
\dot{H} &= -\left(d^{\dagger}dH + d(d^{\dagger}H-c\star(H\wedge H)) - L(H)\right)\\
\dot{\phi} &=\sum_a e^{-2\phi}(\nabla_{a}-2e^{-2\varphi}\langle h_a\psi,\psi\rangle)\big\{e^{2\phi}(\nabla_{a}\phi +e^{-2\phi}\langle h_{a}\psi, \psi \rangle)\big\}\\
\end{aligned}
\end{equation}
with the various notions entering this formula defined as follows:

\smallskip

$\bullet$ We work in a local trivialization of the spin bundle $S\rightarrow M$, which projects on a local trivialization of the bundle of orthonormal frames, given by a choice of orthonormal frames $\{e_a\}=\{e_a{}^j\}$ in a local coordinate system $\{x^j\}$. The corresponding Dirac matrices are denoted by $\gamma^a$, and our conventions for them and for the curvature tensor can be found in the Appendix. All connections and covariant derivatives are taken with respect to the evolving metric $g_{ab}$.  The Clifford algebra defined by the time varying metrics $g(t)$ acts on the spinor bundle $S$ via parallel transport by the Bourguignon-Gauduchon connection.  We refer the reader to Appendix~\ref{sec: BGApp}.

\smallskip

$\bullet$ $\na^H$ is the connection defined in (\ref{nablaH}), where $\lambda_1,\lambda_2$, as well as the coefficient $c$ in the third equation of~\eqref{flow} and in the definition of $L(H)$ below, are given coupling constants.  $(\nabla^{H})^{\dagger}$ is the formal $L^2$ adjoint of $\nabla^{H}$, and $h_a$ denotes the symmetrization of $(\lambda|H)_{a}$ as an operator on spinors,
\footnote{We observe that with the conventions for the Clifford algebra spelled out in Appendix, and the flux $H$ given by a $k$-form with $3k=n+1$, then the coefficient of $\lambda_1$ is Hermitian if and only if the coefficient of $\lambda_2$ is anti-Hermitian, and vice versa. Thus $(\lambda|H)_j$ is not Hermitian unless one of the two coefficients $\lambda_1$ or $\lambda_2$ is $0$, and the hermitian part $h_j$ reduces to 
the contribution of the $\lambda_1$ term $k-1=0,1$ mod $4$, and to the contribution of the $\lambda_2$ term if $k+1=0,1$ mod 4.}
\bea
\label{h}
h_a={1\over 2}((\lambda|H)_{a}+(\lambda|H)_{a}^\dagger).
\eea

 $\bullet$ The tensor $Q_{p\ell}^H$, the scalar time-dependent coefficient $c(t)$, and the $k$-form $L(H)$ are defined respectively by
 \bea
&&
 Q_{p\ell}^H= \nabla_{a}\langle \gamma^{a\ell}\psi, \nabla_{p}^{H}\psi \rangle
 \\
 &&
 c(t)=e^{-2\phi}|\nabla^{H}\psi|^2
 -2e^{-2\phi} \sum_{a}\langle h_{a}\psi, \psi \rangle(\nabla_{a}\phi +e^{-2\phi}\langle(h_{a}\psi, \psi \rangle) \label{eq: flowConst}
 \\
 &&
  \langle L(H), \beta \rangle = \frac{1}{2}\langle c\star(\beta \wedge H + H \wedge \beta), d^{\dagger}H-c\star(H\wedge H)\rangle \label{eq: L(H)def}
 \eea
 for any $k$-form $\beta$. Finally, the brackets $\{p \ell\}$ denote symmetrization in these indices.
 
 \medskip
 In the second setting, we fix a smooth normalization $\varphi$, and consider flows of $3$-tuples $(g_{ij},\psi,H)$ given by
 \begin{equation}
\begin{aligned}
\label{flow1}
\dot{g} &= -\frac{1}{4}e^{-2\phi}{\rm Re}\left(Q_{\{p\ell\}}^H\right) + \frac{1}{2}\langle \nabla_{\{p}^{H}\psi, \nabla_{\ell \}}^{H}\psi\rangle - \frac{1}{4} g_{p\ell}|\nabla^{H}\psi|^2\\
\dot{\psi} &= -(\nabla_a^{H})^{\dagger}\nabla^{H}_{a} \psi  + c(t) \psi\\
\dot{H} &= -\left(d^{\dagger}dH + d(d^{\dagger}H-c\star(H\wedge H)) - L(H)\right)
\end{aligned}
\end{equation}
where the coefficient $c(t)$ is now given by
\begin{equation}\label{eq: flow1Const}
 c(t):=  e^{-2\phi}|\nabla^H\psi|^2-e^{-2\phi}\nabla_{a}\left(e^{2\phi}(\nabla_{a}\phi +  e^{-2\phi}\langle h_{a}\psi, \psi \rangle\right)
\end{equation}

 \medskip
We can now state the main theorems:
 
 \begin{thm}
 \label{Main1}
 
 Consider the flows (\ref{flow}) and (\ref{flow1}), under the above respective hypotheses. Then in either case,
 the normalization condition (\ref{normalization}) is preserved at all times. Furthermore,
 
 {\rm (a)} At any stationary point, the spinor field $\psi$ is covariantly constant with respect to the connection $\na^H$, and the flux satisfies the equation
 \bea
 \label{equationH}
 dH=0 \quad d^{\dagger}H = c\star H\wedge H.
 \eea
 
 {\rm (b)} The flow (\ref{flow}) is weakly parabolic, and it always exists for some positive time interval, for any smooth initial data satisfying the normalization constraint (\ref{normalization}).
 
 {\rm (c)} The following Shi-type estimates hold along the flow~\eqref{flow} with dynamical $\phi$: assume that on the time interval $(0, {\tau\over A}]$ we have the following bound
 \[
 |\na\psi|^2,|\na^2\psi|,|H|^2,|\na H|, |Rm|\leq A, \qquad \|\phi\|_{L^{\infty}} \leq \hat{A} 
 \]
 Then for any integer $p\geq 0$, there is a constant $C_p$, depending only on $p, \dim M, \hat{A}$ and an upper bound for $\tau$ so that
 \bea
 |\na^p Rm|+|\na^{p+2}\psi|+|\na^{p+2}e^{2\varphi}|+|\na^{p+1}H|\leq {C_pA\over t^{p\over 2}}.
 \eea
 uniformly on $[0,\frac{\tau}{A})$
 

 \end{thm}
 
Shi-type estimates analogous to those in Theorem~\ref{Main1} part (c) for flows with fixed $\phi$, that is, given by~\eqref{flow1} are also established, but with a slightly different statement; we refer the reader to Theorem~\ref{thm: ShiEstConstPhi} below.  We also remark that in the Shi-type estimates in Theorem~\ref{Main1} part (c) the condition $\|\phi\|_{L^{\infty}} \leq \hat{A}$ can be replaced by a more geometric condition concerning $C^1$ bounds for the time dependent $1$-form $\mathfrak{A}_{a} =e^{-2\phi}{\rm Re}\left(\langle H_a \psi, \psi\rangle\right)$.  Note that, thanks to the normalization condition $e^{-2\phi}|\psi|^2=1$, the form $\mathfrak{A}_{a}$ depends only on the section of the projectivized spinor bundle $\mathbb{P}(S)$ induced by the non-vanishing spinor $\psi$.  We refer the reader to Section~\ref{sec: prelimEst} for results in this direction.

The equation (\ref{equationH}) for the flux $H$ arises from supergravity theories in $5$ and $11$ dimensions, and the flows (\ref{flow}) and (\ref{flow1}) were purposely designed so as to be weakly parabolic and insure this equation in the limit. We expect that our methods are be adaptable to other equations for the flux which may arise in applications. As a coupled system, the flows (\ref{flow}) and (\ref{flow1}) also have the following attractive property:

\begin{thm}
\label{Main2}
Assume that along either of the flows (\ref{flow}) and (\ref{flow1}) the follows bounds hold
\bea
|\na\psi|^2,|\na^2\psi|,|H|^2, |\na H|,|\na^2 H|^{\frac{2}{3}}\leq A
\eea
and $\|\phi\|_{L^{\infty}} \leq \hat{A}$ on some time interval $[0,\tau)$. Then there exists a constant $C$, depending only $A, \hat{A}, \dim M$, an upper bound for $\tau$ and 
\[
\kappa:= \int_{M} |Rm(0)|^{p}\sqrt{g(0)}+ {\rm Vol}(M, g(0)),
\]
such that $|Rm(t)|_{g(t)} \leq C$ uniformly on $[0,\tau)$.
\end{thm}

\begin{rk}
By Theorem~\ref{Main2}, the bound on the Riemann tensor in the Shi-type estimates of Theorem~\ref{Main1} can be replaced by a bound on $|\nabla^2H|$.
\end{rk}

\bigskip
This paper is organized as follows. In \S\ref{sec: Flux}, we identify the conditions on flows which preserve the normalization (\ref{normalization}), and show that they are satisfied by the flow (\ref{flow}).  Furthermore, in \S\ref{sec: Flux} we prove part (a) of Theorem~\ref{Main1}. In \S\ref{sec: shortTime}, we derive the basic formulas for the variations of the spin connection $\na^H$ under a variation of metric and of flux $H$.  An important ingredient is the variation of the spin connection under the Bourguignon-Gauduchon connection obtained in \cite{AWW}, which we re-derive here. Part(b) of Theorem~\ref{Main1}, namely the short-time existence, are proved in \S\ref{sec: shortTime}.  Furthermore, in \S\ref{sec: shortTime} we explain how the flows~\eqref{flow} and~\ref{flow1} are gauge equivalent to a coupled Ricci-flow system.  An analogous result for the spinor flow introduced in \cite{AWW} is obtained in \cite{HeWang}.  In \S\ref{sec: prelimEst} we establish some preliminary estimates for $\phi$ and the volume along the flow.  Here it is observed that bounds for $\phi$ along the flow~\eqref{flow} are closely related to bounds for the $1$-form $\mathfrak{A}_a:= e^{-2\phi}{\rm Re}\left(\langle H_a \psi, \psi \rangle\right)$, which we view as a connection $1$-form on a trivial real line bundle over $M$ depends only on the section of the projectivized spinor bundle $\mathbb{P}(S)$ induced by the non-vanishing spinor $\psi$ .   The Shi-type estimates are proved in \S\ref{sec: Shi}. Because the diffusion operator in the spinor flows are not diagonal, it seems difficult to get Shi-type estimates using only the maximum principle, as in the Ricci and other flows in the literature.  Instead we rely on proving integral estimates, which permit us to exploit integration by parts heavily; such techniques have been instrumental in the recent works \cite{KMW, HeWang}.  We then obtain Shi-type estimates in sup-norm, as stated in Theorem~\ref{Main1} by appealing to the Sobolev embedding theorem. Theorem~\ref{Main2} is proved in \S~\ref{sec: RmBd} building on ideas introduced in \cite{KMW, HeWang}. Finally, we have included an Appendix for the readers' convenience which explains the notation and conventions used throughout the paper, and recounts some of the basic theory of spinors.

We end the introduction by noting that in an interesting different direction, parallel spinors on globally hyperbolic Lorentzian four-manifolds have recently been studied by Murcia-Shahbazi \cite{MuSh, MuSh1} building on work of Cort\'es-Lazaroiu-Shahbazi \cite{CLS}.  In particular, in \cite{MuSh} the authors show that the existence of a globally defined parallel spinor in a globally hyperbolic Lorentzian 4 manifold induces a solution of a hyperbolic flow called the ``parallel spinor flow" on a Cauchy hypersurface, and vice versa.

\section{Fluxes and normalizations}\label{sec: Flux}

We begin by discussing normalization conditions in the presence of fluxes. The considerations are simple, but they play a major part in the choice of the flow.

\subsection{Normalization constraints on the flows of $\psi$ and $\varphi$}\hfil\break
Consider then a time-dependent smooth spinor field $\psi$ on a spin manifold $M$, and assume that it is never zero at any point, so that the normalization condition (\ref{normalization}) holds for a suitable time-dependent scalar function $\varphi$.  Note that any non-zero spinor field $\psi$ satisfying $\nabla^H\psi=0$ is necessarily non-vanishing, as follows by integrating the parallel transport along any curve in $M$.   
Differentiating the condition (\ref{normalization}) with respect to time gives
\bea
-\dot\varphi+e^{-2\varphi}\Re \langle\dot\psi,\psi\rangle=0.
\eea
If we assume that the flow of $\psi$ is of the form
\bea
\dot\psi=-(\na_a^H)^\dagger\na_a^H\psi+c(t)\psi
\eea
then the normalization (\ref{normalization}) holds at all times if and only if it holds at the initial time and
\bea
\label{cond-normalization}
-\dot\varphi+c(t)-e^{-2\phi}\Re \langle (\na^H)^\dagger\na^H\psi,\psi\rangle=0.
\eea

\medskip

The key observation is that, using the normalization condition (\ref{normalization}), the Laplacian $(\na^H)^\dagger\na^H\psi$ can be converted into a Laplacian on $\varphi$, up to lower order terms. More specifically, on one hand, we can
use the following pointwise version of the usual integration by parts formula giving $\na^\dagger$,
\bea
\langle \psi,(\na^H)^\dagger\na^H\psi\rangle
&=&
|\na^H\psi|^2-\na_a\langle\psi,\na_a^H\psi\rangle
\eea
and after taking real parts,
\bea
\Re
 \langle (\na^H)^\dagger\na^H\psi,\psi\rangle
=
|\na^H\psi|^2-\na_a\Re \langle\psi,\na_a\psi\rangle-\na_a\langle\psi,h_a\psi\rangle.
\eea
On the other hand, differentiating the normalization condition with respect to space gives
\bea
\label{dvarphi}
\p_a\varphi=e^{-2\varphi}\Re\langle\psi,\na_a\psi\rangle.
\eea
Thus we find
\bea
\Re
 \langle (\na^H)^\dagger\na^H\psi,\psi\rangle
=
|\na^H\psi|^2
-\na_a\big\{ e^{2\varphi}(\p_a\varphi+e^{-2\varphi}\langle h_a\psi,\psi\rangle)\big\}
\eea
and the condition for the normalization (\ref{normalization}) to be preserved can be rewritten as
\bea
\label{cond-normalization1}
c(t)=\dot\varphi- e^{-2\phi}\na_a\{ e^{2\varphi}(\p_a\varphi+e^{-2\varphi}\langle h_a\psi,\psi\rangle)\}
+e^{-2\phi}|\na^H\psi|^2.
\eea

\subsection{A necessary condition for the limit of $\varphi$}\hfil\break
The next consideration is that, at stationary points, we would like to have
\bea
0=\na_a^H\psi=\na_a\psi+(\lambda|H)_a\psi
\eea
and hence
\bea
\Re\langle \na_a\psi,\psi\rangle=-\Re\langle (\lambda|H)_a\psi\rangle
=-\langle h_a\psi,\psi\rangle.
\eea
In view of the equation (\ref{dvarphi}), we need the following condition for stationary points,
\bea
\label{stat-varphi}
\p_a\varphi+e^{-2\varphi}\langle h_a\psi,\psi\rangle=0.
\eea
\begin{rk}\label{rk: flatConnection}
We now discuss one possible interpretation of this constraint.  Multiplying the equation~\eqref{stat-varphi} by $e^{\phi}$ we get
\[
\nabla_a e^{\phi} + e^{-2\varphi}\langle h_a\psi,\psi\rangle e^{\phi}=0
\]
and hence $e^{\phi}$ defines a parallel section of the trivial $\mathbb{R}$ bundle over $M$ with respect to the connection $1$-form
\[
\mathfrak{A}_{a} :=  e^{-2\varphi}\langle h_a\psi,\psi\rangle.
\]
An important remark is that, since $e^{2\phi}= |\psi|^2$,  the connection $1$-form $\mathfrak{A}_{a}$ does not depend on $\phi$.  In fact, $\mathfrak{A}_a$ depends only on the section of $\mathbb{P}(S)$, the projectivized spinor bundle, induced by the non-vanishing spinor $\psi$.  Note that since $\mathfrak{A}_a$ admits a parallel section it is automatically a flat connection.
\end{rk}

\smallskip

\subsection{Admissible flows}
\hfil\break
The flow (\ref{flow}) satisfies the criteria identified in sections \S 2.1 and \S 2.2. Indeed, $\varphi$ is chosen there to flow by
\bea
\dot\varphi=
e^{-2\phi}(\nabla_{a}-2e^{-2\varphi}\langle h_a\psi,\psi\rangle)\big\{e^{2\phi}(\nabla_{a}\phi +e^{-2\phi}\langle h_{a}\psi, \psi \rangle)\big\}.
\eea
and the coefficient $c(t)$ in the flow of the spinor field $\psi$ is then chosen so that (\ref{cond-normalization1}) holds. Next, an integration by parts gives
\bea
\label{stat-varphi1}
\int_M \dot\varphi e^{4\varphi}\sqrt g dx
=-2\int_M\sum_ae^{2\varphi}|\p_ae^{\varphi}+e^{-2\varphi}\langle h_a\psi,\psi\rangle e^{\phi}|^2\sqrt g dx
\eea
Thus the desired equation (\ref{stat-varphi}) follows immediately from the vanishing of $\dot\varphi$. It follows that the normalization (\ref{normalization}) is preserved by the flow (\ref{flow}), as asserted in Theorem \ref{Main1}.

\smallskip
The condition (\ref{cond-normalization}) provides in effect a natural way of choosing the leading, second order terms in the flow of $\varphi$ so that the leading operators in the flows of $\varphi$ and $\psi$ are simultaneously elliptic. Indeed, if the flow of $\varphi$ is of the form
\bea
\dot\varphi=
e^{-2\phi}\sum_a \na_a(e^{2\varphi}\p_a\varphi)+\cdots
\eea
where $\cdots$ denote terms involving only first derivatives of all the fields in the flow, then the condition (\ref{cond-normalization}) insures that
\bea
c(t)=c(\na\varphi,\na H,\na\psi)
\eea
so that the leading, second-order, term in the flow of $\psi$ remains $(\na^H)^\dagger\na^H\psi$, and its ellipticity is not affected. Of course lower order terms will affect the stationary points, but the condition (\ref{stat-varphi}) still allows some freedom. For example, the arguments leading to (\ref{stat-varphi1}) would work with suitable lower order terms and pairing with corresponding powers of $e^{\varphi}$. Thus the flow (\ref{flow}) should be thought of as only a generic example of well-behaved flows with covariantly constant spinors as stationary points.

\subsection{Stationary points}\label{subsec: statPoints}
\hfil\break
We now show that the stationary points of the flow correspond exactly to triples $(g, H \psi)$ solving the equations
\[
\begin{aligned}
 \nabla^{H}\psi &=0\\
 dH=0, \quad &d^{\dagger} H =c\star H\wedge H
 \end{aligned}
 \]
 with $\psi \not\equiv 0$.  First, let us show that such triples are stationary points of the flow.  Of course, if $\psi \not\equiv 0$ and $\nabla^{H}\psi=0$, then $\psi$ is non-vanishing and we can write $|\psi|^2= e^{2\phi}$ for some smooth function $\phi$.  From the preceding considerations we see that~\eqref{stat-varphi} holds, and hence $\dot{\phi}=0$ along the flow~\eqref{flow}.  Clearly we also have $\dot{g}=0$ along~\eqref{flow} and hence $g$ is stationary.  Since $\nabla^H\psi=0$ and $\phi$ satisfies~\eqref{stat-varphi}, we see that the constant $c(t)$ defined by~\eqref{eq: flowConst} vanishes, and hence $\dot{\psi}=0$.  Finally, we show that $\dot{H}=0$.  It only suffices to note that the $k$-form $L(H)$ defined by~\eqref{eq: L(H)def} clearly vanishes.  Thus, solutions are stationary points.  
 
 Conversely, suppose $(g,\psi, H, \phi)$ is a stationary point of the flow~\eqref{flow}.  First, tracing $\dot{g}$ with respect to $g$ we get
 \[
0= {\rm Tr}_{g} \dot{g} = -\frac{1}{4}e^{-2\phi} \nabla_{a}\langle \gamma^{ap}\psi, \nabla_{p}^{H}\psi\rangle - \frac{n-2}{4}|\nabla^{H}\psi|^2.
\]
Note that, up to the factor of $e^{-2\phi}$ the first term is a divergence.  Thus, multiplying by $e^{2\phi}$ and integrating yields
\[
0= -\frac{n-2}{4}\int_{M}e^{2\phi}|\nabla^{H}\psi|^2 \sqrt{g}.
\]
It follows that $\nabla^{H}\psi=0$ provided $n\ne 2$; see the discussion in Remark~\ref{rk generalFlows} for alternative flows which have the desired critical points in dimension $2$.  Since $|\psi|^2=e^{2\phi}$, for some smooth function $\phi$ we see that $\psi \not\equiv 0$ and from~\eqref{stat-varphi1} we see that $\phi$ satisfies~\eqref{stat-varphi}.  We are reduced to considering the equation $\dot{H}=0$.  Taking the inner-product with $H$ and integrating by parts we get
\[
0 = \int_{M}|dH|^2 + \langle d^\dagger H- c\star H\wedge H , d^{\dagger}H \rangle - \langle L(H), H \rangle \sqrt{g}
\]
From the definition of $L(H)$ in~\eqref{eq: L(H)def} we have
\[
\langle L(H),H \rangle = \langle c\star H\wedge H, d^{\dagger}H-c\star H\wedge H\rangle
\]
and so we obtain
\[
0 = \int_{M}|dH|^2 +| d^\dagger H- c\star H\wedge H|^2  \sqrt{g}
\]
and the result follows. 

\begin{rk}\label{rk: generalFlows}
The flows~\eqref{flow},~\eqref{flow1} are just one choice of a large family of flows with the desired properties.  For example, one can consider a general flow for the metric given by
\[
\dot{g} = -\frac{1}{4}e^{-2\phi}{\rm Re}\left(Q_{\{p\ell\}}^H\right) + \frac{c_1}{2}\langle \nabla_{\{p}^{H}\psi, \nabla_{\ell \}}^{H}\psi\rangle - \frac{c_2}{4} g_{p\ell}|\nabla^{H}\psi|^2
\]
for $c_1, c_2$ constants.  This flow, coupled to the flows for $(H, \psi, \phi)$ in~\eqref{flow} (or just the flows for $(H,\psi)$ in case~\eqref{flow1} has all the desired properties as long as $\frac{c_1}{2}-\frac{nc_2}{4}\ne 0$.  In particular, it is rather easy to avoid the degeneracy for the critical points occurring when $n=2$ for the flows~\eqref{flow},~\eqref{flow1}. Our choice of the flows ~\eqref{flow},~\eqref{flow1} is made only so that when $H=0, \phi =0$ the flow reduces to the spinor flow of \cite{AWW}.    Nevertheless, the reader can convince themselves that the results of Theorem~\ref{Main1} and Theorem~\ref{Main2} continue to hold for these other flows.  Furthermore, we remark that if the coupling constant $c=0$ in the flow for $H$ is taken to be zero, then the dimensional condition $n=3k+1$ can be dropped.  In this case $H$ evolves the the heat flow.  On the other hand, it is not hard to write down interesting, intrinsic couplings between $H, dH, d^{\dagger}H$ which can be implemented by analogs of the flows~\eqref{flow},~\eqref{flow1}.  
\end{rk}

\section{The short-time existence of the flow}\label{sec: shortTime}

The goals of this section are twofold.  First, we explain the short time existence of the flows~\eqref{flow} and~\eqref{flow1}.  Secondly, in preparation for the Shi-type smoothing estimates proved in subsequent sections,  we rewrite the flow (\ref{flow}) as a perturbation of the reparametrized Ricci flow, coupled to parabolic flows for the spinor $\psi$, the flux $H$, and the normalization $\varphi$.   From these formulas we obtain useful equations for stationary points.

\subsection{Parabolicity and short-time existence} \hfil\break
We would like to show that the spinor flow~\eqref{flow} is weakly-parabolic and always admits solutions at least for a short-time. Our first task is to compute the linearization of the operator $Q^H_{p\ell}$.  The first step is to compute the variation of the flux connection with respect to variations in the metric.  This calculation requires the use of the Bourguignon-Gauduchon, for which we refer the reader to Appendix~\ref{sec: BGApp}
\begin{lem}\label{lem: varFluxConn}
Suppose that $g(s)= g+su$ is a family of Riemannian metrics for $s\in (-\epsilon, \epsilon)$.  Then, with respect to the Bourguignon-Gauduchon connection, and in a $g=g(0)$ orthonormal frame, we have
\[
\begin{aligned}
\frac{d}{ds}\big|_{s=0} \nabla^H_{p}\psi &= -\frac{1}{2}u_{pj}\nabla^H_j \psi -\frac{1}{4}\nabla_{a}u_{pb}\gamma^{ab}\psi +\frac{\lambda_2}{2}u_{pj}H_{a_1\ldots a_k}\gamma^{ja_1\cdots a_k}\psi\\
&\quad -\frac{\lambda_1}{2}\sum_{i=1}^{k-1}u_{a_im}H_{pa_1\ldots a_{i-1}ma_{i+1}\cdots a_{k-1}}\gamma^{a_1\cdots a_{k-1}}\psi\\
&\quad -\frac{\lambda_2}{2}\sum_{i=1}^{k}u_{a_im}H_{a_1\ldots a_{i-1} m a_{i+1}\cdots k}\gamma^{pa_1\cdots a_k}\psi
\end{aligned}
\]
\end{lem}
\begin{proof}
Let $\{e_a\}$ be a $g=g(0)$ orthonormal frame near a point $x\in M$, and let $\{e_a(s)\}$ be the parallel transport along the Bourguignon-Gauduchon connection.  We may as well assume, in addition, that $\nabla_{e_a}e_{b}(x)=0$.  We can write
\[
e_{a}(s) = w_a\,^{b}e_b(s)
\]
for $w(s) \in {\rm End}(TM)$ satisfying $w(0)= \delta_{ab}$ and  $\frac{d}{ds}\big|_{s=0}w_a\,^{b} = -\frac{1}{2}u_a\,^{b}$ where the index is raised using $g(0)$.  In this case, we have
\[
\nabla_{e_p(s)} \psi = \del_{e_{p}(s)}\psi - \frac{1}{4}\langle \nabla^s_{e_{p}(s)}e_{a}(s), e_{b}(s) \rangle_{g(s)} \gamma^{ab}\psi
\]
Differentiating in $s$ and evaluating at $s=0$ yields
\[
\begin{aligned}
\frac{d}{ds}\big|_{s=0}\nabla_{e_p(s)} \psi &= -\frac{1}{2}u_{pj}\left(\del_{e_j}\psi -\frac{1}{4}\langle \nabla_{e_{j}}e_{a}, e_{b} \rangle_{g} \gamma^{ab}\psi\right)\\
&+ \frac{1}{8}\del_{e_p}u_{ab} - \frac{1}{4}\left(\frac{d}{ds}\big|_{s=0} \Gamma_{pab}(s)\right)\gamma^{ab}\psi.
\end{aligned}
\]
From the standard formula for the variation of the Levi-Civita connection we have
\[
\frac{d}{ds}\big|_{s=0} \Gamma_{pab}(s)=\frac{1}{2}(\nabla_pu_{ab} + \nabla_{a}u_{pb} - \nabla_{b}u_{pa})
\]
Cancelling terms we arrive at
\[
\begin{aligned}
\frac{d}{ds}\big|_{s=0}\nabla_{e_p(s)} \psi &= -\frac{1}{2}u_{pj}\nabla_j \psi - \frac{1}{8} (\nabla_{a}u_{pb} - \nabla_{b}u_{pa})\gamma^{ab}\psi\\
&= -\frac{1}{2}u_{pj}\nabla_j \psi - \frac{1}{4}\nabla_{a}u_{pb}\gamma^{ab}\psi.
\end{aligned}
\]
To complete the calculation we only need to compute the derivative at $s=0$ of
\[
\lambda_1H(e_{p}(s), e_{a_1}(s),\ldots,e_{a_{k-1}}(s))\gamma^{a_1\cdots a_{k-1}} + \lambda_2H(e_{a_1}(s),\ldots,e_{a_k}(s))\gamma^{pa_1\cdots a_k}
\]
This is straightforward, using the properties of $e_a(s)$.  We get
\[
\begin{aligned}
\frac{d}{ds}\big|_{s=0}(\lambda|H)\psi &= -\frac{\lambda_1}{2}u_{pj}H_{ja_1,\ldots,a_{k-1}}\gamma^{a_1\cdots a_{k-1}}\gamma^{a_1\cdots a_{k-1}}\psi\\
&\quad  -\frac{\lambda_1}{2}\sum_{i=1}^{k-1}\sum_{m}u_{a_im}H_{pa_1\ldots a_{i-1}ma_{i+1}\cdots a_{k-1}}\gamma^{a_1\cdots a_{k-1}}\psi\\
&\quad - \frac{ \lambda_2}{2}\sum_{i=1}^{k}H_{a_1,\ldots,a_{i-1}, m, a_{i+1},\cdots, k}\gamma^{pa_1\cdots a_k}\psi.
\end{aligned}
\]
This easily implies the desired formula.
\end{proof}

Note that the evolution equations for $g, \psi$ involve only terms which are of order less than or equal to one in $H, \phi$.   So, from the point of view of determining the symbol of the evolution we may as well assume that $H, \phi$ are zero.  Then we have
\begin{lem}\label{lem: symbol}
Let $(u,\sigma) \in {\rm Sym}^2T^*M \oplus S$.  Then the principal symbol of the evolution of $g,\psi$ is given by the map
\[
\begin{pmatrix}u\\ \sigma\end{pmatrix} \mapsto \begin{pmatrix} e^{-2\phi}|\psi|^2\frac{1}{16}\left(|\xi|^2u_{p\ell} - \xi_b\xi_{\{\ell}u_{p\}b}\right)-\frac{1}{4}e^{-2\phi} \Re\left(\langle \gamma^{m\{\ell}\psi, \sigma \rangle\right)\xi_{p\}}\xi_m\\
\left(|\xi|^2\sigma - \frac{1}{4}u_{pb}\xi_a\xi_p\gamma^{ab}\psi\right)
\end{pmatrix}
\]
\end{lem}
\begin{proof}
From the calculation in Lemma~\ref{lem: varFluxConn} it follows easily that the contribution to the symbol from $-\frac{1}{4}{\rm Re}(Q^H_{\{p\ell\}})$ is given by
\[
(u,\sigma) \mapsto \frac{\xi_m\xi_a}{16}\Re\left(\langle \gamma^{m\{\ell}\psi,\gamma^{ab}\psi \rangle)u_{p\}b}\right) -\frac{1}{4}\xi_m\xi_{\{p}\Re\left(\langle \gamma^{m\ell\}}\psi, \sigma\rangle \right)
\]
Now we can write
\[
\begin{aligned}
\langle \gamma^{m\ell}\psi,\gamma^{ab}\psi \rangle\xi_m\xi_{a}u_{pb} &= \langle (\gamma^{m}\gamma^\ell -\delta_{m\ell})\psi,(\gamma^{a}\gamma^{b}- \delta_{ab})\psi \rangle\xi_m\xi_{a}u_{pb}\\
&= |\xi|^2\langle \gamma^{\ell}\psi,\gamma^{a}\psi \rangle u_{pa} + \sum_{m\ne a} \langle \gamma^{m}\gamma^\ell \psi,\gamma^{a}\gamma^{b}\psi \rangle\xi_m\xi_{a}u_{pb}\\
&\quad - \langle \gamma^{m}\gamma^{\ell}\psi, \psi \rangle \xi_m\xi_a u_{pa} - \langle \psi, \gamma^{a}\gamma^{b}\psi \rangle \xi_{\ell}\xi_{a}u_{pb} + |\xi|^2\xi_{\ell}\xi_{a}u_{pa}.
\end{aligned}
\]
The second term on the second line is purely imaginary, while the first term is purely imaginary if $\ell \ne a$, and similarly for the second and third terms on the third line.  All together we find 
\[
\Re\left(\langle \gamma^{m\ell}\psi,\gamma^{ab}\psi \rangle\right)\xi_m\xi_{a}u_{pb} = |\xi|^2|\psi|^2u_{p\ell} - |\psi|^2\xi_{\ell} \xi_{a} u_{p a}.
\]
This suffices to establish the first line of the desired formula.  For the second line, it suffices to compute the linearized evolution of $\psi$.  The linearization in the $\psi$ direction is trivial.  For the linearized evolution in the $g$ direction we can use Lemma~\ref{lem: varFluxConn}.  Let $\nabla^{H,s}$ denote the spin connection with flux defined by $g(s)$.  Then we have
\[
\frac{d}{ds}\big|_{s=0}-(\nabla^{H,s})^{\dagger_s}\nabla^{H_s} \psi = -\frac{1}{4}\sum_{p}\nabla_p\nabla_{a}u_{pb}\gamma^{ab}\psi,
\]
and the result follows.
\end{proof}

It follows from the work of Ammann, Weiss and Witt that symbol of the flow as computed in Lemma~\ref{lem: symbol} is non-negative definite, but degenerate (see \cite[Corollary 5.6]{AWW}), with degeneracy arising from the action of the connected component of the identity in the diffeomorphism group.   For our purposes is suffices to observe that, as in \cite{AWW}, the equation can be made strictly parabolic by a DeTurck type trick.  Fix a background metric $\bar{g}$, and define a vector field $X$ by
\begin{equation}\label{eq: deTurckVec}
X^k := 2g^{m k}g^{ij}\overline{\nabla}_{i}g_{jm}.
\end{equation}
where $\overline{\nabla}$ denote the Levi-Civita connection of $\bar{g}$.  Compute
\[
\mathcal{L}_{X}g_{p\ell} = \langle \nabla_{e_p}\tilde{X}, e_{\ell} \rangle + \langle \nabla_{e_{\ell}}\tilde{X}, e_p \rangle + O(g,\del g)
\]
where $O(g,\del g)$ denotes terms depending only $g$ up to order $1$.  If $g(s)=g+su$ is a family of metrics, then
\[
\frac{d}{ds}\big|_{s=0}\nabla^s_{e_p}X(g(s)) = g^{m k}g^{ij}\nabla_{p}\nabla_iu_{jm} + O(u, \del u).
\]
Thus, we have
\[
\frac{d}{ds}\big|_{s=0}\mathcal{L}_{X(g(s))}g(s)_{p\ell} =4g^{ij}\nabla_{\{p}\nabla_iu_{j\ell\}} + O(u, \del u)
\]
Net we consider the Lie derivative, as defined by Bourguignon-Gauduchon \cite{BG} given in an orthonormal frame by
\[
\mathcal{L}_{X}\psi = X^{m}\nabla_{m}\psi +\frac{1}{4}\left(\frac{\nabla_{r}X^p- \nabla_{p}X^{r}}{2}\right)\gamma^{rp}\psi
\]
If we consider the variation with respect to $g$, then a straightforward calculation using the Bourguignon-Gauduchon connection yields 
\[
\frac{d}{ds}\big|_{s=0}\mathcal{L}_{X(g(s))} \psi = \frac{1}{2}\nabla_{r}\nabla_iu_{ip}\gamma^{rp}\psi + O(\nabla \phi, \nabla u, \nabla \psi).
\]

\begin{lem}
Consider the gauge modified spinor flow with flux
\begin{equation}\label{eq: deTurckFlow}
\begin{aligned}
\dot{g} &= -\frac{1}{4}e^{-2\phi}{\rm Re}\left(Q_{\{p\ell\}}^H\right) + \mathcal{L}_{X}g + \frac{1}{2} \langle \nabla_{\{p}^{H}\psi, \nabla_{\ell \}}^{H}\psi\rangle - \frac{1}{4} g_{p\ell}|\nabla^{H}\psi|^2\\
\dot{\psi} &= -(\nabla_a^{H})^{\dagger}\nabla^{H}_{a} \psi  + c(t) \psi + \mathcal{L}_{X}\psi\\
\dot{H} &= -\left(d^{\dagger}dH + d(d^{\dagger}H-c\star(H\wedge H)) - L(H)\right) + \mathcal{L}_{X}H\\
\dot{\phi} &=e^{-2\phi}\sum_a (\nabla_{a}+e^{-2\varphi}\langle h_a\psi,\psi\rangle)\big\{e^{2\phi}(\nabla_{a}\phi +e^{-2\phi}\langle h_{a}\psi, \psi \rangle)\big\} + \nabla_{X}\phi
\end{aligned}
\end{equation}
where $c(t)$ is given by~\eqref{eq: flowConst}, $L(H)$ is defined by~\eqref{eq: L(H)def} and $X$ is defined by~\eqref{eq: deTurckVec}. This flow is strictly parabolic at any tuple $(g, \psi, H, \phi)$ satisfying the normalization condition $e^{-2\phi}|\psi|^2=1$.  In particular, the spinor flow with flux~\eqref{flow} admits a unique, short time solution.
\end{lem}
\begin{proof}
The preceding discussion shows that the gauge modified flow can be written in the form
\bea
\p_t\begin{pmatrix}g\\ \psi\end{pmatrix}
&=&F( D^2g,D^2\psi, \cdots)\\
\p_t H&=&-\Box H+ \mathcal{L}_{X}H + \cdots
\nonumber\\
\p_t\varphi&=&\Delta\varphi+\nabla_{X}\phi + \cdots
\eea
where the dependence of $F$ on the fields $H, \phi$ is lower order, and hence suppressed to ease notation.  Now the preceding calculation, together with Lemma~\ref{lem: symbol} shows that the symbol of $F$, along a variation $g(s)=g+su, \psi(s)=\psi+s\sigma$, is given by
\begin{equation}\label{eq: gaugedSymbol}
\begin{pmatrix}u\\ \sigma\end{pmatrix} \mapsto \begin{pmatrix} e^{-2\phi}|\psi|^2\frac{1}{16}\left(|\xi|^2u_{p\ell} - \xi_b\xi_{\{\ell}u_{p\}b}\right) +4\xi_b\xi_{\{\ell}u_{p\}b}-\frac{1}{4}e^{-2\phi} \Re\left(\langle \gamma^{m\{\ell}\psi, \sigma \rangle\right)\xi_{p\}}\xi_m\\
\left(|\xi|^2\sigma - \frac{1}{4}u_{pb}\xi_a\xi_p\gamma^{ab}\psi\right) +\frac{1}{2} u_{pb}\xi_a\xi_p\gamma^{ab}\psi
\end{pmatrix}
\end{equation}
Using that $e^{-2\phi}|\psi|^2=1$ and taking the real part of the inner product of $~\eqref{eq: gaugedSymbol}$ with $(e^{2\phi}u,\sigma)$ yields
\[
\begin{aligned}
&e^{2\phi}(|\xi|^2|u|^2 + \frac{63}{8}|u(\xi,\cdot)|^2) -\frac{1}{4}\Re\left(\langle \gamma^{ms}\psi, \sigma \rangle\right)\xi_m\xi_ru_{rs}\\
&+|\xi^2||\sigma|^2 + \frac{1}{4}u_{rs}\xi_m\xi_r\Re\left(\langle \gamma^{ms}\psi,\sigma\rangle\right)\\
&=e^{2\phi}\left(|\xi|^2|u|^2 + \frac{63}{8}|u(\xi,\cdot)|^2\right) +|\xi|^2|\sigma|^2.
\end{aligned}
\]
Thus, the symbol of the evolution of $g, \psi$ is non-degenerate.  For the remaining terms we note that the evolutions of $H$ involves a second order term in $g$, both through $-\Box H$ (thanks to the Bochner formula) and $\mathcal{L}_{X}H$ (due to the term $d\iota_{X}H$ appearing in the Lie derivative).  Thus, the symbol of the gauge modified flow~\eqref{eq: deTurckFlow} is given schematically by
\[
\text{Symbol}(\delta g,\delta \psi,\delta H,\delta \varphi)
=
\begin{pmatrix}
dF_{11}&dF_{12}&0&0\\
dF_{21}&dF_{22}&0&0\\
* &0&-\Box&0\\
0&0&0&\Delta\end{pmatrix}
\begin{pmatrix}\delta g\\ \delta\psi\\ \delta H\\ \delta \varphi\end{pmatrix}
\]
where $DF_{ij}$ denote the components of the linearization of $F$, given by ~\eqref{eq: gaugedSymbol}.  Since $dF, -\Box, \Delta$ are all elliptic, and the system is lower block triangular, the gauge modified flow~\eqref{eq: deTurckFlow} is strictly parabolic, and hence the short time existence follows from the standard theory.  Pulling the flow back by the flow of $-X$, yields a solution of the spinor flow~\eqref{flow}, and hence we obtain short time existence of the spinor flow with flux.  Finally, the uniqueness can be deduced in a standard way, following \cite{AWW}, which is itself modeled on the proof of uniqueness for the Ricci-flow by Hamilton \cite{HamSurv}.  We omit the proof. 

\end{proof}

\begin{rk}
The same considerations apply, essentially verbatim, to the flow~\eqref{flow1} with fixed $\phi$. In particular, the short time existence and uniqueness of the flow follows.
\end{rk}

\subsection{The form $Q_{\{p\ell\}}^H$ and the Ricci curvature}\hfil\break
In \cite{HeWang}, He and Wang have shown how to write the term $Q_{\{p\ell\}}^H$ when $H=0$ in terms of the Ricci curvature. The following is an extension of their derivation to the case of general $H$.

\smallskip
We shall compute in an orthonormal frame, and at a point $x\in M$ where $\nabla_{e_{a}}e_{b}=0$.  Then we have
\[
\begin{aligned}
\nabla_{a}\langle \gamma^{a\ell} \psi, \nabla_p^H\psi \rangle &=\nabla_{a}\langle \gamma^{a\ell}\psi, \nabla_{p}\psi \rangle + \nabla_{a} \langle \gamma^{a\ell}\psi, (\lambda|H)_{p}\psi \rangle\\
&= \langle \gamma^{a\ell} \nabla_{a}\psi, \nabla_{p}\psi\rangle + \langle \gamma^{a\ell}\psi, \nabla_{a}\nabla_{p}\psi \rangle +  \nabla_{a} \langle \gamma^{a\ell}\psi, (\lambda|H)_{p}\psi \rangle
\end{aligned}
\]
If we write $\slashed{D} = \gamma^a \nabla_a$ for the Dirac operator on spinors, then we can write
\[
\begin{aligned}
 \langle \gamma^{a\ell} \nabla_{a}\psi, \nabla_{p}\psi\rangle &= -\sum_{a\ne \ell}  \langle \gamma^{\ell}\gamma^{a} \nabla_{a}\psi, \nabla_{p}\psi\rangle\\
 &= -  \langle \gamma^{\ell}\slashed{D}\psi, \nabla_{p}\psi\rangle +\langle \nabla_{\ell} \psi, \nabla_{p}\psi \rangle\\
 &=   - \langle \slashed{D}\psi, \gamma^{\ell}\nabla_{p}\psi\rangle + \langle \nabla_{\ell} \psi, \nabla_{p}\psi \rangle.
 \end{aligned}
 \]
 Similarly
 \[
 \begin{aligned}
 \langle \gamma^{a\ell}\psi, \nabla_{a}\nabla_{p}\psi \rangle &= \sum_{a\ne \ell} \langle \psi, \gamma^{\ell}\gamma^{a}\nabla_{a}\nabla_{p}\psi \rangle\\
 &=  \langle \psi, \gamma^{\ell}\slashed{D}\nabla_{p}\psi \rangle -  \langle \psi, \nabla_{\ell}\nabla_{p}\psi \rangle.
 \end{aligned}
 \]
 Thus, altogether we have
 \[
 \begin{aligned}
 Q_{p\ell}^H &= -\langle \slashed{D}\psi, \gamma^{\ell}\nabla_{p}\psi\rangle +\langle \nabla_{\ell} \psi, \nabla_{p}\psi \rangle\\
 &\quad+\langle \psi, \gamma^{\ell}\slashed{D}\nabla_{p}\psi \rangle -  \langle \psi, \nabla_{\ell}\nabla_{p}\psi \rangle +\nabla_{a} \langle \gamma^{a\ell}\psi, (\lambda|H)_{p}\psi \rangle.
\end{aligned}
\]
The Bochner formula for the Levi-Civita connection (see Lemma~\ref{lem: BochnerApp}) gives
\begin{equation}\label{eq: Bochner}
\slashed{D}\nabla_{p}\psi = \nabla_{p}\slashed{D}\psi + \frac{1}{2}R_{pk}\gamma^k\psi,
\end{equation}
and so we obtain
\[
 \begin{aligned}
 Q_{p\ell}^H &= -\langle \slashed{D}\psi, \gamma^{\ell}\nabla_{p}\psi\rangle +\langle \nabla_{\ell} \psi, \nabla_{p}\psi \rangle\\
 &\quad+\langle \psi, \gamma^{\ell}\left(\nabla_{p}\slashed{D}\psi + \frac{1}{2}R_{pk}\gamma^k\psi\right) \rangle -  \langle \psi, \nabla_{\ell}\nabla_{p}\psi \rangle +\nabla_{a} \langle \gamma^{a\ell}\psi, (\lambda|H)_{p}\psi \rangle.\\
\end{aligned}
\]
Now, since $\langle \psi, \gamma^{i}\gamma^{j}\psi \rangle \in \sqrt{-1}\bR$ when $i\ne j$ and since $|\psi|^2=e^{2\phi}$ we get
\[
\begin{aligned}
{\rm Re}(Q^H_{p\ell}) &= e^{2\phi}\frac{1}{2}R_{p\ell} + {\rm Re}\left(\langle \psi, \gamma^{\ell}\nabla_{p}\slashed{D}\psi \rangle \right)\\
&\quad- {\rm Re}\left(\langle \slashed{D}\psi, \gamma^{\ell}\nabla_{p}\psi\rangle\right) +\nabla_{a}{\rm Re}\left( \langle \gamma^{a\ell}\psi, (\lambda|H)_{p}\psi \rangle\right)\\
&\quad- {\rm Re}\left( \langle \psi, \nabla_{\ell}\nabla_{p}\psi \rangle +\langle \nabla_{\ell} \psi, \nabla_{p}\psi \rangle\right)
\end{aligned}
\]
Now from the relation $e^{-2\phi}|\psi|^2=1$ we get
\[
\begin{aligned}
\nabla_pe^{2\phi} &= 2{\rm Re}\left(\langle \nabla_p \psi, \psi \rangle \right)\\
{\rm Re}\left(\langle \nabla_{\ell}\nabla_p \psi, \psi \rangle\right) &= \frac{1}{2}{\rm Hess}(e^{2\phi})_{\ell p} - {\rm Re}(\langle \nabla_p\psi, \nabla_{\ell}\psi \rangle)
\end{aligned}
\]
where ${\rm Hess}(f)_{\ell p}$ denotes the Riemannian hessian of $f$. Thus, all together we have
\[
\begin{aligned}
{\rm Re}(Q^H_{p\ell}) &= e^{2\phi}\frac{1}{2}R_{p\ell} + {\rm Re}\left(\langle \psi, \gamma^{\ell}\nabla_{p}\slashed{D}\psi \rangle \right) -\frac{1}{2}{\rm Hess}(e^{2\phi})_{\ell p}\\
&\quad- {\rm Re}\left(\langle \slashed{D}\psi, \gamma^{\ell}\nabla_{p}\psi\rangle\right) +\nabla_{a}{\rm Re}\left( \langle \gamma^{a\ell}\psi, (\lambda|H)_{p}\psi \rangle\right)\\
&\quad +2{\rm Re}\left(\langle \nabla_{\ell} \psi, \nabla_{p}\psi \rangle\right)
\end{aligned}
\]
which is the desired formula.

\subsection{The gauge-modification to a modified Ricci flow}\hfil\break
Following He and Wang \cite{HeWang}  we can introduce another gauge modified flow by the following observation.  The term ${\rm Re}\left(\langle \psi, \gamma^{\ell}\nabla_{p}\slashed{D}\psi \rangle \right)$ in the preceding expression for $Q^H$ is of second order in the spinor field $\psi$.  It is convenient to eliminate it by a reparametrization. Thus, define
consider the vector field
\begin{equation}\label{eq: HWVF}
X=  e^{-2\phi}{\rm Re}\left(\langle \psi, \gamma^{\ell}\slashed{D}\psi \rangle \right) e_{\ell} 
\end{equation}
The key observation is that
\[
\begin{aligned}
\mathcal{L}_{X} g_{p\ell} &= e^{-2\phi}\nabla_{p}{\rm Re}\left(\langle \psi, \gamma^{\ell}\slashed{D}\psi\rangle \right) +e^{-2\phi}\nabla_{\ell}{\rm Re}\left(\langle \psi, \gamma^{p}\slashed{D}\psi\rangle \right)\\
&\quad -2e^{-2\phi}{\rm Re}\left(\langle \psi, \gamma^{\ell}\slashed{D}\psi\rangle \right)\nabla_p\phi -2e^{-2\phi}{\rm Re}\left(\langle \psi, \gamma^{p}\slashed{D}\psi\rangle \right)\nabla_{\ell}\phi\\
&= 2e^{-2\phi}{\rm Re}\left(\langle \nabla_{\{p}\psi, \gamma^{\ell\}}\slashed{D}\psi\rangle \right) +2e^{-2\phi}{\rm Re}\left(\langle\psi, \gamma^{\{\ell} \nabla_{p\}}\slashed{D}\psi\rangle \right)-4X^{\{\ell}\nabla_{p\}}\phi\\
&=2e^{-2\phi}{\rm Re}\left(\langle\gamma^{\{\ell} \nabla_{p\}}\psi, \slashed{D}\psi\rangle \right) +2e^{-2\phi}{\rm Re}\left(\langle\psi, \gamma^{\{\ell} \nabla_{p\}}\slashed{D}\psi\rangle \right)-4X^{\{\ell} \nabla_{p\}}\phi
\end{aligned}
\]

and so we can write
\[
\begin{aligned}
-\frac{1}{4}e^{-2\phi}{\rm Re}(Q^{H}_{\{p\ell\}}) &= -\frac{1}{8}R_{p\ell} -\frac{1}{8}\mathcal{L}_{X}g_{p\ell}+ \frac{1}{2}e^{-2\phi}{\rm Re}\left(\langle \slashed{D}\psi, \gamma^{\{\ell}\nabla_{p\}}\psi\rangle\right)\\
&\quad +\frac{1}{8}e^{-2\phi}{\rm Hess}(e^{2\phi})_{\ell p}-  \frac{1}{4}e^{-2\phi}\nabla_{a}{\rm Re}\left( \langle \gamma^{a\{\ell}\psi, (\lambda|H)_{p\}}\psi \rangle\right)\\
&\quad -\frac{1}{2}e^{-2\phi}{\rm Re}\left(\langle \nabla_{\ell} \psi, \nabla_{p}\psi \rangle\right) -\frac{1}{2}e^{-2\phi}{\rm Re}\left(\langle \psi, \gamma^{\{\ell}\slashed{D}\psi \rangle \right)\phi_{p\}}
\end{aligned}
\]
If we reparameterize the flow by the time dependent family of diffeomorphisms $F(t)$ solving $F(0)= Id$ and $F'(t)=\frac{1}{8}X$ then we obtain the following flow for $g$;
\[
\begin{aligned}
\dot{g}_{p\ell} &=  -\frac{1}{8}R_{p\ell} + \frac{1}{2}e^{-2\phi}{\rm Re}\left(\langle \slashed{D}\psi, \gamma^{\{\ell}\nabla_{p\}}\psi\rangle\right) +\frac{1}{4}{\rm Hess}(\phi)_{p\ell} + \frac{1}{2}\phi_p\phi_{\ell}\\
&\quad-  \frac{1}{4}e^{-2\phi}\nabla_{a}{\rm Re}\left( \langle \gamma^{a\{\ell}\psi, (\lambda|H)_{p\}}\psi \rangle\right) -\frac{1}{2}e^{-2\phi}{\rm Re}\left(\langle \nabla_{\ell} \psi, \nabla_{p}\psi \rangle\right)\\
&\quad-\frac{1}{2}e^{-2\phi}{\rm Re}\left(\langle \psi, \gamma^{\{\ell}\slashed{D}\psi \rangle \right)\phi_{p\}} + \frac{1}{2}\langle \nabla_{\{p}^{H}\psi, \nabla_{\ell \}}^{H}\psi\rangle - \frac{1}{4} g_{p\ell}|\nabla^{H}\psi|^2.
\end{aligned}
\]
We therefore obtain the gauge modified flow
\begin{equation}\label{eq: modFlowVar}
\begin{aligned}
\dot{g} &= -\frac{1}{8}R_{p\ell} + \frac{1}{2}e^{-2\phi}{\rm Re}\left(\langle \slashed{D}\psi, \gamma^{\{\ell}\nabla_{p\}}\psi\rangle\right) +\frac{1}{4}{\rm Hess}(\phi)_{p\ell} + \frac{1}{2}\phi_p\phi_{\ell}\\
&\quad-  \frac{1}{4}e^{-2\phi}\nabla_{a}{\rm Re}\left( \langle \gamma^{a\{\ell}\psi, (\lambda|H)_{p\}}\psi \rangle\right) -\frac{1}{2}e^{-2\phi}{\rm Re}\left(\langle \nabla_{\ell} \psi, \nabla_{p}\psi \rangle\right)\\
&\quad-\frac{1}{2}e^{-2\phi}{\rm Re}\left(\langle \psi, \gamma^{\{\ell}\slashed{D}\psi \rangle \right)\phi_{p\}} + \frac{1}{2}\langle \nabla_{\{p}^{H}\psi, \nabla_{\ell \}}^{H}\psi\rangle - \frac{1}{4} g_{p\ell}|\nabla^{H}\psi|^2\\
\dot{\psi} &= -(\nabla^{H})^{\dagger}\nabla^{H}_{a} \psi +\frac{1}{8} \mathcal{L}_{X}\psi + c(t) \psi\\
\dot{H} &= -\left(d^{\dagger}dH + d(d^{\dagger}H-c\star(H\wedge H)) - L(H)\right)+\frac{1}{8}\mathcal{L}_{X}H\\
\dot{\phi}  &=e^{-2\phi}(\nabla_{a}-2e^{-2\phi}\langle h_a\psi, \psi \rangle )\{e^{2\phi}\left(\nabla_{a}\phi +e^{-2\phi}\langle h_a\psi, \psi \rangle\right)\}+\frac{1}{8}\nabla_{X}\phi\\
\end{aligned}
\end{equation}
where $c(t)$ is given by~\eqref{eq: flowConst} and $L(H)$ is defined by~\eqref{eq: L(H)def}

\subsection{The gauge-modified flow with fixed $\phi$}
\hfil\break
The same considerations give immediately the following expression for the gauge-modified corresponding to the flow (\ref{flow1}) with fixed $\varphi$, which we record here for the reader's convenience,
\begin{equation}\label{eq: modFlowFix}
\begin{aligned}
\dot{g} &=  -\frac{1}{8}R_{p\ell} + \frac{1}{2}e^{-2\phi}{\rm Re}\left(\langle \slashed{D}\psi, \gamma^{\{\ell}\nabla_{p\}}\psi\rangle\right) +\frac{1}{4}{\rm Hess}(\phi)_{p\ell} + \frac{1}{2}\phi_p\phi_{\ell}\\
&\quad-  \frac{1}{4}e^{-2\phi}\nabla_{a}{\rm Re}\left( \langle \gamma^{a\{\ell}\psi, (\lambda|H)_{p\}}\psi \rangle\right) -\frac{1}{2}e^{-2\phi}{\rm Re}\left(\langle \nabla_{\ell} \psi, \nabla_{p}\psi \rangle\right)\\
&\quad-\frac{1}{2}e^{-2\phi}{\rm Re}\left(\langle \psi, \gamma^{\{\ell}\slashed{D}\psi \rangle \right)\phi_{p\}} + \frac{1}{2}\langle \nabla_{\{p}^{H}\psi, \nabla_{\ell \}}^{H}\psi\rangle - \frac{1}{4} g_{p\ell}|\nabla^{H}\psi|^2\\
\dot{\psi} &= -(\nabla^{H})^{\dagger}\nabla^{H}_{a} \psi +\frac{1}{8} \mathcal{L}_{X}\psi + c(t) \psi\\
\dot{H} &= -\left(d^{\dagger}dH + d(d^{\dagger}H-c\star(H\wedge H)) - L(H)\right)+\frac{1}{8}\mathcal{L}_{X}H\\
\end{aligned}
\end{equation}
where $c(t)$ is given by ~\eqref{eq: flow1Const} and $L(H)$ is defined by~\eqref{eq: L(H)def}


\section{Preliminary estimates for $\varphi$ and the volume form}\label{sec: prelimEst}

Henceforth we work with the gauge-modified (\ref{eq: modFlowVar}) version of the spinor flow with flux. We begin by deriving some estimates for $\varphi$ and the volume form $\sqrt g$.

\begin{lem}\label{lem: supPhiBnd}
Assume that, along the spinor flow~\eqref{flow} with dynamical $\phi$, we have bounds $|\nabla \psi|^2 + |H|^2+|\nabla H| \leq A$.  Then there is a uniform upper bound
\[
e^{2\phi(t)} \leq -\frac{1}{2}+ (e^{2\phi(0)}+\frac{1}{2})e^{4At}
\]
\end{lem}
\begin{proof}
From the evolution of $\phi$ we have 
\[
 \frac{1}{2}\frac{d}{dt} e^{2\phi} = \frac{1}{2}\Delta e^{2\phi} + \nabla_{a}\langle h_a\psi, \psi \rangle - 2\langle h_a \psi, \psi \rangle \nabla_{a}\phi - 2e^{-2\phi}\sum_a|\langle h_a\psi, \psi \rangle |^2.
 \]
 Then, at the maximum of $\phi$ we have
 \[
 \frac{1}{2}\frac{d}{dt} e^{2\phi} \leq  \nabla_{a}\langle h_a\psi, \psi \rangle.
 \]
From the bounds for $\nabla H, \nabla \psi$ we have
\[
\big| \nabla_{a}\langle h_a\psi, \psi \rangle\big| \leq Ae^{2\phi} +Ae^{\phi} \leq 2Ae^{2\phi} + A.
\]
Now the result follows from ODE comparison.
\end{proof}
\begin{rk}
We remark that an upper bound for $\phi$ also holds along the flow provided the connection $1$-form $\mathfrak{A}_a= e^{-2\phi}\langle h_a\psi, \psi \rangle$ satisfies the bound $|d^{\dagger}\mathfrak{A}|<A$.  To see this we observe that
\[
 \nabla_{a}\langle h_a\psi, \psi \rangle - 2\langle h_a \psi, \psi \rangle \nabla_{a}\phi = e^{2\phi}\nabla_{a}\mathfrak{A}_a.
 \]
However, the bound $|d^{\dagger}\mathfrak{A}|<A$ is not equivalent to the bounds \\$|\nabla \psi|^2 + |H|^2+|\nabla H| \leq A'$.  As we have seen above, the latter bounds imply an estimate
\[
\big| \nabla_{a}\langle h_a\psi, \psi \rangle\big| \leq A'e^{2\phi} +A'e^{\phi}.
\]
On the other hand, the bound $|d^{\dagger}\mathfrak{A}|<A$ implies that, at the maximum or minimum of $\phi$ we have the improved bound
\[
\big| \nabla_{a}\langle h_a\psi, \psi \rangle\big| \leq Ae^{2\phi}.
\]
While the distinction is essentially irrelevant for for upper bounds of $\phi$, it is essential for establishing lower bounds of $\phi$.
\end{rk}

\begin{lem}\label{lem: evoOfVolForm}
Assume that either of the spinor flows~\eqref{flow} or~\eqref{flow1} exists on $[0,\tau)$
\begin{itemize}
\item[$(i)$] Along the flow~\eqref{flow1} with fixed $\phi$, we have
\[
\int_{M}e^{2\phi} \sqrt{g(t)}\leq C\int_{M}e^{2\phi} \sqrt{g(0)}
\]
\item[$(ii)$] Suppose that along the flow~\eqref{flow} with dynamical $\phi$ the connection $1$-form $\mathfrak{A}_a= e^{-2\phi}\langle h_a\psi, \psi \rangle$ satisfies the bound $|d^{\dagger}\mathfrak{A}|_{g(t)}<A$ uniformly on $[0,\tau)$.  Then 
\[
{\int_{M}e^{2\phi(t)} \sqrt{g(t)}\leq e^{At}\int_{M}e^{2\phi(0)} \sqrt{g(0)}}.
\]
\end{itemize} 
\end{lem}
\begin{proof}
We will provide the proof in the case of dynamical $\phi$, that is, for the flow~\eqref{flow}, the case of \eqref{flow1} being easier.  Along~\eqref{flow} we have
\[
\frac{d}{dt} \int_{M}e^{2\varphi(t)} \sqrt{g}(t) = \frac{1}{2} \int_{M}e^{2\varphi(t)} {\rm Tr}_{g}\dot{g} \sqrt{g}(t) + \int_{M}\left(\frac{d}{dt}e^{2\phi}\right)\sqrt{g}.
\]
As discussed in Section~\ref{subsec: statPoints}, along~\eqref{flow} we have
\[
{\rm Tr}_{g} \dot{g} = -\frac{1}{4}e^{-2\phi} \nabla_{a}\langle \gamma^{ap}\psi, \nabla_{p}^{H}\psi\rangle - \frac{n-2}{4}|\nabla^{H}\psi|^2.
\]
and so, by the divergence theorem
\[
 \frac{1}{2} \int_{M}e^{2\varphi(t)} {\rm Tr}_{g}\dot{g} \sqrt{g}(t)= - \frac{n-2}{8}\int_{M}e^{2\phi}|\nabla^{H}\psi|^2 \sqrt{g} \leq 0
 \]
 This suffices to establish $(i)$.  For $(ii)$, from the flow for $\phi$ along~\eqref{flow} we have
 \[
 \frac{1}{2}\frac{d}{dt} e^{2\phi} = \frac{1}{2}\Delta e^{2\phi} + \nabla_{a}\langle h_a\psi, \psi \rangle - 2\langle h_a \psi, \psi \rangle \nabla_{a}\phi - 2e^{-2\phi}\sum_a|\langle h_a\psi, \psi \rangle |^2.
 \]
 We can rewrite the middle two terms as
 \[
 \nabla_{a}\langle h_a\psi, \psi \rangle - 2\langle h_a \psi, \psi \rangle \nabla_{a}\phi= e^{2\phi}\nabla_{a}\mathfrak{A}_{a}
 \]
 Thus, from the bounds for $\mathfrak{A}$ we have
 \[
 \int_{M}\left(\frac{d}{dt}e^{2\phi}\right)\sqrt{g}\leq A\int_{M}e^{2\phi(t)} \sqrt{g(t)}
 \]
 from which the result follows.
 \end{proof}

 Finally, we have
 
 \begin{prop}\label{prop: infPhiBnd}
Suppose that the flow~\eqref{eq: modFlowVar} exists on $[0,\tau]$ and there is a constant $A>0$ so that the connection $1$-form $\mathfrak{A}= e^{-2\phi}\langle h_a\psi, \psi \rangle$ satisfies the bound $|\mathfrak{A}|_{g(t)}^2+ |d^{\dagger}\mathfrak{A}|_{g(t)} \leq A$ uniformly on $[0,\tau]$. Then there is a constant $C$ depending only on $A, \tau$ and the initial data such that
\[
\phi(t) \geq \phi(0)-3At
\]
\end{prop}
\begin{proof}
The proof is another application of the maximum principle.  From the evolution of $\phi$ we have
\[
 \frac{1}{2}\frac{d}{dt} e^{2\phi} = \frac{1}{2}\Delta e^{2\phi} + e^{2\phi}\nabla_{a}\mathfrak{A} - 2e^{2\phi}|\mathfrak{A}|^2.
 \]
 From the bounds for $\mathfrak{A}$, at a minimum of $\phi$ we have
 \[
  \frac{d}{dt} e^{2\phi} \geq -6Ae^{2\phi}
  \]
  and the result follows from ODE comparison.
  \end{proof}
  
  An immediate corollary of Lemma~\ref{lem: evoOfVolForm} and Proposition~\ref{prop: infPhiBnd} is
  
  \begin{cor}
  Suppose that along the flow~\eqref{flow} with dynamical $\phi$ we have $|\mathfrak{A}|^2+|d^{\dagger}\mathfrak{A}|_{g(t)} \leq A$.  Then there is a bound
  \[
  {\rm Vol}(M,g(t)) \leq e^{2{\rm osc}\phi(0)+7At }{\rm Vol}(M, g(0)).
  \]
  \end{cor}
  Alternatively, one can obtain a volume upper bound along the flow directly by assuming a lower bound for $\phi$, in addition to bounds for $\nabla \psi, H$.
  
  \begin{lem}\label{lem: easyVolBnd}
  Suppose that along the flow~\eqref{flow} with dynamical $\phi$ there are bounds $|\nabla \psi|^2 + |H|^2 \leq A$ and $e^{-2\phi} \leq C$.  Then we have
  \[
  {\rm Vol}(M, g(t)) \leq e^{(2C+1)At}{\rm Vol}(M, g(0))
  \]
  \end{lem}
  \begin{proof}
  From the evolution of the volume form we have
  \[
  \begin{aligned}
  \frac{d}{dt}   {\rm Vol}(M, g(t))&=\int_{M} -\frac{1}{4}e^{-2\phi} \nabla_{a}\langle \gamma^{ap}\psi, \nabla_{p}^{H}\psi\rangle - \frac{n-2}{4}|\nabla^{H}\psi|^2 \sqrt{g(t)}\\
  &\leq \int_{M} \frac{1}{4}\nabla_{a}e^{-2\phi} \langle \gamma^{ap}\psi, \nabla_{p}^{H}\psi\rangle \sqrt{g(t)}.
  \end{aligned}
  \]
  From $e^{2\phi} = |\psi|^2$ we get $ |\nabla e^{-2\phi}| \leq |\nabla \psi|e^{-3\phi}$ and so
  \[
  \begin{aligned}
  \big|\nabla_{a}e^{-2\phi} \langle \gamma^{ap}\psi, \nabla_{p}^{H}\psi\rangle\big| &\leq |\nabla\psi|e^{-3\phi} \left(e^{\phi}|\nabla \psi| + |H|e^{2\phi}\right)\\
  &\leq 2AC+A.
  \end{aligned}
  \]
  The result follows from ODE comparison.
  \end{proof}

\section{The Shi-type smoothing estimates}\label{sec: Shi}

We give now the proof of the Shi-type smoothing estimates. First, we discuss the scaling laws for the various fields.

\begin{lem}
Suppose that $(g(t), \psi(t), H(t), \phi(t))$ is a solution of either of the flows ~\eqref{flow} or~\eqref{flow1}, or their gauge modified versions~\eqref{eq: modFlowVar} or~\eqref{eq: modFlowFix}, where $H$ is a $k$-form, and $\dim M = n= 3k-1$.  Fix $\sigma>0$.  Then 
\[
(g_{\sigma}(t), \psi_{\sigma}(t), H_{\sigma}(t), \phi_{\sigma}(t)) = (\sigma g(\sigma^{-1}t), \psi(\sigma^{-1}t), \sigma^{\frac{k-1}{2}}H(\sigma^{-1}t), \phi(\sigma^{-1}t))
\]
is also a solution of the flow.
\end{lem}
\begin{proof}
The proof is straightforward.  If $\{e_{\ell}\}$ is a $g$-orthonormal frame, then $\{\sigma^{-\frac{1}{2}}e_{\ell}\}$ is a $\sigma g$ orthonormal frame.  Furthermore, $\sigma^{-\frac{1}{2}}e_{\ell}$ is parallel along the path $\sigma \mapsto \sigma g$ with respect to the Bourguignon-Gauduchon connection.  Thus, we have
\[
\begin{aligned}
(H_{\sigma}|\lambda) \psi_{\sigma} &= \sigma^{k-1}{2}H(\sigma^{-1/2}e_{i_1}, \ldots,\sigma^{-1/2}e_{i_k}) \gamma^{i_1\cdots i_k}\psi\\
&= \sigma^{-1/2}H_{i_1 \cdots i_k}\gamma^{i_1\cdots i_k}\psi.
\end{aligned}
\]
On the other hand, if we denote by $\nabla^{\sigma}$ the Levi-Civita connection with respect to $g_{\sigma}$, then since $\nabla^{\sigma}_{\sigma^{-1/2}e_{\ell}}\psi_{\sigma} = \sigma^{-1/2}\nabla^g_{e_{\ell}}\psi$ we see that
\[
\nabla^\sigma\psi_{\sigma} + (\lambda|H_{\sigma})\psi_{\sigma} = \sigma^{-1/2}\nabla^{H}\psi.
\]
From this expression it follows easily that $g_{\sigma}(t)$ is a solution of the flow.  It is easily checked that $\psi_{\sigma}, \phi_{\sigma}$ also satisfy the flow equations.  We only need to examine the flow of $H$.  First note that if $\beta$ is a $p$-form, then
\[
\star_{\sigma g} \beta = \sigma^{\frac{n}{2}-p} \star_g \beta
\]
and so
\[
\star_{\sigma g} d\star_{\sigma g}  \beta = \lambda^{-1}\star_gd\star_g\beta.
\]
Now, since $H$ is a $k$-form and $n=3k-1$ we have
\[
\star_{g_{\sigma}}(H\wedge H) = \sigma^{\frac{n}{2}-2k} \star_g (H\wedge H) = \sigma^{-\frac{(1+k)}{2}}\star_{g}(H\wedge H).
\]
Now, including the scaling of $H$ yields
\[
\star_{g_{\sigma}}(H_{\sigma}\wedge H_{\sigma}) =  \sigma^{-1 +\frac{(k-1)}{2}}\star_g(H\wedge H)
\]
From this and the definition of $L(H)$ it follows easily that $H_{\sigma}(t)$ satisfies the flow as well.
\end{proof}

The main purpose of this result is to determine the appropriate scale invariant quantities to consider.  For example, we have
\begin{enumerate}
\item[$(i)$] The quantities $|\psi|, |\phi|$ are homogeneous of degree zero under rescaling.
\item[$(ii)$] The quantities $|\nabla \psi|^2, |H|^2, |\nabla \phi|^2$ are homogeneous of degree $1$ under rescaling.
\item[$(iii)$] For $p\geq 0$, the quantities $|\nabla^{p+2}\psi|^2, |\nabla^{p+1}H|^2, |\nabla^{p+2}\phi|^2, |\nabla^p Rm|^2$ are homogeneous of degree $p+2$ under rescaling.
\end{enumerate}

The first step towards establishing the smoothing estimates is to compute the evolution equations of the quantities $\nabla^pRm,\nabla^pH, \nabla^p\phi, \nabla^p\psi$ for all $p\geq 0$.  In order to streamline the exposition, we will introduce the following notation.  We will denote by $*$ any tensorial contraction involving the metric, multiplication by fixed constants, Clifford multiplication by unit vectors, and the metric on the spinor bundle. For $p\geq 0$, and any field $A$ we will denote
\[
A^p = \overbrace{ A*A*\cdots A}^{\text{$p$-times}}.
\]
In this notation, along the gauge modified flows~\eqref{eq: modFlowVar} or~\eqref{eq: modFlowFix} we have
\[
\begin{aligned}
\dot{g} &= -\frac{1}{8}{\rm Ric} +\frac{1}{8} e^{-2\phi}{\rm Hess}(e^{2\phi})\\
&+\quad e^{-2\phi}\left(\nabla \psi^2 + \psi^2*\nabla H+\psi*\nabla \psi *H + \psi*\nabla \psi *\nabla \phi \right)\\
&\quad + \nabla \psi^2 + H*\psi*\nabla \psi + H^2*\psi^2
\end{aligned}
\]
In order to lighten the notation, we shall introduce the error terms
\begin{equation}\label{eq: errorDefn}
\begin{aligned}
\cE &:= \left(\nabla \psi^2 + \psi^2*\nabla H+\psi*\nabla \psi *H \right)\\
\cE'&:= \left( \nabla \psi^2 + H*\psi*\nabla \psi + H^2*\psi^2+\psi*\nabla \psi *\nabla e^{-2\phi}\right)
\end{aligned}
\end{equation}
so that the flow can be written succinctly as
\begin{equation}\label{eq: flowgError}
\dot{g} = -\frac{1}{8}{\rm Ric} +\frac{1}{8} e^{-2\phi}{\rm Hess}(e^{2\phi}) + e^{-2\phi}\cE + \cE'.
\end{equation}
It's important to stress neither of the terms $\cE, \cE'$ cannot be treated as lower order error terms.

\medskip

\subsection{The flow of $\na^pRm$}\hfil\break
We begin by computing the variation of the Riemann tensor under the flow.  Recall the following formula, which can be found, for example, in \cite{ChowEtAl}

\begin{lem}\label{lem: varCurv}
Using the convention in Appendix~\ref{sec: conv} , if $g(t)$ is a family of Riemannian metrics with $\dot{g}(0) = h$, then
\[
\begin{aligned}
\frac{d}{dt}\big|_{t=0}Rm(g(t))_{jim k} &= \frac{1}{2}\left(\nabla_{i}\nabla_{k}h_{jm} + \nabla_{j}\nabla_{m}h_{ik} - \nabla_{i}\nabla_{m}h_{jk} - \nabla_{j}\nabla_{k}h_{im}\right)\\
&\quad +\frac{1}{2}\left( R_{ji}\,^{q}\,_{k}h_{m q} - R_{ji}\,^{q}\,_{m} h_{kq}\right)\\
\frac{d}{dt}\big|_{t=0} R_{jk} &=-\frac{1}{2}(\Delta h_{jk} + \nabla_j\nabla_k {\rm Tr}_{g}h) + \frac{1}{2}g^{im}(\nabla_i\nabla_kh_{jm} + \nabla_j\nabla_m h_{ik})\\
&\quad +\frac{1}{2}R_{j}\,^{q}h_{kq} - \frac{1}{2}R_{j}\,^{mq}\,_{k}h_{mq}
\end{aligned}
\]
\end{lem}
 The other standard result we need is the formula for the Laplacian of the Riemann tensor, see for example \cite[Lemma 6.13]{ChowEtAl}
 \begin{lem}
 The Laplacian of the Riemann tensor is given by
 \[
 \begin{aligned}
 \Delta Rm_{jimk} :&= g^{pq}\nabla_{p}\nabla_q R_{ji m k}\\
 & = \nabla_{i}\nabla_{m}R_{jk} - \nabla_{i}\nabla_{k}R_{jm} +\nabla_{j}\nabla_{k}R_{im} - \nabla_{j}\nabla_{m}R_{ik} + Rm*Rm
 \end{aligned}
 \]

 \end{lem}

We now apply this to the spinor flow with flux.  We have
\begin{equation}\label{eq: dtRm}
\begin{aligned}
\frac{d}{dt} Rm_{ij\ell k} &=  \frac{1}{2}\left(\nabla_{i}\nabla_{k}\dot{g}_{j\ell} + \nabla_{j}\nabla_{\ell}\dot{g}_{ik} - \nabla_{i}\nabla_{\ell}\dot{g}_{jk} - \nabla_{j}\nabla_{k}\dot{g}_{i\ell} + R_{ij}\,^{q}\,_{k}\dot{g}_{\ell q} - R_{ij}\,^{q}\,_{\ell} \dot{g}_{kp}\right)
\end{aligned}
\end{equation}
Note that for any function $f$, if we denote by $f_{j\ell} = {\rm Hess}(f)_{j\ell}$, then we have
\[
\nabla_{i}\nabla_{k}f_{j\ell} + \nabla_{j}\nabla_{\ell}f_{ik} - \nabla_{i}\nabla_{\ell}f_{jk} - \nabla_{j}\nabla_{k}f_{i\ell} = \nabla Rm*\nabla f + Rm*{\rm Hess}(f).
\]
This can easily be seen by computing at a point in normal coordinates, since
\[
\begin{aligned}
\nabla_{i}\nabla_{k}f_{j\ell}- \nabla_{i}\nabla_{\ell}f_{jk} &= \nabla_{i}(\nabla_k\nabla_{\ell}\nabla_jf - \nabla_{\ell}\nabla_{k}\nabla_{j} f)\\
&= -\nabla_{i}(R_{\ell k}\,^p\,_{j}\nabla_{p}f)
\end{aligned}
\]
from which the claim follows.  Therefore, we obtain
\begin{lem}\label{lem: evolRm}
Along the gauge modified flows~\eqref{eq: modFlowVar} or~\eqref{eq: modFlowFix}  we have
\[
\begin{aligned}
\frac{d}{dt} Rm &= \frac{1}{16}\Delta Rm + (\nabla e^{-2\phi})*\nabla^{3}e^{2\phi} + (\nabla^2e^{-2\phi})*(\nabla^{2}e^{2\phi})\\
&\quad +\nabla^2(e^{-2\phi}\cE) +\nabla^2\cE'\\
&\quad +Rm*Rm + Rm*(e^{-2\phi}{\rm Hess}(e^{2\phi}) + e^{-2\phi}\cE + \cE')
\end{aligned}
\]
where $\cE, \cE'$ are defined in~\eqref{eq: errorDefn}
\end{lem}

As an immediate corollary, we obtain

\begin{lem}\label{lem: evolDkRm}
Along the gauge modified spinor flows~\eqref{eq: modFlowVar} or~\eqref{eq: modFlowFix} we have
\[
\begin{aligned}
\frac{d}{dt} \nabla^pRm&= \frac{1}{16}\Delta \nabla^pRm + \sum_{i=1}^{p}\nabla^{p-i}Rm*\nabla^i Rm \\
&\quad +\nabla e^{-2\phi}*\nabla^{p+3}e^{2\phi} + \sum_{i=0}^{p}(\nabla^{i+2}e^{-2\phi})*\nabla^{p-i+2}e^{2\phi}\\ 
&\quad+\nabla^{p+2}(e^{-2\phi}\cE+\cE')\\
&\quad+\sum_{i=0}^{p}\nabla^{p-i}Rm*\nabla^{i}(e^{-2\phi}{\rm Hess}(e^{2\phi}) + e^{-2\phi}\cE + \cE')
\end{aligned}
\]
where the terms $\cE, \cE'$ are defined in~\eqref{eq: errorDefn}
\end{lem}
\begin{proof}
From the formula for the variation of the Christoffel symbols with respect to the metric we have
\[
\frac{d}{dt}\nabla^pRm = \sum_{i=1}^{p}\nabla^i\dot{g}*\nabla^{p-i}Rm + \nabla^p \frac{d}{dt}Rm
\]
We will understand each of these terms separately, starting with the second term.  First, a straightforward induction argument shows that
\[
\nabla^p\Delta Rm = \Delta \nabla^pRm + \sum_{i=1}^{p}\nabla^iRm*\nabla^{p-i}Rm
\]
 and
 \[
 \begin{aligned}
& \nabla^{p}(Rm*Rm + Rm*(e^{-2\phi}{\rm Hess}(e^{2\phi}) +  e^{-2\phi}\cE + \cE') \\
&= \sum_{i=1}^{p}\nabla^{p-i}Rm*\nabla^i Rm + \sum_{i=0}^p\nabla^{p-i}Rm *\nabla^{i}(e^{-2\phi}{\rm Hess}(e^{2\phi}) + e^{-2\phi}\cE +\cE')
 \end{aligned}
 \]
 
and so, all together we obtain
\begin{equation}\label{eq: evolDkRm1}
\begin{aligned}
\frac{d}{dt} \nabla^pRm&= \frac{1}{16}\Delta \nabla^pRm + \sum_{i=1}^{p}\nabla^i\dot{g}*\nabla^{p-i}Rm+ \sum_{i=1}^{p}\nabla^{p-i}Rm*\nabla^i Rm \\
&\quad +\nabla e^{-2\phi}*\nabla^{p+3}e^{2\phi} + \sum_{i=0}^{p}(\nabla^{i+2}e^{-2\phi})*\nabla^{p-i+2}e^{2\phi}\\ 
&\quad+\nabla^{p+2}(e^{-2\phi}\cE+\cE')\\
&\quad+\sum_{i=0}^{p}\nabla^{p-i}Rm*\nabla^{i}(e^{-2\phi}{\rm Hess}(e^{2\phi}) + e^{-2\phi}\cE + \cE')
\end{aligned}
\end{equation}
Now, from the formula for $\dot{g}$ we have
\begin{equation}\label{eq: Didotg}
\nabla^i\dot{g} = \nabla^{i}Rm + \nabla^{i}(e^{-2\phi}{\rm Hess}(e^{2\phi}) + e^{-2\phi}\cE + \cE')
\end{equation}
and so
\begin{equation}\label{eq: evolDkRm2}
\begin{aligned}
\sum_{i=1}^{p}\nabla^i\dot{g}*\nabla^{k-i}Rm &=  \sum_{i=1}^{p}\nabla^{p-i}Rm*\nabla^i Rm \\
&\quad +\sum_{i=0}^{p}\nabla^{p-i}Rm*\nabla^{i}(e^{-2\phi}{\rm Hess}(e^{2\phi}) + e^{-2\phi}\cE + \cE')
\end{aligned}
\end{equation}
Each of these terms already appear in~\eqref{eq: evolDkRm1}. Thus, we obtain the desired formula.
\end{proof}

Our next goal is to estimate $\frac{d}{dt}\int |\nabla^kRm|^2\sqrt{g}$.  Before doing so, we note the following simple lemma.

\begin{lem}\label{lem: stupidInter}
Suppose $T$ is a tensor constructed out of the metric $g$ and the fields $H, \psi, Rm, e^{2\phi}$.  Then, for any $p \geq 0$ we have
\[
\int_{M}|\nabla^p T|^2\sqrt{g}  \leq \frac{p\epsilon}{p+1} \int_{M}|\nabla^{p+1}T|^2\sqrt{g} + \frac{1}{p+1}\epsilon^{-p}\int_{M}|T|^2\sqrt{g}
\]
Furthermore, for can $j<p$ we have
\[
\int_{M}|\nabla^j T|^2\sqrt{g}  \leq \frac{j\epsilon^{p-j}}{p} \int_{M}|\nabla^{p}T|^2\sqrt{g} + j\left(\sum_{i=j}^{p-1}\frac{1}{i(i+1)}\right)\epsilon^{-j}\int_{M}|T|^2\sqrt{g}
\]
\end{lem}
\begin{proof}
We establish the first statement, and then show that the first statement implies the second statement. Integration by parts yields
\[
\int_{M}|\nabla^p T|^2\sqrt{g}  = - \int_{M}\langle \nabla^{p-1}T, \nabla^{p+1}T\rangle \sqrt{g}.
\]
By Cauchy-Schwarz we obtain
\begin{equation}\label{eq: inductionStep}
 \int_{M}|\nabla^p T|^2\sqrt{g} \leq \frac{\epsilon}{2} \int_{M}|\nabla^{p+1}T|^2 + \frac{1}{2\epsilon}\int_{M}|\nabla^{p-1}T|^2
\end{equation}
We now argue by induction.  If $p=0, 1$ then we are done.  Now suppose that the result holds for all $j\leq p-1$.  Then by~\eqref{eq: inductionStep} and the induction hypothesis we have
\[
\begin{aligned}
\int_{M}|\nabla^p T|^2\sqrt{g}  &\leq \frac{\epsilon}{2} \int_{M}|\nabla^{p+1}T|^2\sqrt{g} + \frac{1}{2\epsilon}\int_{M}|\nabla^{p-1}T|^2\\
& \leq\frac{\epsilon}{2}\int_{M}|\nabla^{p+1}T|^2\sqrt{g}\\
&\quad + \frac{1}{2\epsilon}\left(\frac{(p-1)\epsilon}{p} \int_{M}|\nabla^pT|^2\sqrt{g} + \frac{1}{p\epsilon^{p-1}}\int_{M}|T|^2\sqrt{g}\right)\\
\end{aligned}
\]
Thus we obtain
\[
\frac{p+1}{2p}\int_{M}|\nabla^p T|^2\sqrt{g}   \leq \frac{\epsilon}{2}\int_{M}|\nabla^{p+1}T|^2\sqrt{g}+\frac{1}{2p\epsilon^{p}}\int_{M}|T|^2\sqrt{g}
\]
and the inequality follows.  We now iterate this inequality to obtain the second estimate.  If $p-j=1$, then we have already established the result.  Assume now that the result holds for $p-j<\ell$.  Then
\[
\int_{M}|\nabla^j T|^2\sqrt{g}  \leq \frac{j\epsilon^{p-1-j}}{p-1} \int_{M}|\nabla^{p-1}T|^2\sqrt{g} + j\left(\sum_{i=j}^{p-2}\frac{1}{i(i+1)}\right)\epsilon^{-j}\int_{M}|T|^2\sqrt{g}
\]
By the first inequality we have
\[
\int_{M}|\nabla^{p-1}T|^2\sqrt{g} \leq \frac{(p-1)\epsilon}{p}\int_{M}|\nabla^kT|^2\sqrt{g}+\frac{1}{p}\epsilon^{1-p}\int_{M}|T|^2\sqrt{g}
\]
and so
\[
\int_{M}|\nabla^j T|^2\sqrt{g}  \leq \frac{j\epsilon^{p-j}}{p} \int_{M}|\nabla^{p}T|^2\sqrt{g} + \left(\frac{j}{(p-1)p} + j\sum_{i=j}^{p-2}\frac{1}{i(i+1)}\right)\epsilon^{-j}\int_{M}|T|^2\sqrt{g}
\]
and the result follows.
\end{proof}

We are now in a position to prove the desired estimate.  We will need the following interpolation result, which is due to Hamilton.

\begin{lem}[Corollary 12.6, \cite{Ham}]\label{lem: HamInterp}
If $T$ is any tensor on a manifold $M$, $k \in \mathbb{N}$ and $1\leq i \leq n $, then there is a constant $C = C(n, \dim M)$ such that
\[
\int_{M} |\nabla^{i}T|^{\frac{2n}{i}}\sqrt{g} \leq C \|T\|_{L^{\infty}(M)}^{2(\frac{n}{i}-1)} \int_{M}|\nabla^nT|^2 \sqrt{g}
\]
\end{lem}

We will use the following corollary of Lemma~\ref{lem: HamInterp};  see, for example \cite[Lemma 8.2]{HeWang}

\begin{cor}\label{cor: genHamInterp}
If $T_{1},\ldots, T_{p+r}$ are any tensors, and $\ell_i\in \mathbb{N}$ for $1 \leq i \leq p$ are integers such that $1\leq \ell_i \leq k$, such that $\sum_{i=1}^{p}\ell_i=2k$, then there is a constant $C$, depending on $k$ and $\dim M$ such that 
\[
\begin{aligned}
&\big|\int_{M}\nabla^{\ell_1}T_1 *\cdots *\nabla^{\ell_p}T_p * T_{p+1}*\cdots *T_{p+r} \sqrt{g}\big|\\
& \leq C\prod_{j=1}^{r}\|T_{p+j}\|_{L^{\infty}}\cdot \prod_{s=1}^{p}\|T_s\|_{L^{\infty}}^{1-\frac{\ell_s}{k}} \left(\int_{M}|\nabla^k T_s|^2 \sqrt{g} \right)^{\frac{\ell_s}{2k}}
\end{aligned}
\]
\end{cor}

We now can prove

\begin{prop}\label{prop: BndIntDpRm}
Suppose that $|\nabla \psi|^2, |\nabla^2\psi|, |H|^2, |\nabla H|,  |Rm|$ are bounded by $A$ and that $\phi >-\hat{A}$ along either of the flows~\eqref{eq: modFlowVar} or~\eqref{eq: modFlowFix} on $[0,\tau]$.  Then there exists a constant $C_1$ depending only on $\dim M, p, A, \hat{A}, {\rm osc}_{M}\phi(0)$, and an upper bound for $\tau$ such that we have
\[
\begin{aligned}
&\frac{d}{dt}\int_{M}|\nabla^{p}Rm|^2\sqrt{g} \leq -\frac{1}{16}\int_{M}|\nabla^{p+1}Rm|^2\sqrt{g} + C_1\int_{M}|\nabla^{p+2}H|^2 \sqrt{g}\\
&+ C_1\int_{M}\left(|\nabla^pRm|^2+|\nabla^{p+1}H|^2+|\nabla^{p+2}e^{2\phi}|^2 + |\nabla^{p+2}\psi|^2 +1\right) \sqrt{g}
\end{aligned}
\]
\end{prop}
\begin{proof}
Note that since $|\nabla \psi|^2, |\nabla^2\psi|, |H|^2, |\nabla H|$ are bounded, we have $\phi <C$ for $C$ depending only on $A, \|\phi(0)\|_{L^{\infty}}$ and an upper bound for $\tau$, by Lemma~\ref{lem: supPhiBnd}.  Furthermore, since bounds $\phi, \nabla \psi, \nabla^2\psi$ imply bounds on ${\rm Hess}(e^{2\phi}$ we have a bound $|\dot{g}|<C$.  Thus,
\[
\frac{d}{dt}\int_{M}|\nabla^{p}Rm|^2 \sqrt{g} \leq C\int_{M}|\nabla^pRm|^2 \sqrt{g} + 2\int_{M}\langle \frac{d}{dt}\nabla^pRm, \nabla^pRm \rangle \sqrt{g}
\]
To simplify the notation let us suppress the volume form $\sqrt{g}$.  Furthermore, $C>0$ will be a constant which can change from line to line, but depends only on $A, \hat{A}, p$.  Then by Lemma~\ref{lem: evolDkRm} we have
\[
\begin{aligned}
\frac{d}{dt}\int_{M}|\nabla^{p}Rm|^2 &\leq C\int_{M}|\nabla^pRm|^2 \sqrt{g} -\frac{1}{8} \int_{M}|\nabla^{p+1}Rm|^2\\
&\quad + \sum_{i=1}^{p}\int_{M}\nabla^{p-i}Rm*\nabla^i Rm*\nabla^pRm \\
&\quad +\int_{M}\nabla e^{-2\phi}*\nabla^{p+3}e^{2\phi}*\nabla^pRm\\
&\quad + \int_{M}\sum_{i=0}^{p}(\nabla^{i+2}e^{-2\phi})*(\nabla^{p-i+2}e^{2\phi})*\nabla^{p}Rm\\ 
&\quad+\int_{M}\nabla^{p+2}(e^{-2\phi}\cE+\cE')*\nabla^{p}Rm\\
&\quad+\sum_{i=0}^{p}\int_{M}\nabla^{p-i}Rm*\nabla^{i}(e^{-2\phi}{\rm Hess}(e^{2\phi}) + e^{-2\phi}\cE + \cE')*\nabla^{p}Rm\\
&= C\int_{M}|\nabla^pRm|^2 \sqrt{g} -\frac{1}{8} \int_{M}|\nabla^{p+1}Rm|^2\\
&\quad + (I)+(II)+(III)+(IV)+(V)
\end{aligned}
\]
Most terms above are estimated in essentially the same way, using a combination of integration by parts and the interpolation inequalities Corollary~\ref{cor: genHamInterp} and Lemma~\ref{lem: stupidInter}.  For example, applying Corollary~\ref{cor: genHamInterp} with $k=p$ yields 
\begin{equation}\label{eq: RmBd1}
(I) \leq C\int_{M}|\nabla^pRm|^2 
\end{equation}
for a constant $C$ depending only on $p, \dim M, A$.

For term $(II)$ we integrate by parts once to get
\[
\begin{aligned}
\int_{M}\nabla e^{-2\phi}*\nabla^{p+3}e^{2\phi}*\nabla^pRm&=-\int_{M}\nabla^2 e^{-2\phi}*\nabla^{p+2}e^{2\phi}*\nabla^pRm\\
&\quad -\int_{M}\nabla e^{-2\phi}*\nabla^{p+2}e^{2\phi}*\nabla^{p+1}Rm
\end{aligned}
\]
and hence we obtain
\begin{equation}\label{eq: RmBd2}
|(II)| \leq C\epsilon^{-1}\int_{M} |\nabla^{p+2}e^{2\phi}|^2 + |\nabla^p Rm|^2 + \epsilon \int_{M}|\nabla^{p+1}Rm|^2.
\end{equation}
Similarly, an easy induction using  Corollary~\ref{cor: genHamInterp} and Lemma~\ref{lem: stupidInter} yields
\begin{equation}\label{eq: RmBd3}
|(III)| \leq C \int_{M}|\nabla^{p+2}e^{2\phi}|^2+|\nabla^pRm|^2 +1
\end{equation}
It remains to estimate terms $(IV)$ and $(V)$.  Note that these terms are problematic since, for example $\nabla^{p+2}\cE$ contains terms propositional to $\nabla^{p+3}H$, which cannot be controlled by any term appearing in the evolution of $|\nabla^{p+1}H|^2$.  Thus, integration by parts is a necessity. Let us first consider term $(IV)$.  We integrate by parts to get
\[
(IV) = -\int_{M}\nabla^{p+1}(e^{-2\phi}\cE +\cE')*\nabla^{p+1}Rm
\]
Now observe that
\[
\nabla^{p+1}(e^{-2\phi}\cE) = \sum_{r=0}^{p+1}(\nabla^{r}e^{-2\phi})\nabla^{p+1-r} \left(\nabla \psi^2 + \psi^2*\nabla H+\psi*\nabla \psi *H \right).\\
\]
Clearly there are combinatorial coefficients so that
\begin{equation}\label{eq: derInv}
\nabla^{r}e^{-2\phi} = \sum_{\ell =1}^{r}e^{-2(\ell+1)\phi} \sum_{\substack{(i_1,\ldots, i_\ell)\in \mathbb{Z}_{>0}^\ell\\ i_1+\cdots +i_\ell=r}}c_{i_1,\ldots,i_r}\nabla^{i_1}e^{2\phi}*\cdots*\nabla^{i_\ell}e^{2\phi}.
\end{equation}
Thus, we can apply the interpolation inequality Corollary~\ref{cor: genHamInterp} with $k=p+1$ and Lemma~\ref{lem: stupidInter} to obtain
\[
\begin{aligned}
\big|\int_{M}\nabla^{p+1}(e^{-2\phi}\cE)*\nabla^{p+1}Rm\big| &\leq \epsilon \int_{M}|\nabla^{p+1}Rm|^2\\
&\quad + C\epsilon^{-1} \int_{M} |\nabla^{p+2}\psi|^2  + |\nabla^{p+1}H|^2 + |\nabla^{p+2}e^{2\phi}|^2\\
&\quad+ C\epsilon^{-1}\int_{M}|\nabla^{p+2}H|^2+1.
\end{aligned}
\]
The $\cE'$ term can be treated identically. We obtain
\begin{equation}\label{eq: RmBd4}
\begin{aligned}
|(IV)| &\leq    \epsilon \int_{M}|\nabla^{p+1}Rm|^2\\
&\quad + C\epsilon^{-1} \int_{M} |\nabla^{p+2}\psi|^2  + |\nabla^{p+1}H|^2 + |\nabla^{p+2}e^{2\phi}|^2\\
&\quad+ C\epsilon^{-1}\int_{M}|\nabla^{p+2}H|^2+1.
\end{aligned}
\end{equation}
Finally, term $(V)$ can be treated in the same way as term $(IV)$ using the interpolation inequality~\ref{cor: genHamInterp} with $k=p$ together with Lemma~\ref{lem: stupidInter}
\begin{equation}\label{eq: RmBd5}
|(V)| \leq C\int_{M} |\nabla^pRm|^2 + |\nabla^{p+2}e^{2\phi}|^2 + |\nabla^{p+2}\psi|^2 + |\nabla^{p+1}H|^2 + 1
\end{equation}
Finally, summing up~\eqref{eq: RmBd1}~\eqref{eq: RmBd2}~\eqref{eq: RmBd3}~\eqref{eq: RmBd4}~\eqref{eq: RmBd5} and choosing $\epsilon$ appropriate we obtain the desired result.
\end{proof}

\subsection{The flow of $\na^pH$}\hfil\break
We now turn our attention to the evolution of $H$ and its derivatives. 
\begin{prop}\label{prop: evolDpH}
Along either of the gauge modified spinor flows with flux given by~\eqref{eq: modFlowVar} or~\eqref{eq: modFlowFix}  we have
\[
\begin{aligned}
\frac{d}{dt} \nabla^pH &=\Delta \nabla^pH + \sum_{i=0}^{p}\nabla^{p-i}Rm*\nabla^iH\\
&+\nabla^{p}\left(c\nabla H*H +cH*H*H\right)\\
&+\nabla^{p}\left(\nabla(e^{-2\phi}\psi*\nabla \psi)*H + e^{-2\phi}\psi*\nabla\psi*\nabla H\right)\\
&+ \sum_{i=1}^{p}\nabla^{i}\left(e^{-2\phi}{\rm Hess}(e^{2\phi}) + e^{-2\phi}\cE + \cE'\right)*\nabla^{p-i}H 
\end{aligned}
\]
\end{prop}
\begin{proof}
First, by the Bochner formula we have
\[
-\square H = \Delta H + Rm*H
\]
and so the flow of $H$ can be written as
\[
\left(\frac{d}{dt}-\Delta\right) H=Rm*H -cd(\star(H \wedge H)) +L(H) - \frac{1}{4}(d\iota_{X}H + \iota_{X}dH)
\]
From the definition of $L(H)$ we have
\[
L(H) = c(\nabla H*H+H*H*H).
\]
while from the definition of $X$ in~\eqref{eq: HWVF} we can write
\[
(d\iota_{X}H + \iota_{X}dH)= \nabla(e^{-2\phi}\psi*\nabla \psi)*H + e^{-2\phi} \psi*\nabla\psi*\nabla H.
\]
Furthermore, we have
\[
-cd(\star(H \wedge H))=c\nabla H*H.
\]
Now, we have
\[
\begin{aligned}
\frac{d}{dt} \nabla^pH &= \sum_{i=1}^{p}(\nabla^{i}\dot{g})*\nabla^{p-i}H + \nabla^{p}\frac{d}{dt}H\\
&= \sum_{i=1}^{p}(\nabla^{i}\dot{g})*\nabla^{p-i}H + \nabla^{p}\left(\Delta H +c\nabla H*H +cH*H*H\right)\\
&\quad +\nabla^{p}\left(Rm*H+\nabla(e^{-2\phi}\psi*\nabla \psi)*H + e^{-2\phi}\psi*\nabla\psi*\nabla H\right).
\end{aligned}
\]
A straightforward induction shows that
\[
\nabla^{p}\Delta H= \Delta \nabla^p H + \sum_{i=0}^{p}\nabla^{p-i}Rm*\nabla^iH.
\]
From the flow of $g$ as written in~\eqref{eq: flowgError} we obtain the desired result.
\end{proof}

We now prove
\begin{prop}\label{prop: BndIntDpH}
Suppose that $|\nabla \psi|^2, |\nabla^2\psi|, |H|^2, |\nabla H|,  |Rm|$ are bounded by $A$ and that $\phi >-\hat{A}$ along either of the flows~\eqref{eq: modFlowVar} or~\eqref{eq: modFlowFix} on $[0,\tau]$.  Then there exists a constant $C_2$ depending only on $\dim M, p, A, \hat{A}, {\rm osc}_{M}\phi(0)$ and an upper bound for $\tau$ such that we have
\[
\begin{aligned}
\frac{d}{dt}\int_{M}|\nabla^{p}H|^2 \sqrt{g} &\leq -\int_{M}|\nabla^{p+1}H|^2 \sqrt{g}\\
& \quad+ C_2\int_{M}\left(|\nabla^{p-1}Rm|^2+|\nabla^pH|^2\right)\sqrt{g}\\
&\quad +C_2\int_{M}\left(|\nabla^{p+1}\psi|^2 +|\nabla^{p+1}e^{2\phi}|^2 +1\right)\sqrt{g}
\end{aligned}
\]
\end{prop}
\begin{proof}
We compute
\[
\begin{aligned}
\frac{d}{dt}\int_{M}|\nabla^pH|^2 \sqrt{g} &=\int_{M} (\nabla^pH*\nabla^pH*\dot{g} + |\nabla^pH|^2\frac{1}{2}{\rm Tr}_{g}\dot{g}) \sqrt{g}\\
&\quad+2\int_{M}\langle \frac{d}{dt}\nabla^pH, \nabla^pH\rangle \sqrt{g}
\end{aligned}
\]
From the assumptions we clearly have
\[
\int_{M} (\nabla^pH*\nabla^pH*\dot{g} + |\nabla^pH|^2\frac{1}{2}{\rm Tr}_{g}\dot{g}) \sqrt{g} \leq C\int_{M}|\nabla^pH|^2\sqrt{g}
\]
so it suffices to understand the second term.  From Proposition~\ref{prop: evolDpH} we have
\[
\begin{aligned}
&\int_{M}\langle \frac{d}{dt}\nabla^pH, \nabla^pH\rangle\\
 &= -\int_{M}|\nabla^{p+1}H|^2 +\sum_{i=0}^{p}\int_{M}\nabla^{p-i}Rm*\nabla^iH*\nabla^pH\\
&+\int_{M}\nabla^{p}\left(c\nabla H*H +cH*H*H\right)*\nabla^pH\\
&+\int_{M}\nabla^{p}\left(\nabla(e^{-2\phi}\psi*\nabla \psi)*H + e^{-2\phi}\psi*\nabla\psi*\nabla H\right)*\nabla^pH\\
&+\int_{M} \sum_{i=1}^{p}\nabla^{i}\left(e^{-2\phi}{\rm Hess}(e^{2\phi}) + e^{-2\phi}\cE + \cE'\right)*\nabla^{p-i}H*\nabla^pH\\
&= -\int_{M}|\nabla^{p+1}H|^2 + (I)+(II)+(III)+(IV)
 \end{aligned}
 \]
 where we have suppressed the volume form for simplicity.  We can now estimate each of these terms using the interpolation inequalities in Corollary~\ref{cor: genHamInterp} and Lemma~\ref{lem: stupidInter} and integration by parts.  We begin with term $(I)$.  We estimate the terms with $i=0, p$ separately.  When $i=p$ we have the trivial bound
 \[
 \big|\int_{M}Rm*\nabla^pH*\nabla^pH \big| \leq C\int_{M}|\nabla^pH|^2.
 \]
 For $i=0$, integrating by parts we have
 \[
\int_{M} \nabla^{p}Rm*H*\nabla^pH = -\int_{M}\nabla^{p-1}Rm*\nabla H *\nabla^pH + \nabla^{p-1}Rm*H*\nabla^{p+1}H
\]
Now, since $H, \nabla H$ are bounded we get
\[
\big|\int_{M} \nabla^{p}Rm*H*\nabla^pH\big| \leq  \epsilon\int_{M}|\nabla^{p+1}H|^2+ C\epsilon^{-1}\int_{M}|\nabla^{p-1}Rm|^2 +|\nabla^pH|^2
\]
for $C$ depending only on $A, \dim M, p$.  For $i=p$   For the remaining terms we write
\[
 \sum_{i=1}^{p-1}\big|\int_{M} \nabla^{p-i}Rm*\nabla^iH*\nabla^pH\big| =  \sum_{j=0}^{p-2}\big|\int_{M} \nabla^{p-1-j}Rm*\nabla^{j}(\nabla H)*\nabla^{p-1}(\nabla H)\big|
 \]
 and apply Corollary~\ref{cor: genHamInterp} with $k=p-1$ to obtain
\[
\begin{aligned}
\int_{M} \sum_{i=1}^{p}\big|\int_{M} \nabla^{p-i}Rm*\nabla^iH*\nabla^pH\big| &\leq  C\left(\int_{M} |\nabla^{p-1}Rm|^2+|\nabla^pH|^2\right).
\end{aligned}
\]
Thus, in total we have
\begin{equation}\label{eq: H1Bnd}
|(I)|\leq \epsilon\int_{M}|\nabla^{p+1}H|^2+C\epsilon^{-1}\left(\int_{M} |\nabla^{p-1}Rm|^2+|\nabla^pH|^2\right),
\end{equation}
for $C$ depending only on $\dim M, p, A$.

The interpolation inequality allows us to easily bound the second term
\begin{equation}\label{eq: H2Bnd}
|(II)| \leq\epsilon\int_{M}|\nabla^{p+1}H|^2+ C\epsilon^{-1}\int_{M}|\nabla^pH|^2
\end{equation}
For term $(III)$, we write
\[
\begin{aligned}
(III)&= \int_{M} \nabla^{p}(\nabla(e^{-2\phi}\psi*\nabla\psi)*H)*\nabla^{p}H\\
&+\int_{M}\nabla^{p}(e^{-2\phi}\psi*\nabla \psi*\nabla H)*\nabla^{p}H\\
&= (IIIa)+(IIIb)
\end{aligned}
\]
To estimate $(IIIa)$ we integrate by parts once to get
\begin{equation}\label{eq: split3a}
\begin{aligned}
(IIIa) &= - \int_{M} \nabla^{p-1}(\nabla(e^{-2\phi}\psi*\nabla\psi)*H)*\nabla^{p}(\nabla H)\\
&= -\int_{M} \nabla^{p}(e^{-2\phi}\psi*\nabla\psi)*H *\nabla^p(\nabla H)\\
&\quad -\int_{M} \sum_{r=1}^{p-1} \nabla^{p-r}(e^{-2\phi} \psi*\nabla\psi)*\nabla^{r}H*\nabla^{p}(\nabla H)
\end{aligned}
\end{equation}
The second term can be easily estimated by combining~\eqref{eq: derInv} with the interpolation inequality from Corollary~\ref{cor: genHamInterp} for $k=p$, together with Lemma~\ref{lem: stupidInter} to get
\[
\begin{aligned}
&\big|\int_{M} \sum_{r=1}^{p-1} \nabla^{p-r}(\psi*\nabla\psi)*\nabla^{r}H*\nabla^{p}(\nabla H)\big| \\ 
&\leq\epsilon \int_{M}|\nabla^{p+1}H|^2+  C\epsilon^{-1}\int_{M}|\nabla^{p+1}\psi|^2 + |\nabla^{p}H|^2 + |\nabla^{p+1}e^{2\phi}|^2 +1.
\end{aligned}
\]
For the first term appearing~\eqref{eq: split3a} we have
\[
\begin{aligned}
\big|\int_{M}& \nabla^{p}(e^{-2\phi} \psi*\nabla\psi)*H *\nabla^p(\nabla H)\big| \leq \big|\int_{M}e^{-2\phi}\psi *\nabla^{p+1}\psi*H*\nabla^{p}(\nabla H)\big|\\
&\quad +\big|\int_{M} \sum_{r=0}^{p-1}\nabla^{p-r}(e^{-2\phi}\psi) *\nabla^{r}(\nabla \psi)*H*\nabla^{p}(\nabla H)\big|\\
&\quad \leq \epsilon \int_{M}|\nabla^{p+1}H|^2 + C\epsilon^{-1}\int_{M} |\nabla^{p+1}\psi|^2 + |\nabla^{p+1}e^{2\phi}|^2 + 1
\end{aligned}
\]
where we bounded second term is bounded from Corollary~\ref{cor: genHamInterp} with $k=p$, Lemma~\ref{lem: stupidInter} and Young's inequality, while the first term is bounded by H\"older's inequality and Young's inequality.  So in total we have
\[
(IIIa) \leq \epsilon \int_{M}|\nabla^{p+1}H|^2 +  C\epsilon^{-1}\int_{M}|\nabla^{p+1}\psi|^2 + |\nabla^{p}H|^2+1
\]
For term $(IIIb)$ it follows immediately from Corollary~\ref{cor: genHamInterp} with $k=p$ that
\[
(IIIb)\leq \epsilon \int_{M}|\nabla^{p+1}H|^2 + C\epsilon^{-1}\int_{M}|\nabla^{p+1}\psi|^2+ |\nabla^{p}H|^2 + |\nabla^{p+1}e^{2\phi}|^2+1
\]
In total, we have
 \begin{equation}\label{eq: H3Bnd}
 |(III)| \leq \epsilon \int_{M}|\nabla^{p+1}H|^2 + C\epsilon^{-1}\int_{M}|\nabla^{p+1}\psi|^2+ |\nabla^{p}H|^2 + |\nabla^{p+1}e^{2\phi}|^2+1
 \end{equation}
 for any $0<\epsilon <1$.
 
 To estimate term $(IV)$ we first treat the $i=p$ term in the sum.  Integrating by parts yields
 \[
 \begin{aligned}
&\int_{M} \nabla^{p}\left(e^{-2\phi}{\rm Hess}(e^{2\phi}) + e^{-2\phi}\cE + \cE'\right)*H*\nabla^pH\\
&= -\int_{M} \nabla^{p-1}\left(e^{-2\phi}{\rm Hess}(e^{2\phi}) + e^{-2\phi}\cE + \cE'\right)*\nabla H*\nabla^pH\\
&\quad -\int_{M} \nabla^{p-1}\left(e^{-2\phi}{\rm Hess}(e^{2\phi}) + e^{-2\phi}\cE + \cE'\right)* H*\nabla^{p+1}H
\end{aligned}
  \]
 Now, recalling that by~\eqref{eq: errorDefn} $\cE, \cE'$ are expressions of order at most $1$ in $H,\psi, e^{2\phi}$, Corollary~\ref{cor: genHamInterp} with $k=p-1$ and Lemma~\ref{lem: stupidInter} imply
 \[
 \begin{aligned}
&\big|\int_{M} \nabla^{p}\left(e^{-2\phi}{\rm Hess}(e^{2\phi}) + e^{-2\phi}\cE + \cE'\right)*H*\nabla^pH\big|\\
&\leq \epsilon \int_{M}|\nabla^{p+1}H|^2 + C\epsilon^{-1}\int_{M}|\nabla^{p+1}\psi|^2+ |\nabla^{p+1}e^{2\phi}|^2 + |\nabla^{p}H|^2 +1
\end{aligned}
\]
Finally, for the remaining terms in $(IV)$ we write
\[
\begin{aligned}
&\int_{M} \sum_{i=1}^{p-1}\nabla^{i}\left(e^{-2\phi}{\rm Hess}(e^{2\phi}) + e^{-2\phi}\cE + \cE'\right)*\nabla^{p-i}H*\nabla^pH\\
&=\int_{M} \sum_{i=1}^{p-1}\nabla^{i}\left(e^{-2\phi}{\rm Hess}(e^{2\phi}) + e^{-2\phi}\cE + \cE'\right)*\nabla^{p-i-1}(\nabla H)*\nabla^{p-1}(\nabla H)
\end{aligned}
\]
and an application of Corollary~\ref{cor: genHamInterp} and Lemma~\ref{lem: stupidInter} together with~\eqref{eq: derInv} yields
\begin{equation}\label{eq: H4Bnd}
|(IV)| \leq \epsilon \int_{M}|\nabla^{p+1}H|^2 + C\epsilon^{-1}\int_{M}|\nabla^{p+1}\psi|^2+ |\nabla^{p+1}e^{2\phi}|^2 + |\nabla^{p}H|^2 +1
\end{equation}
 Adding~\eqref{eq: H1Bnd}~\eqref{eq: H2Bnd}~\eqref{eq: H3Bnd} and ~\eqref{eq: H4Bnd}  and choosing $\epsilon$ sufficiently small we obtain
\[
\begin{aligned}
\frac{d}{dt}\int_{M}|\nabla^{p}H|^2 \sqrt{g} &\leq -\int_{M}|\nabla^{p+1}H|^2 \sqrt{g}\\
& \quad+ C\int_{M}\left(|\nabla^{p-1}Rm|^2+|\nabla^pH|^2\right)\sqrt{g}\\
&\quad +C\int_{M}\left(|\nabla^{p+1}\psi|^2 +|\nabla^{p+1}e^{2\phi}|^2 +1\right)\sqrt{g}
\end{aligned}
\]
which is the desired estimate.
\end{proof}

\subsection{The flow of the spinor field $\na^p\psi$}\hfil\break
Next we need to understand the evolution of $\psi$ and its derivatives.  

\begin{prop}\label{prop: evolDpPsi}
Along either of the gauge modified spinor flows with flux given by~\eqref{eq: modFlowVar} or~\eqref{eq: modFlowFix}  we have
\[
\begin{aligned}
\frac{d}{dt}\nabla^{p}\psi &=\Delta \nabla^{p}\psi+  \frac{1}{16} \sum_{i,j}e^{-2\phi}{\rm Re}\left(\langle \psi, \gamma^{[j}\nabla_{i]}\slashed{D}\nabla^{p}\psi \rangle\right) \gamma^{ij}\psi + \nabla^{p}(c(t)\psi)\\
&\quad+\sum_{i=0}^{p}\nabla^{p-i}Rm *\nabla^{i}\psi +\sum_{i=1}^{p}\nabla^{p-i}\psi*\nabla^i(e^{-2\phi}{\rm Hess}(e^{2\phi}) + e^{-2\phi}\cE +\cE')\\
&\quad+ \sum_{\ell=1}^{p} \nabla^{\ell}(e^{-2\phi}\psi)*\nabla^{p-\ell+2}\psi*\psi+\sum_{r=1}^{p}\sum_{\ell=0}^{p-r}\nabla^{p-r-\ell}(e^{-2\phi}\psi) *\nabla^{\ell+2}\psi *\nabla^{r}\psi\\
&\quad +\sum_{r=0}^{p}e^{-2\phi}\nabla^{r}Rm*\nabla^{p-r}\psi+\nabla^{p}\left(e^{-2\phi}\nabla \psi*\nabla\psi *\psi+\nabla e^{-2\phi}*\nabla\psi *\psi^2\right).
\end{aligned}
\]
\end{prop}
\begin{proof}
We compute
\[
\begin{aligned}
\frac{d}{dt}\nabla^p\psi = \sum_{i=1}^{p}\nabla^{i}\dot{g}*\nabla^{p-i}\psi + \nabla^{p}\frac{d\psi}{dt}
\end{aligned}
\]
From the evolution of $\psi$ we have
\[
\frac{d\psi}{dt} =-\sum_{a} (\nabla_{a}^{H})^{\dagger}\nabla_{a}^{H}\psi + \frac{1}{8}\mathcal{L}_{X}\psi + c(t)\psi
\]
where $X= e^{-2\phi}{\rm Re}\left(\langle \psi, \gamma^{\ell}\slashed{D}\psi \rangle \right)e_{\ell}$ and $\mathcal{L}_{X}\psi$ denotes the spinorial Lie derivative of Bourguignon-Gauduchon \cite{BG}.  The function $c(t)$ depends on whether $\phi$ is dynamical or fixed, according to~\eqref{eq: flowConst} and~\eqref{eq: flow1Const} respectively.  We will treat these two cases separately. Define
\[
\Delta^{H}\psi := \sum_{a}( \nabla_{a}^{H})^{\dagger}\psi \qquad \Delta \psi = \sum_{a} (\nabla_{a})^{\dagger}\nabla_{a}\psi.
\]
Clearly we have
\[
\Delta^{H}\psi = \Delta\psi +\nabla (H*\psi) +H*\nabla \psi +H*H*\psi,
\]
and so, commuting derivatives through the Laplacian yields
\[
\begin{aligned}
\nabla^{p} \Delta^{H}\psi &= \Delta \nabla^{p}\psi + \sum_{i=0}^{p}\nabla^{p-i}Rm *\nabla^{i}\psi \\
&\quad + \nabla^{p+1}(H*\psi) + \nabla^p(H*\nabla \psi+H*H*\psi).
\end{aligned}
\]
Thus, using the formula~\eqref{eq: flowgError} for $\nabla^{i}\dot{g}$ in total we have
\begin{equation}\label{eq: step1EvolDpPsi}
\begin{aligned}
\frac{d}{dt}\nabla^p \psi  &= \Delta \nabla^{p}\psi + \frac{1}{8}\nabla^{p}\mathcal{L}_{X}\psi + \nabla^{p}(c(t)\psi)+\sum_{i=0}^{p}\nabla^{p-i}Rm *\nabla^{i}\psi \\
&\quad+\sum_{i=1}^{p}\nabla^{p-i}\psi*\nabla^i(e^{-2\phi}{\rm Hess}(e^{2\phi}) + e^{-2\phi}\cE +\cE')\\
&\quad+ \nabla^{p+1}(H*\psi) + \nabla^p(H*\nabla \psi+H*H*\psi) 
\end{aligned}
\end{equation}
Recall that the spinorial Lie derivative  defined in \cite{BG} is given by
\[
\mathcal{L}_{X}\psi = \nabla_{X}\psi +\frac{1}{4}dX^{\flat}\cdot \psi
\]
where $X^{\flat}$ denotes the $1$-form dual to $X$ using the metric, and $dX^{\flat} \cdot\psi$ denotes Clifford multiplication.  Note that the plus sign in front of the $dX^{\flat}$ term is due to our convention for the Clifford algebra; see Appendix~\ref{sec: conv}.  In particular, with $X=e^{-2\phi} {\rm Re}\left(\langle \psi, \gamma^{\ell}\slashed{D}\psi \rangle \right)e_{\ell}$ we have
\[
d X^{\flat} (e_i, e_i) = \langle \nabla_{i} X, e_{j}\rangle - \langle \nabla_{j}X, e_i \rangle.
\]
Computing at a point, in an orthonormal frame where $\nabla_{e_i}e_j=0$ we have
\[
\langle\nabla_{i}X, e_j \rangle =  e^{-2\phi}{\rm Re}\left(\langle\nabla_{i} \psi, \gamma^{j}\slashed{D}\psi \rangle \right) +  e^{-2\phi}{\rm Re}\left(\langle \psi, \gamma^{j}\nabla_{i}\slashed{D}\psi \rangle \right) + \nabla_ie^{-2\phi}{\rm Re}\left(\langle \psi, \gamma^{j}\slashed{D}\psi \rangle \right)
\]
and so
\[
\begin{aligned}
dX^{\flat}\cdot \psi &= 2\sum_{i,j}e^{-2\phi}\left( {\rm Re}\left(\langle\nabla_{[i} \psi, \gamma^{j]}\slashed{D}\psi \rangle \right) +  {\rm Re}\left(\langle \psi, \gamma^{[j}\nabla_{i]}\slashed{D}\psi \rangle \right)\right)\gamma^{ij}\psi\\
&\quad +2\sum_{i,j} \nabla_{[i}e^{-2\phi}{\rm Re}\left(\langle \psi, \gamma^{j]}\slashed{D}\psi \rangle \right)\gamma^{ij}\psi.
\end{aligned}
\]
Now we can write $\nabla_{X}\psi = e^{-2\phi}\nabla \psi*\nabla\psi*\psi$, and similarly
\[
\begin{aligned}
2e^{-2\phi}\sum_{i,j}{\rm Re}\left(\langle\nabla_{[i} \psi, \gamma^{j]}\slashed{D}\psi \rangle \right)\gamma^{ij}\psi &= e^{-2\phi}\nabla \psi*\nabla\psi*\psi\\
2\sum_{i,j} \nabla_{[i}e^{-2\phi}{\rm Re}\left(\langle \psi, \gamma^{j]}\slashed{D}\psi \rangle \right)\gamma^{ij}\psi &= \nabla e^{-2\phi}*\nabla\psi *\psi^2
\end{aligned}
\]
Finally, we have
\[
\begin{aligned}
\nabla^{p}\left( e^{-2\phi}\sum_{i,j}{\rm Re}\left(\langle \psi, \gamma^{j}\nabla_{i}\slashed{D}\psi \rangle \right)\gamma^{ij}\psi\right) &= \left( e^{-2\phi}\sum_{i,j}{\rm Re}\left(\langle \psi, \gamma^{j}\nabla^{p}\nabla_{i}\slashed{D}\psi \rangle \right)\gamma^{ij}\psi\right)\\
&+ \sum_{\ell=1}^{p} \nabla^{\ell}(e^{-2\phi}\psi)*\nabla^{p-\ell+2}\psi*\psi\\
&+\sum_{r=1}^{p}\sum_{\ell=0}^{p-r}\nabla^{p-r-\ell}(e^{-2\phi}\psi) *\nabla^{\ell+2}\psi *\nabla^{r}\psi.
\end{aligned}
\]
We now focus on the first term.  Commuting derivatives yields
\[
\begin{aligned}
\nabla^{p}\nabla_{i}\slashed{D}\psi &= \nabla_i\nabla^p\slashed{D}\psi +\sum_{r=0}^{p-1}\nabla^{r}Rm*\nabla^{p-1-r}\slashed{D}\psi\\
&=\nabla_i\slashed{D}\nabla^{p}\psi +\sum_{r=0}^{p}\nabla^{r}Rm*\nabla^{p-r}\psi.\\
\end{aligned}
\]
Thus, in total we have
\begin{equation}\label{eq: DpLxPsi}
\begin{aligned}
\nabla^{p}\mathcal{L}_{X}\psi &=\frac{1}{2} \sum_{i,j}e^{-2\phi}{\rm Re}\left(\langle \psi, \gamma^{[j}\nabla_{i]}\slashed{D}\nabla^{p}\psi \rangle\right) \gamma^{ij}\psi\\
&+ \sum_{\ell=1}^{p} \nabla^{\ell}(e^{-2\phi}\psi)*\nabla^{p-\ell+2}\psi*\psi+\sum_{r=1}^{p}\sum_{\ell=0}^{p-r}\nabla^{p-r-\ell}(e^{-2\phi}\psi) *\nabla^{\ell+2}\psi *\nabla^{r}\psi\\
&+\sum_{r=0}^{p}e^{-2\phi}\nabla^{r}Rm*\nabla^{p-r}\psi+\nabla^{p}\left(e^{-2\phi}\nabla \psi*\nabla\psi *\psi+\nabla e^{-2\phi}*\nabla\psi *\psi^2\right).
\end{aligned}
\end{equation}
Substituting this formula into~\eqref{eq: step1EvolDpPsi} yields the result.
\end{proof}

 Before beginning the estimate of $\frac{d}{dt}\int_{M}|\nabla^p\psi|^2\sqrt{g}$, let us first prove the following estimate
 
 \begin{lem}\label{lem: boundDpc(t)}
Suppose that $|\nabla \psi|^2, |\nabla^2\psi|, |H|^2, |\nabla H|,  |Rm|$ are bounded by $A$ along either of the flows~\eqref{eq: modFlowVar} or~\eqref{eq: modFlowFix} on $[0,\tau]$ and assume that $|\phi|\leq \hat{A}$ on $[0,\tau]$. Then the following estimates hold
\begin{enumerate}
\item[$(i)$] Along the flow~\eqref{eq: modFlowFix} with fixed $\phi$, we have
\[
\big|\int_{M}\langle \nabla^{p}(c(t)\psi),\nabla^{p}\psi\rangle\big| \leq \epsilon \int_{M} |\nabla^{p+1}\psi|^2 +C\epsilon^{-1} \int_{M}|\nabla^{p}H|^2+|\nabla^{p}\psi|^2 + |\nabla^{p+1} e^{2\phi}|^2+1
\]
for any $0<\epsilon<1$ and $C$ a constant depending only on $\dim M, p, A, \hat{A}$.
\item[$(ii)$] Along the flow~\eqref{eq: modFlowVar} with dynamical $\phi$ we have
\[
\big|\int_{M}\langle \nabla^{p}(c(t)\psi),\nabla^{p}\psi\rangle\big| \leq \epsilon \int_{M}|\nabla^{p+1}\psi|^2 + C\epsilon^{-1}\int_{M} |\nabla^p H|^2 + |\nabla^p \psi|^2 + |\nabla^p e^{2\phi}|^2 +1
\]
for any $0<\epsilon<1$ and $C$ a constant depending only on $\dim M, p, A, \hat{A}$
\end{enumerate}
\end{lem}
\begin{proof}
We first consider the case $(ii)$, when $\phi$ is variable.  In this case we have
\[
c(t)= e^{-2\phi}|\nabla^{H}\psi|^2 -2\sum_{a}e^{-2\phi}\langle h_{a}\psi, \psi \rangle(\nabla_{a}\phi +e^{-2\phi}\langle h_{a}\psi, \psi \rangle)
\]
We need to estimate terms of the form $\nabla^{r}c(t)*\nabla^{p-r}\psi *\nabla^p\psi$.  To this end, an easy, but somewhat tedious calculation using Corollary~\ref{cor: genHamInterp} and Lemma~\ref{lem: stupidInter} shows that
\[
\begin{aligned}
&\sum_{r=0}^{p}\sum_{j=0}^{r}\big|\int_{M}\nabla^{r-j}e^{-2\phi} *\nabla^j|\nabla^H\psi|^2 *\nabla^{p-r}\psi*\nabla^p\psi\big|\\
& \leq \epsilon \int_{M}|\nabla^{p+1}\psi|^2 + C\epsilon^{-1}\int_{M}|\nabla^pH|^2+|\nabla^p \psi|^2 +|\nabla^p e^{-2\phi}|^2 +1
\end{aligned}
\]
On the other hand, since $\phi$ is bounded from above and below on $[0,\tau]$ by assumption, it follows from Lemma~\ref{lem: stupidInter} that 
\[
 \int_{M} |\nabla^p e^{-2\phi}|^2 \leq C\int_{M} (|\nabla^p e^{2\phi}|^2 +1).
 \]
 A similar argument gives, for any $r\leq p$, the bound
 \[
 \big|\int_{M} \nabla^{r}\left(\sum_{a}(e^{-2\phi}\langle h_{a}\psi, \psi \rangle)^2\right)* \nabla^{p-r}\psi *\nabla^p\psi\big| \leq C\int_{M} |\nabla^p H|^2 + |\nabla^p \psi|^2 + |\nabla^p e^{2\phi}|^2 +1.
 \]
 It only remains to consider the integral
 \[
\sum_{r=0}^{p} \big|\int_{M} \nabla^{r}\left((\nabla_{a}e^{-2\phi})\langle h_{a}\psi, \psi \rangle\right)*\nabla^{p-r}\psi *\nabla^{p}\psi\big|.
\]
When $r<p$, the interpolation inequalities yield
\[
\sum_{r=0}^{p-1} \big|\int_{M} \nabla^{r}\left((\nabla_{a}e^{-2\phi})\langle h_{a}\psi, \psi \rangle\right)*\nabla^{p-r}\psi *\nabla^{p}\psi\big| \leq C\int_{M} |\nabla^p H|^2 + |\nabla^p \psi|^2 + |\nabla^p e^{2\phi}|^2 +1.
\]
When $r=p$ we integrate by parts to get
\[
\begin{aligned}
\big|\int_{M} \nabla^{p}\left(\nabla_{a}e^{-2\phi}\langle h_{a}\psi, \psi \rangle\right)\psi \nabla^{p}\psi\big| &\leq \big|\int_{M} \nabla^{p-1}\left(\nabla_{a}e^{-2\phi}\langle h_{a}\psi, \psi \rangle\right)\psi \nabla^{p+1}\psi\big|\\
&+C\int_{M} |\nabla^p H|^2 + |\nabla^p \psi|^2 + |\nabla^p e^{2\phi}|^2 +1
\end{aligned}
\]
and another application of interpolation yields the estimate.

We now consider case $(i)$, when $\phi$ is fixed.  In this case we have
\[
c(t)= e^{-2\phi}|\nabla^{H}\psi|^2-\frac{1}{2}e^{-2\phi}\Delta e^{2\phi}-e^{-2\phi}\nabla_{a}\langle h_{a}\psi, \psi \rangle)
\]
The first and last terms are estimate in an identical way to case $(ii)$, and so the only new term to bound is the Laplacian term.  We need to bound
\[
\sum_{r=0}^{p}\int_{M}\nabla^{p-r}(e^{-2\phi}\Delta e^{2\phi})*\nabla^r\psi*\nabla^p\psi.
\]
First we consider the contribution from $r\geq 1$.  The interpolation inequalities  Corollary~\ref{cor: genHamInterp} and Lemma~\ref{lem: stupidInter}  yield
\[
\big|\sum_{r=1}^{p}\int_{M}\nabla^{p-r}(e^{-2\phi}\Delta e^{2\phi})*\nabla^r\psi*\nabla^p\psi\big| \leq C\int_{M}|\nabla^{p}\psi|^2 +|\nabla^{p+1}e^{2\phi}|^2+1.
\]
For $r=0$ we integrate by parts to get
\[
\begin{aligned}
\int_{M}\nabla^{p}(e^{-2\phi}\Delta e^{2\phi})*\nabla\psi*\nabla^p\psi&=-\int_{M}\nabla^{p-1}(e^{-2\phi}\Delta e^{2\phi})*\nabla\psi*\nabla^p\psi\\
&-\int_{M}\nabla^{p-1}(e^{-2\phi}\Delta e^{2\phi})*\psi*\nabla^{p+1}\psi
\end{aligned}
\]
Now another application of interpolation yields
\[
\big|\int_{M}\nabla^{p}(e^{-2\phi}\Delta e^{2\phi})*\nabla\psi*\nabla^p\psi \big| \leq \epsilon \int_{M}|\nabla^{p+1}\psi|^2 + C\epsilon^{-1} \int_{M}|\nabla^{p+1}e^{2\phi}|^2 +1,
\]
and the result follows.
\end{proof}

Finally, we come to the estimate

\begin{prop}\label{prop: BndIntDpPsi}
Suppose that $|\nabla \psi|^2, |\nabla^2\psi|, |H|^2, |\nabla H|,  |Rm|$ are bounded by $A$ along either of the flows~\eqref{eq: modFlowVar} or~\eqref{eq: modFlowFix} on $[0,\tau]$ and assume that $|\phi| \leq \hat{A}$ on $[0,\tau]$.  Then
\[
\frac{d}{dt}\int_{M}|\nabla^{p}\psi|^2 \leq -\int_{M}|\nabla^{p+1}\psi|^2 + C_3\int_{M} |\nabla^p\psi|^2+ |\nabla^{p}H|^2 + |\nabla^{p-1}Rm|^2 +|\nabla^{p+1}e^{2\phi}|^2+1
\]
for a constant $C_3$ depending only on $A, \hat{A}, p, \dim M$.
\end{prop}
\begin{proof}
We first compute
\[
\frac{d}{dt}\int_{M}|\nabla^p\psi|^2 \sqrt{g} \leq C\int_{M}|\nabla^p \psi|^2 + 2\int_{M}  {\rm Re}\left(\langle\frac{d}{dt}\nabla^p \psi, \nabla^p \psi \rangle\right)
\]
Now by Proposition~\ref{prop: evolDpPsi} we have
\[
\begin{aligned}
\int_{M}{\rm Re}\left( \langle \frac{d}{dt}\nabla^p \psi, \nabla^p \psi \rangle\right)&= -\int_{M} |\nabla^{p+1}\psi|^2 + \int_{M} {\rm Re}\left(\langle\nabla^{p}(c(t)\psi), \nabla^{p}\psi \rangle\right)\\
&\quad + \frac{1}{16}\int_{M}e^{-2\phi}{\rm Re}\left(\langle\psi, \gamma^{[j}\nabla_{i]}\slashed{D}\nabla^{p}\psi \rangle \right){\rm Re}\left(\langle\gamma^{ij}\psi, \nabla^{p}\psi \rangle\right)\\
 &\quad +\sum_{i=0}^{p}\int_{M}\nabla^{p-i}Rm *\nabla^{i}\psi*\nabla^p\psi \\
 &\quad +\sum_{i=1}^{p}\int_{M}\nabla^{p-i}\psi*\nabla^i(e^{-2\phi}{\rm Hess}(e^{2\phi}) + e^{-2\phi}\cE +\cE')*\nabla^p\psi\\
&\quad+ \sum_{\ell=1}^{p} \int_{M}\nabla^{\ell}(e^{-2\phi}\psi)*\nabla^{p-\ell+2}\psi*\psi *\nabla^p\psi\\
&\quad +\sum_{r=1}^{p}\sum_{\ell=0}^{p-r}\nabla^{p-r-\ell}(e^{-2\phi}\psi) *\nabla^{\ell+2}\psi *\nabla^{r}\psi*\nabla^p\psi\\
&\quad +\sum_{r=0}^{p}\int_{M}e^{-2\phi}\nabla^{r}Rm*\nabla^{p-r}\psi*\nabla^p\psi\\
&\quad+\int_{M}\nabla^{p}\left(e^{-2\phi}\nabla \psi*\nabla\psi *\psi+\nabla e^{-2\phi}*\nabla\psi *\psi^2\right)*\nabla^p\psi\\
&=  -\int_{M} |\nabla^{p+1}\psi|^2 + (I)+ (II) +\cdots + (VIII)
 \end{aligned}
\]
We will explain how to estimate each term.  Term $(I)$ was bounded in Lemma~\ref{lem: boundDpc(t)}.  Term $(II)$ we bound using integration by parts.  First note that we can drop the anti-symmetrization in the indices $(i,j)$, thanks to the presence of the anti-symmetric symbol $\gamma^{ij}$.  Then integrating by parts yields
\[
\begin{aligned}
(II) &= \frac{1}{16}\int_{M}e^{-2\phi}{\rm Re}\left(\langle\psi, \nabla_{i}\gamma^{j}\slashed{D}\nabla^{p}\psi \rangle \right){\rm Re}\left(\langle\gamma^{ij}\psi, \nabla^{p}\psi \rangle\right)\\
&= -\frac{1}{16}  \int_{M}e^{-2\phi}{\rm Re}\left(\langle\nabla_{i}\psi, \gamma^{j}\slashed{D}\nabla^{p}\psi \rangle \right){\rm Re}\left(\langle\gamma^{ij}\psi, \nabla^{p}\psi \rangle\right)\\
&\quad-\frac{1}{16}\int_{M}e^{-2\phi}{\rm Re}\left(\langle\psi, \gamma^{j}\slashed{D}\nabla^{p}\psi \rangle \right){\rm Re}\left(\langle\gamma^{ij}\nabla_i\psi, \nabla^{p}\psi \rangle\right)\\
&\quad-\frac{1}{16}\int_{M}e^{-2\phi}{\rm Re}\left(\langle\psi, \gamma^{j}\slashed{D}\nabla^{p}\psi \rangle \right){\rm Re}\left(\langle\gamma^{ij}\psi, \nabla_i\nabla^{p}\psi \rangle\right)\\
&\quad -\frac{1}{16}\int_{M}(\nabla_ie^{-2\phi}){\rm Re}\left(\langle\psi, \gamma^{j}\slashed{D}\nabla^{p}\psi \rangle \right){\rm Re}\left(\langle\gamma^{ij}\psi, \nabla^{p}\psi \rangle\right)\\
&= -(IIa)-(IIb)-(IIc)-(IId)
\end{aligned}
\]
Now we have
\[
\big|(IIa) +(IIb)+(IId)| \leq \epsilon \int_{M}|\nabla^{p+1}\psi|^2 + C\epsilon^{-1}\int_{M}|\nabla^p\psi|^2
\]
for $C$ depending on $A, \hat{A}$.  For $(IIc)$, since $\gamma^{ij} = \gamma^i\gamma^j-\delta_{ij}$ we have
\[
{\rm Re}\left(\langle\gamma^{ij}\psi, \nabla_i\nabla^{p}\psi \rangle\right) = {\rm Re}\left(\langle\psi, \gamma^{j}\slashed{D}\nabla^{p}\psi \rangle\right) -{\rm Re}\left(\langle\psi, \nabla_j\nabla^{p}\psi \rangle\right)\delta_{ij}
\]
and so
\[
\begin{aligned}
(II) &\leq -\frac{1}{16}\int_{M}e^{-2\phi}\big|{\rm Re}\left(\langle\psi, \gamma^{j}\slashed{D}\nabla^{p}\psi \rangle \right)\big|^2\\
&\quad +\frac{1}{16}\int_{M}e^{-2\phi}{\rm Re}\left(\langle\psi, \gamma^{j}\slashed{D}\nabla^{p}\psi \rangle \right){\rm Re}\left(\langle\psi, \nabla_j\nabla^{p}\psi \rangle\right)\\
&\quad +\epsilon \int_{M}|\nabla^{p+1}\psi|^2 + C\epsilon^{-1}\int_{M}|\nabla^p\psi|^2
\end{aligned}
\]
On the other hand, we have
\[
{\rm Re}\left(\langle\psi, \nabla_j\nabla^{p}\psi \rangle\right)= \nabla_j\nabla^{p}e^{2\phi} + \sum_{\ell=1}^{p+1}\nabla^{\ell} \psi*\nabla^{p+1-\ell}\psi.
\]
Now writing
\[
\sum_{\ell=1}^{p+1}\nabla^{\ell} \psi*\nabla^{p+1-\ell}\psi*\nabla^{p+1}\psi*\psi  = \sum_{\ell=1}^{p}\nabla^{\ell-1} (\nabla \psi)*\nabla^{p+1-\ell}\psi*\nabla^{p}(\nabla\psi)*\psi 
\]
it follows from Corollary~\ref{cor: genHamInterp} that
\begin{equation}\label{eq: Psi2Bnd}
(II) \leq \epsilon \int_{M}|\nabla^{p+1}\psi|^2 + C\epsilon^{-1}\int_{M}|\nabla^p\psi|^2 + |\nabla^{p+1}e^{2\phi}|^2
\end{equation}
for $C$ depending only on $A, \dim M, p, \hat{A}$.

For term $(III)$ we first bound the $i=0$ term.  By integration by parts we have
\[
\int_{M}\nabla^pRm*\psi*\nabla^p\psi = -\int_{M}\nabla^{p-1}Rm *\nabla \psi *\nabla^p\psi + \nabla^{p-1}Rm*\psi*\nabla^{p+1}\psi
\]
and hence
\[
\big|\int_{M}\nabla^pRm*\psi*\nabla^p\psi\big| \leq \epsilon \int_{M}|\nabla^{p+1}\psi|^2+ C\epsilon^{-1}\int_{M} |\nabla^{p-1}Rm|^2+|\nabla^{p}\psi|^2
\]
For the remaining terms we write
\[
\sum_{i=1}^{p}\int_{M}\nabla^{p-i}Rm *\nabla^{i}\psi*\nabla^p\psi =\sum_{i=1}^{p}\int_{M}\nabla^{p-i}Rm *\nabla^{i-1}(\nabla \psi)*\nabla^{p-1}(\nabla \psi)
\]
and apply Corollary~\ref{cor: genHamInterp} with $k=p-1$ to obtain
\begin{equation}\label{eq: Psi3Bnd}
|(III)| \leq \epsilon \int_{M}|\nabla^{p+1}\psi|^2+ C\epsilon^{-1}\int_{M} |\nabla^{p-1}Rm|^2+|\nabla^{p}\psi|^2.
\end{equation}

For term $(IV)$ we treat the $i=p$ term separately.  We have
\[
\begin{aligned}
&\int_{M}\psi*\nabla^p(e^{-2\phi}{\rm Hess}(e^{2\phi}) + e^{-2\phi}\cE +\cE')*\nabla^p\psi\\
&= -\int_{M}\nabla \psi*\nabla^{p-1}(e^{-2\phi}{\rm Hess}(e^{2\phi}) + e^{-2\phi}\cE +\cE')*\nabla^p\psi\\
&\quad -\int_{M} \psi*\nabla^{p-1}(e^{-2\phi}{\rm Hess}(e^{2\phi}) + e^{-2\phi}\cE +\cE')*\nabla^{p+1}\psi
\end{aligned}
\]
The first term can be easily bounded by Corollary~\ref{cor: genHamInterp} and Lemma~\ref{lem: stupidInter} using the definition of $\cE,\cE'$ in ~\eqref{eq: errorDefn}
\[
\begin{aligned}
&\big|\int_{M}\nabla \psi*\nabla^{p-1}(e^{-2\phi}{\rm Hess}(e^{2\phi}) + e^{-2\phi}\cE +\cE')*\nabla^p\psi\big|\\
 & \leq C\int_{M}|\nabla^{p+1}e^{2\phi}|^2 +|\nabla^{p}H|^2 + |\nabla^{p}\psi|^2+1
\end{aligned}
\]
Similarly, the second term can be bounded as
\[
\begin{aligned}
&\big|\int_{M} \psi*\nabla^{p-1}(e^{-2\phi}{\rm Hess}(e^{2\phi}) + e^{-2\phi}\cE +\cE')*\nabla^{p+1}\psi\big|\\
&\leq \epsilon \int_{M}|\nabla^{p+1}\psi|^2+ C\epsilon^{-1} \int_{M}|\nabla^{p+1}e^{2\phi}|^2 +|\nabla^{p}H|^2 + |\nabla^{p}\psi|^2+1
\end{aligned}
\]
and so in total we have
\begin{equation}\label{eq: Psi4Bnd}
|(IV)| \leq \epsilon \int_{M}|\nabla^{p+1}\psi|^2+ C\epsilon^{-1} \int_{M}|\nabla^{p+1}e^{2\phi}|^2 +|\nabla^{p}H|^2 + |\nabla^{p}\psi|^2+1
\end{equation}

Terms $(V)$ and $(VI)$ can be treated in the same way, by a combination of integration by parts and interpolation.  The reader can check that the following bounds hold
\begin{equation}\label{eq: Psi56Bnd}
|(V)+(VI)| \leq \epsilon \int_{M}|\nabla^{p+1}\psi|^2+ C\epsilon^{-1}\int_{M} |\nabla^{p}\psi|^2+1.
\end{equation}
For term $(VII)$ we treat the $r=p$ term separately.  Integration by parts yields
\[
\int_{M}e^{-2\phi}\nabla^{p}Rm*\psi*\nabla^p\psi=-\int_{M}\nabla^{p-1}Rm*\nabla (e^{-2\phi}\psi)*\nabla^p\psi + e^{-2\phi}\nabla^{p-1}Rm*\psi*\nabla^{p+1}\psi
\]
which yields the estimate
\[
\big|\int_{M}e^{-2\phi}\nabla^{p}Rm*\psi*\nabla^p\psi\big| \leq \epsilon \int_{M}|\nabla^{p+1}\psi|^2 + C\epsilon^{-1}\int_{M}|\nabla^{p-1}Rm|^2 + |\nabla^{p}\psi|^2. 
\]
For the remaining terms we write
\[
\sum_{r=0}^{p-1}\int_{M}e^{-2\phi}\nabla^{r}Rm*\nabla^{p-r}\psi*\nabla^p\psi=\sum_{r=0}^{p}\int_{M}e^{-2\phi}\nabla^{r}Rm*\nabla^{p-r-1}(\nabla\psi)*\nabla^{p-1}(\nabla\psi)
\]
and apply Corollary~\ref{cor: genHamInterp} with $k=p-1$ to obtain
\begin{equation}\label{eq: Psi7Bnd}
|(VII)| \leq \epsilon \int_{M}|\nabla^{p+1}\psi|^2 + C\epsilon^{-1}\int_{M}|\nabla^{p-1}Rm|^2 + |\nabla^{p}\psi|^2
\end{equation}
 Finally, we come to term $(VIII)$.  Integration by parts gives
 \[
 (VIII) = -\int_{M}\nabla^{p-1}\left(e^{-2\phi}\nabla \psi*\nabla\psi *\psi+\nabla e^{-2\phi}*\nabla\psi *\psi^2\right)*\nabla^p(\nabla \psi)
 \]
 and now Corollary~\ref{cor: genHamInterp} and Lemma~\ref{lem: stupidInter} yield the bound
 \begin{equation}\label{eq: Psi8Bnd}
 |(VIII)| \leq \epsilon \int_{M}|\nabla^{p+1}\psi|^2 + C\epsilon^{-1} |\nabla^{p}\psi|^2 + |\nabla^{p+1}e^{2\phi}|^2 +1
 \end{equation}
 Combining the estimates from Lemma~\ref{lem: boundDpc(t)} with the bounds~\eqref{eq: Psi2Bnd}~\eqref{eq: Psi3Bnd} \eqref{eq: Psi4Bnd}~\eqref{eq: Psi56Bnd}~\eqref{eq: Psi7Bnd}~\eqref{eq: Psi8Bnd} and choosing $\epsilon$ sufficiently small yields the result.
 \end{proof}

Before proceeding, let us prove the smoothing estimates in the case of fixed $\phi$.
\begin{thm}\label{thm: ShiEstConstPhi}
Suppose $(g(t), \psi(t), H(t))$ flow along the spinor flow with flux~\eqref{flow1} with fixed $\phi$ on the interval $[0,\frac{\tau}{A}]$. Suppose in addition that $|\nabla \psi|^2, |\nabla^2\psi|, |H|^2, |\nabla H|,  |Rm|$ are bounded by $A$ on $[0,\frac{\tau}{A}]$.  Suppose in addition that there are constants $D_p$ so that
\[
\frac{1}{{\rm Vol}(M,g(t))}\int_{M}|\nabla^{p}e^{2\phi}|_{g(t)}^2\sqrt{g(t)} \leq D_pA^{p}.
\]
for all $t\in [0,\frac{\tau}{A})$.  Then we have
\[
\begin{aligned}
\frac{1}{{\rm Vol}(M,g(t))} \int_{M} |\nabla^{p}Rm(t)|_{g(t)}^2 \sqrt{g(t)} &\leq \frac{CA^2}{t^{p}}\\
\frac{1}{{\rm Vol}(M,g(t))}\int_{M} |\nabla^{p+2}\psi(t)|_{g(t)}^2 \sqrt{g(t)}&\leq \frac{CA^2}{t^{p}}\\
\frac{1}{{\rm Vol}(M,g(t))}\int_{M} |\nabla^{p+1}H(t)|_{g(t)}^2 \sqrt{g(t)} &\leq \frac{CA^2}{t^{p}}
\end{aligned}
\]
for a constant $C$ depending only on $\dim M, A, \|\phi\|_{L^{\infty}}, \tau, p, \max_{0\leq j \leq p+3} D_j$.
\end{thm}
\begin{proof}
Since all the norms are diffeomorphism invariant, we can assume that $(g(t), \psi(t), H(t))$ evolve along the gauge modified spinor flow~\eqref{eq: modFlowFix}.  Furthermore, by rescaling we may assume that $A=1$.   We consider the quantity
\[
F_p(t) = \int_{M} |\nabla^{p+2}\psi|^2 + B_1|\nabla^{p+1}H|^2 + B_2|\nabla^pRm|^2 \sqrt{g}
\]
By Propositions~\ref{prop: BndIntDpRm}, \ref{prop: BndIntDpH} and, \ref{prop: BndIntDpPsi} we have
\[
\begin{aligned}
\frac{d}{dt}F_{p} &\leq -\int_{M}|\nabla^{p+3}\psi|^2 + C_3\int_{M}|\nabla^{p+2}\psi|^2 +|\nabla^{p+2}H|^2 + |\nabla^{p+1}Rm|^2+C_3D_p\\
&-B_1\int_{M}|\nabla^{p+2}H|^2 + B_1C_2\int_{M}\int_{M}|\nabla^pRm|^2 + |\nabla^{p+1}H|^2 +|\nabla^{p+2}\psi|^2 + 1\\
&-\frac{B_2}{16} \int_{M}|\nabla^{p+1}Rm|^2 + B_2C_1\int_{M} |\nabla^{p+2}H|^2\\
&+B_2\int_{M}|\nabla^pRm|^2 + |\nabla^{p+1}H|^2 +|\nabla^{p+2}\psi|^2 + 1
\end{aligned}
\]
Now choose $B_{2}$ such that $C_3-\frac{B_{2}}{16} =-1$, and choose $B_{1}$ such that $C_3+B_2C_1-B_1=-1$.  The constants $B_2, B_1$ depend only on $\dim M, \|\phi\|_{L^{\infty}}, p$.  Then we have
\[
\frac{dF_p}{dt} \leq -\int_{M}|\nabla^{p+3}\psi|^2 +|\nabla^{p+2}H|^2+|\nabla^{p+1}Rm|^2 + C_p (F_{p}+1) + C_pD_p.
\]
or in other words, we have
\[
\frac{dF_p}{dt} \leq  -\delta_pF_{p+1} + C_p F_p + C_p {\rm Vol}(t) + C_pD_p.
\]
for $\delta_p$ a universal constant.  Since ${\rm Vol}(t)$ has a uniform upper bound on $[0,\tau)$ thanks to Lemma~\ref{lem: evoOfVolForm}, a standard argument (see, e.g. \cite{HeWang}) shows that
\[
F_{p}(t) \leq \frac{C}{t^p}
\]
for $C$ depending only on $p, \max_{0\leq j \leq p}D_j, \|\phi\|_{L^{\infty}}$ and an upper bound for $\tau$.  Rescaling yields the desired dependence on $A$.
\end{proof}

 We now return to the case of dynamical $\phi$.  The only thing that remains is to work out the evolution $\nabla^p e^{2\phi}$. 
 
 \begin{prop}\label{prop: DpPhiEvo}
Along the gauge modified spinor flow with flux with dynamical $\phi$~\eqref{eq: modFlowFix}  we have
 \[
 \begin{aligned}
 \frac{d}{dt}\nabla^{p}e^{2\phi} &=  \Delta \nabla^{p}e^{2\phi} + \sum_{r=1}^{p}\nabla^{p-r}Rm*\nabla^re^{2\phi}\\
 &\quad+\sum_{r=1}^{p-1}\nabla^{p-r}e^{2\phi}*\nabla^{r}\left(e^{-2\phi}{\rm Hess}(e^{2\phi}) + e^{-2\phi}\cE + \cE'\right)\\
 &\quad+\nabla^{p}\left(e^{2\phi}\nabla(e^{-2\phi}H*\psi*\psi) + e^{-2\phi}\psi*\nabla\psi*\nabla e^{2\phi}+e^{-2\phi}H^2*\psi^4\right)
 \end{aligned}
 \]
 where $\cE,\cE'$ are defined in~\eqref{eq: errorDefn}.
 \end{prop}
 \begin{proof}
 Recall that the evolution of $e^{2\phi}$ is given by
 \[
 \begin{aligned}
 \frac{d}{dt}e^{2\phi} &= \Delta e^{2\phi} + 2e^{2\phi}\nabla_{a}(e^{-2\phi}\langle h_a,\psi, \psi \rangle) +\frac{1}{4}\nabla_{X}e^{2\phi}-4\sum_{a}e^{-2\phi}|\langle h_{a}\psi, \psi \rangle|^2
 \end{aligned}
 \]
From the definition of $X$ in ~\eqref{eq: HWVF} we can write this as 
 \[
 \begin{aligned}
 \frac{d}{dt}e^{2\phi} &= \Delta e^{2\phi} + e^{2\phi}\nabla(e^{-2\phi}H*\psi*\psi) + e^{-2\phi}\psi*\nabla\psi*\nabla e^{2\phi}+e^{-2\phi}H^2*\psi^4
 \end{aligned}
 \]
Therefore
\[
\begin{aligned}
\frac{d}{dt}\nabla^{p}e^{2\phi} &= \Delta \nabla^{p}e^{2\phi} + \sum_{r=1}^{p}\nabla^{p-r}Rm*\nabla^re^{2\phi}\\
&+ \nabla^{p}\left(e^{2\phi}\nabla(e^{-2\phi}H*\psi*\psi) + e^{-2\phi}\psi*\nabla\psi*\nabla e^{2\phi}+e^{-2\phi}H^2*\psi^4\right)
\end{aligned}
\]
Finally, we have
\[
\frac{d}{dt}\nabla^{p}e^{2\phi} = \sum_{r=1}^{p-1}\nabla^r(\dot{g}) *\nabla^{p-r}e^{2\phi} + \nabla^{p}\frac{d}{dt}e^{2\phi}
\]
and hence the desired formula follows from~\eqref{eq: flowgError}.
\end{proof}

We now have

\begin{prop}\label{prop: IntDpPhiBnd}
Suppose that along the gauge-modified spinor flow with flux with dynamical $\phi$ ~\eqref{eq: modFlowVar} we have that $|\nabla \psi|^2, |\nabla^2\psi|, |H|^2, |\nabla H|,  |Rm|$ are bounded by $A$ on $[0,\tau]$ and $\|\phi\|_{L^{\infty}}$ is bounded by $\hat{A}$ on $[0,\tau]$.  Then we have
\[
\begin{aligned}
\frac{d}{dt}\int_{M} |\nabla^{p}e^{2\phi}|^2 \sqrt{g} &\leq -\int_{M}|\nabla^{p+1}e^{2\phi}|^2 \\
&\quad + C_4\int_{M}|\nabla^{p}e^{2\phi}|^2 + |\nabla^{p-1}Rm|^2 + |\nabla^{p}H|^2 +|\nabla^{p}\psi|^2+1
\end{aligned}
\]
where $C_4$ depends only on $\dim M, p, A, \hat{A}$.
\end{prop}
\begin{proof}
From Proposition~\ref{prop: DpPhiEvo} together with the assumed bounds we have
\[
\begin{aligned}
\frac{d}{dt}\int_{M} |\nabla^{p}e^{2\phi}|^2  &\leq -\int_{M}|\nabla^{p+1}e^{2\phi}|^2  + C\int_{M}|\nabla^{p}e^{2\phi}|^2 \\
&\quad+ \sum_{r=1}^{p}\int_{M}\nabla^{p-r}Rm*\nabla^re^{2\phi}*\nabla^{p}e^{2\phi}\\
 &\quad+\sum_{r=1}^{p-1}\int_{M}\nabla^{p-r}e^{2\phi}*\nabla^{r}\left(e^{-2\phi}{\rm Hess}(e^{2\phi}) + e^{-2\phi}\cE + \cE'\right)*\nabla^{p}e^{2\phi}\\
 &\quad+\int_{M}\nabla^{p}\left(e^{2\phi}\nabla(e^{-2\phi}H*\psi*\psi) + e^{-2\phi}\psi*\nabla\psi*\nabla e^{2\phi}\right)*\nabla^{p}e^{2\phi}\\
 &\quad+\int_{M}\nabla^{p}\left(e^{-2\phi}H^2*\psi^4\right)*\nabla^{p}e^{2\phi}\\
 &=-\int_{M}|\nabla^{p+1}e^{2\phi}|^2  + C\int_{M}|\nabla^{p}e^{2\phi}|^2 +(I)+(II)+(III)+(IV).
 \end{aligned}
 \]
 Each term is estimated using the interpolation inequalities and integration by parts.  For $(I)$ we write
 \[
 \begin{aligned}
  \int_{M}\sum_{r=1}^{p}\nabla^{p-r}Rm*\nabla^re^{2\phi}*\nabla^{p}e^{2\phi} =  \int_{M}\sum_{r=1}^{p}\nabla^{p-r}Rm*\nabla^{r-1}(\nabla e^{2\phi})*\nabla^{p-1}(\nabla e^{2\phi}).
  \end{aligned}
  \]
  Applying the interpolation inequality with $k=p-1$ yields
  \begin{equation}\label{eq: phi1Bnd}
  (I) \leq C \int_{M} |\nabla^{p-1}Rm|^2 + |\nabla^{p}e^{2\phi}|^2
  \end{equation}
  For term $(II)$, write
  \[
  \begin{aligned}
  (II)=(IIa)+(IIb) &:= \int_{M}\sum_{r=1}^{p-1}\nabla^{p-r}e^{2\phi}*\nabla^{r}\left(e^{-2\phi}{\rm Hess}(e^{2\phi})\right)*\nabla^{p}e^{2\phi}\\
  &\quad +\int_{M}\sum_{r=1}^{p-1}\nabla^{p-r}e^{2\phi}*\nabla^{r}\left(e^{-2\phi}\cE + \cE'\right)*\nabla^{p}e^{2\phi}
  \end{aligned}
 \]
 For $(IIa)$, when $r< p-1$ we write
 \[
  \begin{aligned}
  &\sum_{r=1}^{p-2}\nabla^{p-r}e^{2\phi}*\nabla^{r}\left(e^{-2\phi}{\rm Hess}(e^{2\phi})\right)*\nabla^{p}e^{2\phi}\\
  &= \sum_{r=1}^{p-2}\nabla^{p-r-2}(\nabla^2e^{2\phi})*\nabla^{r}\left(e^{-2\phi}{\rm Hess}(e^{2\phi})\right)*\nabla^{p-2}(\nabla^2e^{2\phi}).
  \end{aligned}
  \]
  Applying the interpolation inequalities Corollary~\ref{cor: genHamInterp} with $k=p-2$ and Lemma~\ref{lem: stupidInter} yields bound
 \[
\big|\sum_{r=1}^{p-2}\nabla^{p-r}e^{2\phi}*\nabla^{r}\left(e^{-2\phi}{\rm Hess}(e^{2\phi})\right)*\nabla^{p}e^{2\phi}\big| \leq C\int_{M}|\nabla^{p}e^{2\phi}|^2 + 1
 \]
 For the remaining term with $r=p-1$ we integrate by parts once to get
 \[
 \begin{aligned}
& \int_{M}\nabla e^{2\phi}*\nabla^{p-1}\left(e^{-2\phi}{\rm Hess}(e^{2\phi})\right)*\nabla^{p}e^{2\phi}\\
&= - \int_{M}\nabla^2 e^{2\phi}*\nabla^{p-2}\left(e^{-2\phi}{\rm Hess}(e^{2\phi})\right)*\nabla^{p}e^{2\phi}\\
 &\quad- \int_{M}\nabla e^{2\phi}*\nabla^{p-2}\left(e^{-2\phi}{\rm Hess}(e^{2\phi})\right)*\nabla^{p+1}e^{2\phi}
 \end{aligned}
 \]
 Each term is easily bounded by interpolation to yield
 \[
 |(IIa)| \leq \epsilon \int_{M} |\nabla^{p+1}e^{2\phi}|^2 + C\epsilon^{-1}\int_{M}|\nabla^p e^{2\phi}|^2 +1
 \]
 For $(IIb)$, since $r\leq p-1$ and $\cE, \cE'$ depend only on $\psi, \nabla \psi, H, \nabla H, e^{2\phi}, \nabla e^{2\phi}$ (see equation~\eqref{eq: errorDefn}) an application of Corollary~\ref{cor: genHamInterp} with $k=p-1$, together with Lemma~\ref{lem: stupidInter} yields the bound
 \[
 |(IIb)| \leq C\int_{M} |\nabla^{p}\psi|^2 + |\nabla^{p}H|^2 +|\nabla^{p}e^{2\phi}|^2 +1
 \]
 Thus, in total we have
 \begin{equation}\label{eq: phi2Bnd}
 |(II)| \leq  \epsilon \int_{M}|\nabla^{p+1}e^{2\phi}|^2 + C\epsilon^{-1}\int_{M}|\nabla^{p}\psi|^2 + |\nabla^{p}H|^2 +|\nabla^{p}e^{2\phi}|^2 +1
 \end{equation}
 for any $0<\epsilon<1$.
 
 For term $(III)$, we integrate by parts once to get
 \[
 \begin{aligned}
(III)&=\\
&=- \int_{M}\nabla^{p-1}\left(e^{2\phi}\nabla(e^{-2\phi}H*\psi*\psi) + e^{-2\phi}\psi*\nabla\psi*\nabla e^{2\phi}\right)*\nabla^{p}(\nabla e^{2\phi})\\
\end{aligned}
\]
Applying Corollary~\ref{cor: genHamInterp} with $k=p$ together with Lemma~\ref{lem: stupidInter} yields
\[
|(IIIa)| \leq \epsilon \int_{M}|\nabla^{p+1}e^{2\phi}|^2 + C\epsilon^{-1}\int_{M} |\nabla^pH|^2+ |\nabla^{p}\psi|^2 + |\nabla^{p}e^{2\phi}|^2 +1.
\]
For $(IV)$ a direct application of interpolation yields the bound
\[
|(IV)| \leq C\int_{M} |\nabla^{p}e^{2\phi}|^2 + |\nabla^{p}H|^2+|\nabla^p\psi|^2 +1
\]
Thus, in total we have
\begin{equation}\label{eq: phi34Bnd}
|(III)+ (IV)| \leq \epsilon \int_{M}|\nabla^{p+1}e^{2\phi}|^2 + C\epsilon^{-1}\int_{M} |\nabla^pH|^2+ |\nabla^{p}\psi|^2 + |\nabla^{p}e^{2\phi}|^2 +1.
 \end{equation}
 Summing up~\eqref{eq: phi1Bnd}~\eqref{eq: phi2Bnd} and~\eqref{eq: phi34Bnd} yields the result.
 \end{proof}

We can now prove

\begin{thm}\label{thm: ShiEstVarPhi}
Suppose $(g(t), \psi(t), H(t), \phi(t))$ evolve along the spinor flow with flux~\eqref{flow} on the interval $[0,\frac{\tau}{A}]$. Suppose in addition that\\ $|\nabla \psi|^2, |\nabla^2\psi|, |H|^2, |\nabla H|,  |Rm|$ are bounded by $A$  and $\|\phi\|_{L^{\infty}} \leq \hat{A}$ on $[0,\frac{\tau}{A}]$.  Then we have
\[
\begin{aligned}
 \frac{1}{{\rm Vol}(M,g(t))}\int_{M} |\nabla^{p}Rm(t)|_{g(t)}^2\sqrt{g(t)} &\leq \frac{CA^2}{t^{p}}\\
 \frac{1}{{\rm Vol}(M,g(t))} \int_{M} |\nabla^{p+2}\psi(t)|_{g(t)}^2 \sqrt{g(t)} &\leq \frac{CA^2}{t^{p}}\\
 \frac{1}{{\rm Vol}(M,g(t))}\int_{M} |\nabla^{p+2}e^{2\phi(t)}|_{g(t)}^2 \sqrt{g(t)} &\leq \frac{CA^2}{t^{p}}\\
  \frac{1}{{\rm Vol}(M, g(t))}\int_{M} |\nabla^{p+1}H(t)|_{g(t)}^2\sqrt{g(t)} &\leq \frac{CA^2}{t^{p}}
\end{aligned}
\]
for a constant $C$ depending only on $\dim M, p, \hat{A}$ and an upper bound for $\tau$.
\end{thm}
\begin{proof}
As in the proof of Theorem~\ref{thm: ShiEstConstPhi} the norms are all diffeomorphism invariant and hence we can assume that $(g(t), \psi(t), H(t), \phi(t))$ evolve by the gauge-modified spinor flow~\eqref{eq: modFlowVar}.  Furthermore, by rescaling we may assuming $A=1$.  Consider the function
\[
F_p = \int_{M} |\nabla^{p+2}\psi|^2  + B_1|\nabla^pRm|^2 \sqrt{g} + B_2|\nabla^{p+1}H|^2 +  B_{4}\int_{M}|\nabla^{p+2}e^{2\phi}|^2
\]
Then by Propositions~\ref{prop: BndIntDpRm}, \ref{prop: BndIntDpH}~\ref{prop: BndIntDpPsi} and~\ref{prop: IntDpPhiBnd} we have
\[
\begin{aligned}
\frac{dF_{p}}{dt} &\leq -\int_{M}|\nabla^{p+3}\psi|^2 +C_3\left(\int_{M}|\nabla^{p+2}H|^2 + \int_{M}|\nabla^{p+1}Rm|^2 + |\nabla^{p+3}e^{2\phi}|^2\right)\\
&\quad - \frac{B_1}{16}\int_{M}|\nabla^{p+1}Rm|^2 + B_1C_1\int_{M}|\nabla^{p+2}H|^2\\
&\quad -B_2\int_{M}|\nabla^{p+2}H|^2\\
&\quad -B_{4}\int_{M}|\nabla^{p+3}e^{2\phi}|^2 +C_4B_4\int_{M}|\nabla^{p+1}Rm|^2 + |\nabla^{p+2}H|^2\\
&\quad+ D\left(\int_{M}|\nabla^{p+2}\psi|^2  + |\nabla^pRm|^2 + |\nabla^{p+1}H|^2 +  \int_{M}|\nabla^{p+2}e^{2\phi}|^2+1\right)
\end{aligned}
\]
where $D=C_3+B_1C_1+B_2C_2+B_4C_4$.  We now choose $B_1, B_2, B_4 \geq 1$ by inspecting the coefficients of the higher order derivatives.  We have
\begin{itemize}
\item The coefficient of $|\nabla^{p+2}H|^2$ is
\[
B_1C_1+C_3+B_4C_4-B_2
\]
\item The coefficient of $|\nabla^{p+1}Rm|^2$ is
\[
C_3+C_4B_4 -\frac{B_1}{16}
\]
\item The coefficient of $|\nabla^{p+3}e^{2\phi}|^2$ is
\[
C_3-B_4
\]
\item The coefficient of $|\nabla^{p+3}\psi|^2$ is $-1$.
\end{itemize}
Choose $B_4=C_3+1 \gg 1$.  Next choose $B_1 = 16(C_3+C_4B_4+1) \gg 1$.  Finally choose $B_{2}=B_1C_1+C_3+B_4C_4+1 \gg 1$.  Note from the dependence of the $C_i$ it follows that the $B_i$ depend only on $\dim M, p, \hat{A}$.  Then we have
\[
\frac{dF_p}{dt} \leq -\delta_p F_{p+1} +   C_p(F_p + {\rm Vol}(M, g(t)))
\]
Now, since $F_0$ is bounded, and the volume is bounded from above by Lemma~\ref{lem: easyVolBnd} a standard argument (see, e.g. \cite{HeWang}) shows that
\[
F_{p}(t) \leq \frac{C}{t^p}
\]
for $C$ depending only on $\dim M, p, \hat{A}$ and an upper bound for $\tau$.  Rescaling yields the desired bounds.
\end{proof}

\begin{rk}\label{rk: supNormShi}
We remark that the integral estimates obtained in Theorems~\ref{thm: ShiEstVarPhi} and Theorem~\ref{thm: ShiEstConstPhi} imply estimates in sup norm; this follows from an argument of Hamitlon; see \cite[Lemma 14.3]{Ham}.  Indeed, in the setting of Theorem~\ref{thm: ShiEstVarPhi} the bounds on $|\nabla \psi|^2, |\nabla^2\psi|, |H|^2, |\nabla H|,  |Rm|$  and $\|\phi\|_{L^{\infty}}$ imply that $g(t)$ is uniformly equivalent to $g(0)$, and hence the constants appearing in the Sobolev imbedding theorem for functions on $(M,g(t))$ are uniformly controlled depending only on $A,\hat{A}$ and an upper bound for $\tau$.  Now the interpolation inequality Lemma~\ref{lem: HamInterp} together with Kato's inequality implies bounds
\[
 |\nabla^{p}Rm(t)|_{g(t)}^2+|\nabla^{p+2}\psi(t)|_{g(t)}^2+ |\nabla^{p+2}e^{2\phi(t)}|_{g(t)}^2+|\nabla^{p+1}H(t)|_{g(t)}^2 \leq \frac{CA^2}{t^{p}}\\
\]
for $C$ depending only on $\dim M, p, \hat{A}$ and an upper bound for $\tau$.
\end{rk}

\section{Bounding the Riemann curvature tensor}\label{sec: RmBd}
In this section we explain how bounds on the fields $\psi, H, \phi$ yield bounds on the Riemann curvature.  The argument here builds on ideas of Kotschwar, Munteanu, and Wang \cite{KMW} and adaptions to the the spinor flow by He and Wang \cite{HeWang}.    First, consider the evolution of the Riemann curvature, as stated in Lemma~\ref{lem: evolRm}.  We have
\begin{equation}\label{eq: evoNormSqRm}
\begin{aligned}
\left(\frac{d}{dt} -\frac{1}{16}\Delta\right) |Rm|^2 &= -\frac{1}{8}|\nabla Rm|^2 + Rm*(\nabla e^{-2\phi})*\nabla^{3}e^{2\phi}\\
&\quad + Rm*(\nabla^2e^{-2\phi})*(\nabla^{2}e^{2\phi}) +Rm*(\nabla^2(e^{-2\phi}\cE) +\nabla^2\cE')\\
&\quad +Rm^3 + Rm^2*(e^{-2\phi}{\rm Hess}(e^{2\phi}) + e^{-2\phi}\cE + \cE')
\end{aligned}
\end{equation}
where $\cE, \cE'$ are defined in~\eqref{eq: errorDefn}.  Suppose $\psi, \nabla \psi, \nabla^2\psi, H, \nabla H, \nabla^2H$ are bounded and $\|\phi\|_{L^{\infty}}$ is bounded.  Suppose in addition that we have a bound $\int |Rm|^{p} dVol < C$ for some large $p$; in fact any $p>\frac{n}{2}$ will do.  Then we have
\begin{equation}\label{eq: evol|Rm|^2withBnds}
\begin{aligned}
\left(\frac{d}{dt}-\frac{1}{16}\Delta\right)|Rm|^2 &\leq   -\frac{1}{8}|\nabla Rm|^2 +C(|Rm|+|Rm|^2+ |Rm|^3) \\
&\quad +Rm*\nabla(\nabla(e^{-2\phi}\cE) +\nabla\cE' + \nabla e^{-2\phi}*\nabla^2e^{2\phi})
\end{aligned}
\end{equation}
If we set $u= |Rm|^2$ and $f=|Rm|$, then we have
\begin{equation}\label{eq: evolIntup}
\begin{aligned}
\frac{d}{dt} \int_{M}u^p \sqrt{g}dx &\leq C\int_{M}u^p + p\int_{M}u^{p-1} \left(\frac{1}{16}\Delta u - \frac{1}{8}|\nabla Rm|^2\right)\\
&\quad+p\int_{M}u^{p-1}\left(C(u+1) +uf\right)\\
&\quad+p\int_{M} Rm*\nabla\left(\nabla(e^{-2\phi}\cE) +\nabla \cE'+\nabla e^{-2\phi}*\nabla^2e^{2\phi}\right) \sqrt{g}dx
\end{aligned}
\end{equation}
From the forms of $\cE, \cE'$ the terms $\nabla\left(\nabla(e^{-2\phi}\cE) +\nabla \cE'+\nabla e^{-2\phi}*\nabla^2e^{2\phi}\right)$ are not under control.  However, $\nabla \cE, \nabla \cE', \nabla e^{-2\phi}*\nabla^2e^{2\phi}$ are under control and hence an integration by parts is needed.  First, we bound
\[
\frac{1}{16}\int_{M}u^{p-1} \Delta u \leq - \frac{(p-1)}{4(p+1)^2}\int_{M}|\nabla u^{\frac{p+1}{2}}|^2 
\]
We can similarly estimate
\[
\begin{aligned}
&p\int_{M}u^{p-1}Rm*\nabla\left(\nabla(e^{-2\phi}\cE) +\nabla \cE'\right) \sqrt{g}dx\\
& = -p(p-1)\int_{M}u^{p-2}\nabla u*Rm*\left(\nabla(e^{-2\phi}\cE) +\nabla \cE'\right) \sqrt{g}dx\\
&\quad-p\int_{M}u^{p-1}\nabla Rm*\left(\nabla(e^{-2\phi}\cE) +\nabla \cE'\right) \sqrt{g}dx\\
&=:(I)+(II)
\end{aligned}
\]
Term $(I)$ can be bounded by
\[
|(I)| \leq\frac{2p(p-1)}{p+1}C\int_{M}|u|^{\frac{p-1}{2}}|\nabla u^{\frac{p+1}{2}}|
\]
while term $(II)$ and be controlled by Cauchy-Schwarz
\[
|(II)| \leq \epsilon p\int_{M}u^{p-1}|\nabla Rm|^2 + Cp\epsilon^{-1}\int_{M}u^{p-1}
\]
where $C$ depends on $A, \hat{A}$.  Evidently both terms  $(I)$ and $(II)$ can be absorbed by the negative gradient terms in~\eqref{eq: evolIntup} to yield the bound
\[
\frac{d}{dt} \int_{M}u^p \sqrt{g}dx \leq - \frac{1}{10 p} \int_{M}|\nabla u^{\frac{p+1}{2}}|^2 + Cp^3\int u^{p} +u^{p-1}+ p\int_{M}fu^{p}
\]
As long as $f \in L^{q}$ for $q>\frac{n}{2}$, the final term can be controlled by interpolation.  Thus, in total we get
\[
\frac{d}{dt} \int_{M}u^p \sqrt{g}dx + \frac{1}{10 p}\int_{M}|\nabla u^{\frac{p+1}{2}}|^2\sqrt{g}dx \leq Cp^3(\int (u^{p}+1) \sqrt{g} dx).
\]
Now, the bounds on  $\psi, \nabla \psi, \nabla^2\psi, H, \nabla H, \nabla^2H$ and $\|\phi\|_{L^{\infty}}$ imply that $g(t)$ is uniformly equivalent to $g(0)$ and so the Sobolev constant of $g(t)$  on functions is uniformly bounded.  In particular we can apply parabolic Moser iteration to obtain a $C^0$.

Thus, it suffices to show that we have a uniform bound for $\int_{M}|Rm|^p \sqrt{g}dx$ for $p$ sufficiently large.  We will prove the following
\begin{lem}
Suppose that $|\nabla \psi|^2, |\nabla^2\psi|, |H|^2, |\nabla H|, |\nabla^2H|^{2/3}$ are bounded by $A$ along the flow on $[0,\tau]$ and $\|\phi\|_{L^{\infty}}$ is bounded by $\hat{A}$.  Then, for any $p \geq 3$ there exists a constant $C$, depending only $A, \hat{A}, \dim M$ and $p$ and an upper bound for $\tau$ such that
\[
\int_{M} |Rm(t)|^{p}\sqrt{g(t)}dx \leq C\left(\int_{M} |Rm(0)|^{p}\sqrt{g(0)}+ {\rm Vol}(M, g(0))\right).
\]
\end{lem}

This estimate will be derived by an ODE comparison argument.  Let us compute the evolution of the quantity on the left hand side.
\[
\begin{aligned}
\frac{d}{dt} \int |Rm|^p \sqrt{g(t)}dx &= \int \frac{p}{2} |Rm|^{p-2} \frac{d}{dt}|Rm|^2 +\int_{M} |Rm|^{p} \frac{1}{2}{\rm Tr}_{g}\dot{g} \sqrt{g}dx\\
\end{aligned}
\]
Since $\dot{g}$ is bounded (recall that bounds on $|\nabla^2\psi|, \|\phi\|_{L^{\infty}}$ imply bounds on $|Ric|$ by the Bochner formula; see Lemma~\ref{lem: BochnerApp})  we obtain
\[
\frac{d}{dt} \int |Rm|^p \sqrt{g(t)}dx \leq C\int_{M}|Rm|^p +   p\int |Rm|^{p-2} \langle\frac{d}{dt} Rm, Rm \rangle
\]
Now the key point is to use the evolution of $Rm$ in terms of $\nabla^2Ric$-- that is, to use~\eqref{eq: dtRm}.  This yields
\[
\frac{d}{dt}Rm = \mathcal{V}\left(-\frac{1}{8}{\rm Ric} +\frac{1}{8} e^{-2\phi}{\rm Hess}(e^{2\phi}) + e^{-2\phi}\cE + \cE'\right)+ Rm*\dot{g}
\]
where
\[
\mathcal{V}(h) = \frac{1}{2}\left(\nabla_{i}\nabla_{k}h_{j\ell} + \nabla_{j}\nabla_{\ell}h_{ik} - \nabla_{i}\nabla_{\ell}h_{jk} - \nabla_{j}\nabla_{k}h_{i\ell}\right).
\]
and $\cE, \cE'$ are defined in~\eqref{eq: errorDefn}.  The key idea that formula avoids introducing a term of the order $Rm^3$.  Now one can calculate (see, for example, the proof of Lemma~\ref{lem: evolRic} below) that
\[
\mathcal{V}({\rm Hess}(e^{2\phi})) = \nabla Ric *\nabla e^{2\phi} + Rm*{\rm Hess}(e^{2\phi})
\]
Thus, since $\dot{g}$ is bounded we have
\[
\begin{aligned}
&\frac{d}{dt} \int |Rm|^p \sqrt{g(t)}dx\\ &\leq C\int_{M}|Rm|^p +p\int_{M} |Rm|^{p-2} Rm* \mathcal{V}\left(-\frac{1}{8}Ric\right)\\
&\quad +p\int_{M}|Rm|^{p-2}Rm*\left(e^{-2\phi}\nabla Ric *\nabla e^{2\phi} + e^{-2\phi}Rm*{\rm Hess}(e^{2\phi})\right)\\
&\quad+ p\int_{M}|Rm|^{p-2}Rm*\nabla\left(\nabla e^{-2\phi}\nabla^2e^{2\phi} + \nabla (e^{-2\phi}\cE+\cE') \right)
\end{aligned}
\]
For the last line we integrate by parts to get
\[
\int_{M}|Rm|^{p-2}Rm*\nabla\left(\nabla e^{-2\phi}\nabla^2e^{2\phi} + \nabla (e^{-2\phi}\cE+\cE') \right) \leq  C\int_{M}|Rm|^{p-2}|\nabla Rm|
 \]
 for a constant $C$ depending on $p, A, \hat{A}, \dim M$. 
 \begin{rk}
 Note that since we only need to prove a bound for $\int_{M}|Rm|^p$ for $p$ large, we can ignore the dependence of the constants on $p$.
 \end{rk}
 We can deal with the first line similarly to get
 \[
 \int_{M} |Rm|^{p-2} Rm* \mathcal{V}\left(-\frac{1}{8}Ric\right) \leq C\int_{M}|Rm|^{p-2}|\nabla Rm||\nabla Ric| + |Rm|^{p-2}|\nabla Rm|
 \]
 Thus, in total we have
 \begin{equation}\label{eq: RmpBd1}
 \begin{aligned}
 \frac{d}{dt} \int |Rm|^p \sqrt{g(t)}dx &\leq C\int_{M}|Rm|^p + C\int_{M}|Rm|^{p-2}|\nabla Rm||\nabla Ric| \\
 &\quad +C\int_{M} |Rm|^{p-2}|\nabla Rm|\\
 &\quad +C\int_{M}|Rm|^{p-1}|\nabla Ric| \\
 &= C\int_{M}|Rm|^{p} + (I)+(II)+(III)
 \end{aligned}
 \end{equation}
 Now we use the evolution of $|Rm|^2$ in terms of the heat operator ~\eqref{eq: evoNormSqRm}, which we write in the form
 \begin{equation}\label{eq: RmpBd2}
\begin{aligned}
\frac{1}{8}|\nabla Rm|^2 &\leq \left(\frac{1}{16}\Delta -\frac{d}{dt}\right) |Rm|^2 + C(|Rm|+Rm|^2+ |Rm|^3) \\
&\quad +Rm*\nabla(\nabla(e^{-2\phi}\cE) +\nabla\cE' + \nabla e^{-2\phi}*\nabla^2e^{2\phi})
\end{aligned}
\end{equation}
 We shall also need the evolution of the Ricci curvature.
 
 \begin{lem}\label{lem: evolRic}
 Along the gauge modified spinor flows~\eqref{eq: modFlowVar} or~\eqref{eq: modFlowFix}, we have
 \[
 \begin{aligned}
 \frac{d}{dt} R_{jk} &= \frac{1}{16} \Delta R_{jk} +e^{-2\phi}\nabla Ric * \nabla e^{2\phi} + e^{-2\phi}Rm*{\rm Hess}(e^{2\phi}) + Rm*Ric\\
 &\quad +(\nabla^2e^{-2\phi})*\nabla^{2}e^{2\phi} + (\nabla e^{-2\phi})*\nabla^{3}e^{2\phi}\\ 
 &\quad + \nabla^2(e^{-2\phi}\cE+\cE') + Rm*(e^{-2\phi}\cE+\cE')
 \end{aligned}
 \]
 \end{lem}
 \begin{proof}
 Let
 \[
 A_{jk} = \frac{1}{8}\left(-R_{jk} +e^{-2\phi}{\rm Hess}(e^{2\phi})_{jk}\right).
 \]
 By Lemma~\ref{lem: varCurv} we have
 \[
 \begin{aligned}
 \frac{d}{dt} R_{jk} &= -\frac{1}{2}\left(\Delta A_{jk} + \nabla_j\nabla_k {\rm Tr}_{g}A\right) + \frac{1}{2}g^{im}\left(\nabla_i\nabla_kA_{jm} + \nabla_j\nabla_mA_{ik}\right) +Rm*A\\
 &\quad+\nabla^2(e^{-2\phi}\cE+\cE') + Rm*(e^{-2\phi}\cE+\cE')
 \end{aligned}
 \]
 where $\cE, \cE'$ are defined in~\eqref{eq: errorDefn}. We need to simplify the terms involving $A$.  First, for a smooth function $f$ we have
 \[
 \begin{aligned}
 \nabla_{i}\nabla_m\nabla_j\nabla_k f &= \nabla_i\nabla_j \nabla_m \nabla_k f + \nabla_i(R_{jmk}\,^s\nabla_sf)\\
 &=\nabla_{j}\nabla_i\nabla_m\nabla_k f+ Rm*\nabla^2f +\nabla_i(R_{jmk}\,^s\nabla_sf)\\
 &=\nabla_{j}\nabla_i\nabla_k\nabla_m f+ Rm*\nabla^2f +\nabla_i(R_{jmk}\,^s\nabla_sf) \\
 &= \nabla_{j}\nabla_k\nabla_i\nabla_m f + \nabla_{j}(R_{kim}\,^{s}\nabla_s f)+\nabla_i(R_{jmk}\,^s\nabla_sf)+Rm*\nabla^2f\\
 &= \nabla_{j}\nabla_k\nabla_i\nabla_m f +Rm*\nabla^2f +  \nabla_{j}R_{kim}\,^{s}\nabla_s f+\nabla_iR_{jmk}\,^s\nabla_sf
 \end{aligned}
 \]
 By the second Bianchi identity we have
 \[
 \begin{aligned}
 \nabla_iR_{jmk\ell} = -\nabla_{k}R_{jm\ell i} - \nabla_{\ell}R_{jmik}
  \end{aligned}
  \]
  Thus, in total we have
 \[
 \Delta \nabla_j\nabla_k f= \nabla_j \nabla_k \Delta f +(\nabla_jR_{\ell k} + \nabla_{k}R_{j\ell} - \nabla_{\ell}R_{jk})\nabla^{\ell}f + Rm*\nabla^2f
 \]
 Similarly we have
 \[
 \begin{aligned}
 \nabla_i\nabla_k\nabla_j\nabla_m f = \nabla_j \nabla_k\nabla_i\nabla_m f+ \nabla_k R_{jim\ell}\nabla^{\ell}f + Rm*\nabla^2f
 \end{aligned}
 \]
 which yields
 \[
 \begin{aligned}
 g^{im} \nabla_i\nabla_k\nabla_j\nabla_m f &=\nabla_j\nabla_k \Delta f + \nabla_kR_{j\ell} \nabla^{\ell}f + Rm*\nabla^2f\\
 g^{im} \nabla_j\nabla_m\nabla_i\nabla_k f &=\nabla_j\nabla_k \Delta f + \nabla_jR_{k\ell} \nabla^{\ell}f + Rm*\nabla^2f
 \end{aligned}
 \]
 Thus, we have
 \[
 \begin{aligned}
& -\frac{1}{2}\left(\Delta {\rm Hess}(e^{2\phi})_{jk} + \nabla_j\nabla_k \Delta e^{2\phi}\right) + \frac{1}{2}g^{im}\left(\nabla_i\nabla_k{\rm Hess}(e^{2\phi})_{jm} + \nabla_j\nabla_m{\rm Hess}(e^{2\phi})_{ik}\right)\\
 &= \nabla Ric * \nabla e^{2\phi} + Rm*{\rm Hess}(e^{2\phi})
 \end{aligned}
 \]
 Next we analyze the Ricci terms.  We have
 \[
 \begin{aligned}
 &\Delta(R_{jk}) + \nabla_{j}\nabla_{k}(R) - g^{im}(\nabla_{i}\nabla_{k}(R_{jm}) +\nabla_j \nabla_m (R_{ik})) \\
 &= \left(\Delta(R_{jk}) + \nabla_{j}\nabla_{k}R) - g^{im}(\nabla_{i}\nabla_{k}R_{jm} +\nabla_j \nabla_m R_{ik})\right)\\
 \end{aligned}
 \]
 Now the second Bianchi identity implies $\nabla^iR_{ik} = \frac{1}{2}\nabla_k R$ and so
 \[
 g^{im}\nabla_i\nabla_kR_{jm} = Rm*Ric + \frac{1}{2}\nabla_{k}\nabla_{j}R
 \]
 In total, we get
 \[
 \left(\Delta(R_{jk}) + \nabla_{j}\nabla_{k}R) - g^{im}(\nabla_{i}\nabla_{k}R_{jm} +\nabla_j \nabla_m R_{ik})\right)=\Delta R_{jk} + Rm*Ric
 \]
 Finally, putting everything together we obtain the desired formula.
 \end{proof}
 
We immediately obtain the following corollary.
 
 \begin{cor}\label{cor: evolNormRic}
 Suppose $ |\nabla \psi|^2, |\nabla^2\psi|, |H|^2, |\nabla H|, |\nabla^2H|^{2/3}$ are bounded by $A$ along the spinor flow with flux and assume that $\|\phi\|_{L^{\infty}}$ is bounded by $\hat{A}$.  Then we have
 \[
 \begin{aligned}
 \left(\frac{d}{dt} -\frac{1}{16}\Delta\right)|Ric|^2 &\leq -\frac{1}{16}|\nabla Ric|^2 + C(|Rm| +1)  \\
 &\quad +Ric* \nabla^2(e^{-2\phi}\cE+\cE') + Ric*\nabla(\nabla e^{-2\phi} \nabla^2e^{2\phi})
 \end{aligned}
 \]
 for a constant $C$ depending only on $\dim M, A, \hat{A}$.
 \end{cor}
 \begin{proof}
 The proof follows from Lemma~\ref{lem: evolRic} together with the fact that a bound on $\nabla^2\psi$ and a bound $\|\phi\|_{L^{\infty}}$ implies a bound on the Ricci curvature, thanks to the Bochner formula; see Lemma~\ref{lem: BochnerApp}.
 \end{proof}
 We can now explain how to estimate the right hand side of~\eqref{eq: RmpBd1}.  First, we bound term $(II)$
 \[
 (II)= \int_{M}|Rm|^{p-2}|\nabla Rm| \leq \int_{M} |Rm|^{p-3}|\nabla Rm|^2 + |Rm|^{p-1}
 \]
 By the bound~\eqref{eq: RmpBd2} we get
 \begin{equation}\label{eq: TermIIBnd1}
 \begin{aligned}
\int_{M} |Rm|^{p-3}|\nabla Rm|^2 &\leq C\int_{M}|Rm|^{p}+|Rm|^{p-1}+ |Rm|^{p-2}\\
& C\int_{M} |Rm|^{p-3}\left(\frac{1}{16}\Delta - \frac{d}{dt}\right)|Rm|^2\\
  &+C\int |Rm|^{p-3}Rm*\nabla(\nabla(e^{-2\phi}\cE) +\nabla\cE' + \nabla e^{-2\phi}*\nabla^2e^{2\phi}) \\
  &= C\int_{M}|Rm|^p+|Rm|^{p-1}+|Rm|^{p-2}+ C\left((IIa)+(IIb)\right)
  \end{aligned}
 \end{equation}
We bound $(IIb)$ by integration by parts.  We have
 \[
 \begin{aligned}
(IIb) &= -(p-3) \int |Rm|^{p-5}Rm*\nabla Rm*Rm*(\nabla(e^{-2\phi}\cE) +\nabla\cE' + \nabla e^{-2\phi}*\nabla^2e^{2\phi})\\
 &\quad+\int |Rm|^{p-3}*\nabla Rm*(\nabla(e^{-2\phi}\cE) +\nabla\cE' + \nabla e^{-2\phi}*\nabla^2e^{2\phi})\\
 & \leq C\int_{M} |Rm|^{p-3}|\nabla Rm|\\
 &\leq \epsilon \int_{M} |Rm|^{p-3}|\nabla Rm|^2 +C^2\epsilon^{-1}\int_{M} |Rm|^{p-3}
 \end{aligned}
 \]
 where we used that $\nabla(e^{-2\phi}\cE) +\nabla\cE' + \nabla e^{-2\phi}*\nabla^2e^{2\phi}$ are bounded by constants depending only on $A, \hat{A}$.  Choosing $\epsilon$ sufficiently small we can absorb the $|Rm|^{p-3}|\nabla Rm|^2$ term on the left hand side of~\eqref{eq: TermIIBnd1} to arrive at
 \begin{equation}\label{eq: termIIBnd2}
 \begin{aligned}
 \int_{M} |Rm|^{p-3}|\nabla Rm|^2 &\leq C\int_{M}|Rm|^{p} +|Rm|^{p-1} + |Rm|^{p-2}+ |Rm|^{p-3}\\
 &- C\frac{d}{dt}\int_{M} |Rm|^{p-1} + C\int_{M}|Rm|^{p-3}\frac{1}{16}\Delta |Rm|^2.
\end{aligned}
\end{equation}
Here we also used that $\frac{d}{dt}\sqrt{g}$ is bounded in order to bring the $\frac{d}{dt}$ outside the integral.  It only remains to bound the Laplacian term. If $p\geq 3$, then integration by parts shows
\[
\begin{aligned}
\int_{M} |Rm|^{p-3}\frac{1}{16}\Delta |Rm|^2 \leq 0
\end{aligned}
\]
and we finally obtain
\begin{equation}\label{eq: Term0.5}
\begin{aligned}
 \int_{M} |Rm|^{p-3}|\nabla Rm|^2 &\leq C\int_{M}|Rm|^{p} +|Rm|^{p-1} +|Rm|^{p-2}+ |Rm|^{p-3}\\
 &\quad - C\frac{d}{dt}\int_{M} |Rm|^{p-1}.
 \end{aligned}
 \end{equation}
In total, this yields the bound for term $(II)$
 \begin{equation}\label{eq: Term1}
 \begin{aligned}
(II)= \int  |Rm|^{p-2}|\nabla Rm|&\leq C\int_{M}|Rm|^{p} +|Rm|^{p-1} +|Rm|^{p-2}+ |Rm|^{ p-3}\\
&\quad- C\frac{d}{dt}\int_{M} |Rm|^{p-1}
\end{aligned}
 \end{equation}
 Next we turn our attention to bounding term $(III)$.
 \[
(III)=  \int_{M}|Rm|^{p-1}|\nabla Ric|\leq \int_{M}|Rm|^{p-1}|\nabla Ric|^2+ |Rm|^{p-1}.
 \]
  Using Corollary~\ref{cor: evolNormRic} we can bound
  \begin{equation}\label{eq: TermIIIBnd1}
  \begin{aligned}
  \int_{M}|Rm|^{p-1}|\nabla Ric|^2&\leq C\int_{M}|Rm|^p + |Rm|^{p-1} + 16\int_M |Rm|^{p-1}\left(\frac{1}{16}\Delta - \frac{d}{dt}\right)|Ric|^2\\
  &\quad + C\int |Rm|^{p-1} \left(Ric* \nabla^2(e^{-2\phi}\cE+\cE')\right)\\
  &\quad+C\int_{M}|Rm|^{p-1}Ric*\nabla(\nabla e^{-2\phi} \nabla^2e^{2\phi})\\
  &=C\int_{M}|Rm|^p + |Rm|^{p-1} + (IIIa)+ (IIIb)
  \end{aligned}
  \end{equation}
  Terms $(IIIa)$ and $(IIIb)$ are dealt with in an identical fashion, we only explain the bound for $(IIIb)$.  Integration by parts yields
  \[
  \begin{aligned}
 (IIIb)&= \int |Rm|^{p-1}Ric * Ric*\nabla(\nabla e^{-2\phi} \nabla^2e^{2\phi})\\
 &= -\int |Rm|^{p-1}\nabla Ric *(\nabla e^{-2\phi} \nabla^2e^{2\phi})\\
  &\quad  - \int |Rm|^{p-3}Rm*\nabla Rm *Ric *(\nabla e^{-2\phi} \nabla^2e^{2\phi})\\
  & \leq \epsilon \int_{M}|Rm|^{p-1}|\nabla Ric|^2 + C^2\epsilon^{-1}\int |Rm|^{p-1}\\
  &\quad +\int |Rm|^{p-2}|\nabla Rm|
  \end{aligned}
  \]
  where we used again that $Ric$ is bounded in terms of bounds for $\nabla^2\psi, \phi$. Choosing $\epsilon$ small the term $|Rm|^{p-1}|\nabla Ric|^2$ can be absorbed on the right hand side of~\eqref{eq: TermIIIBnd1},  while the final term can be bounded using~\eqref{eq: Term1} to get
   \begin{equation}\label{eq: termIIIbnd2}
  \begin{aligned}
  \int_{M}|Rm|^{p-1}|\nabla Ric|^2&\leq C\int_M |Rm|^{p-1}\left(\frac{1}{16}\Delta - \frac{d}{dt}\right)|Ric|^2\\
   &\quad+C\int_{M}|Rm|^p+|Rm|^{p-1} +|Rm|^{p-2}+ |Rm|^{ p-3} \\
   &-C\frac{d}{dt}\int_{M} |Rm|^{p-1}
  \end{aligned}
  \end{equation}
  We now estimate the Laplacian term using integration by parts.
  \[
  \begin{aligned}
  \int_M |Rm|^{p-1}\Delta |Ric|^2 &= -\frac{p-1}{2} \int_{M}|Rm|^{p-3}Rm*\nabla Rm*Ric*\nabla Ric\\
  &\leq C\int_{M}|Rm|^{p-2}|\nabla Rm||\nabla Ric| 
  \end{aligned}
  \]
  Since we have the trivial inequalities
  \[
  \begin{aligned}
   |Rm|^{p-1}|\nabla Ric|&\leq \epsilon|Rm|^{p-1}|\nabla Ric|^2 + \epsilon^{-1}|Rm|^{p-1}\\
   |Rm|^{p-2}|\nabla Rm||\nabla Ric| &\leq \epsilon |Rm|^{p-1}|\nabla Ric|^2 +\epsilon^{-1}|Rm|^{p-3}|\nabla Rm|^2
   \end{aligned}
   \]
    we can leverage the bound~\eqref{eq: Term0.5} conclude that
     \begin{equation}\label{eq: Term2HalfWay}
  \begin{aligned}
  \int_{M}|Rm|^{p-1}|\nabla Ric|^2&\leq -C \int_{M}|Rm|^{p-1}\frac{d}{dt}|Ric|^2  \\
  &\quad  +C\int_{M}\left(|Rm|^{p}+ |Rm|^{p-1} +|Rm|^{p-2}+ |Rm|^{ p-3} \right)\\
  &\quad- C\frac{d}{dt}\int_{M} |Rm|^{p-1}
  \end{aligned}
  \end{equation}

The goal now is to replace the term $\int_{M}|Rm|^{p-1}\frac{d}{dt}|Ric|^2 $ with a total derivative.  We have
\begin{equation}\label{eq: TotDerStep1}
\begin{aligned}
-\int_{M}|Rm|^{p-1}\frac{d}{dt}|Ric|^2 &= -\frac{d}{dt}\int_{M}|Rm|^{p-1}|Ric|^2\\ 
&\quad+ \int_{M}|Ric|^2 \frac{(p-1)}{2}|Rm|^{p-3}\frac{d}{dt}|Rm|^2 \\
&\quad + \int_{M}|Ric|^2|Rm|^{p-1}\frac{1}{2}{\rm Tr}_{g}\dot{g}\sqrt{g}
\end{aligned}
\end{equation}
Since $\dot{g}$ is bounded and $Ric$ is bounded we have
\[
\int_{M}|Ric|^2|Rm|^{p-1}\frac{1}{2}{\rm Tr}_{g}\dot{g}\sqrt{g} \leq C\int_{M}|Rm|^{p-1}.
\]
Therefore, we only need to bound from above the term$\int_{M}|Ric|^2 |Rm|^{p-3}\frac{d}{dt}|Rm|^2 $.  To bound this term we use the evolution~\eqref{eq: evol|Rm|^2withBnds}.  This yields
\begin{equation}\label{eq: TorDerStep2}
\begin{aligned} 
&\int_{M}|Ric|^2 |Rm|^{p-3}\frac{d}{dt}|Rm|^2\\ &\leq - \int_{M}|Ric|^2 |Rm|^{p-3}\frac{1}{8} |\nabla Rm|^2\\
&\quad+C \int_{M} |Ric|^2(|Rm|^p+|Rm|^{p-1}+|Rm|^{p-2}) \\
&\quad +\leq \int_{M}|Ric|^2 |Rm|^{p-3}\frac{1}{16}\Delta |Rm|^2\\
&\quad+\int_{M} |Ric|^2|Rm|^{p-3}Rm*\nabla(\nabla(e^{-2\phi}\cE) + \nabla \cE' + \nabla e^{-2\phi}\nabla^2e^{2\phi})\\
&= - \int_{M}|Ric|^2 |Rm|^{p-3}\frac{1}{8} |\nabla Rm|^2+ C\int_{M} |Rm|^{p} + |Rm|^{p-1} + |Rm|^{p-2} + (IVa) +(IVb)
\end{aligned}
\end{equation}
The term$(IVa)$ is easily dealt with by integration by parts.  Let us first deal with the Laplacian term $(IVa)$.  Integration by parts yields
\[
\begin{aligned}
(IVa)    &= -\int_{M} \frac{1}{16}|Rm|^{p-3}\langle \nabla |Ric|^2,\nabla |Rm|^2  \rangle\\
&\quad -\frac{(p-3)}{2}\int_{M}|Ric|^2 \frac{1}{16}|Rm|^{p-5}\langle \nabla |Rm|^2,\nabla |Rm|^2  \rangle
\end{aligned}
\]
If we take $p \geq 3$ the the final term is negative, while the first terms is easily bounded to give
\begin{equation}\label{eq: TermIVaBnd}
\begin{aligned}
(IVa) &\leq C\int_{M} |Rm|^{p-2}|\nabla Rm| +C\int_{M} |Rm|^{p-2}|\nabla Ric||\nabla Rm|\\
& \leq  C\int_{M} |Rm|^{p-2}|\nabla Rm| \\
&\quad+C^2\epsilon^{-1}\int_{M} |Rm|^{p-3}|\nabla Rm|^2 \\
&\quad + \epsilon \int_{M}|Rm|^{p-1}|\nabla Ric|^2
\end{aligned}
\end{equation}
Next we consider the term $(IVb)$. Integration by parts gives
\begin{equation}\label{eq: TermIVbBnd}
\begin{aligned}
(IVb)&=-\int_{M} |Ric|^2|Rm|^{p-3}\nabla Rm*(\nabla(e^{-2\phi}\cE) + \nabla \cE' + \nabla e^{-2\phi}\nabla^2e^{2\phi})\\
&\quad-\int_{M}  Ric* \nabla Ric* |Rm|^{p-3} Rm*(\nabla(e^{-2\phi}\cE) + \nabla \cE' + \nabla e^{-2\phi}\nabla^2e^{2\phi})\\
&\quad-\frac{(p-3)}{2}\int_{M} |Ric|^2|Rm|^{p-5}Rm^2*\nabla Rm*(\nabla(e^{-2\phi}\cE) + \nabla \cE')\\
&\quad + -\frac{(p-3)}{2}\int_{M}|Ric|^2|Rm|^{p-5}Rm^2*\nabla Rm*(\nabla e^{-2\phi}\nabla^2e^{2\phi})\\
&\leq C\int_{M}|Rm|^{p-3}|\nabla Rm| +C\int_{M}|Rm|^{p-2}|\nabla Ric|\\
&\leq C\int_{M}|Rm|^{p-3}|\nabla Rm|^2 +C\int_{M}|Rm|^{p-3} + |Rm|^{p-1}
\end{aligned}
\end{equation}
where in the last line we used Young's inequality and the bound $|\nabla Ric| \leq n|\nabla Rm|$ which follows from Cauchy-Schwarz.  Thus, combining ~\eqref{eq: TermIVaBnd},~\eqref{eq: TermIVbBnd} with~\eqref{eq: TotDerStep1}
and~\eqref{eq: TorDerStep2} we arrive at
\begin{equation}\label{eq: TotDerStep3}
\begin{aligned}
-\int_{M}|Rm|^{p-1}\frac{d}{dt}|Ric|^2 &\leq -\frac{d}{dt}\int_{M}|Rm|^{p-1}|Ric|^2 + C\int_{M}|Rm|^p + |Rm|^{p-1} + |Rm|^{p-3}\\
&+C^2\epsilon^{-1} \int_{M}|Rm|^{p-3}|\nabla Rm|^{2} + \epsilon \int_{M} |Rm|^{p-1}|\nabla Ric|^2\\
\end{aligned}
\end{equation}
Finally, plugging this bound into~\eqref{eq: Term2HalfWay} yields
\[
  \begin{aligned}
  \int_{M}|Rm|^{p-1}|\nabla Ric|^2&\leq C^2\epsilon^{-1} \int_{M}|Rm|^{p-3}|\nabla Rm|^{2} + \epsilon \int_{M} |Rm|^{p-1}|\nabla Ric|^2  \\
  &\quad  C\int_{M}\left(|Rm|^{p}+ |Rm|^{p-1} +|Rm|^{p-2}+ |Rm|^{ p-3} + |Rm|^{p-5}\right)\\
  &\quad- C\frac{d}{dt}\int_{M} |Rm|^{p-1} -C\frac{d}{dt}\int_{M}|Rm|^{p-1}|Ric|^2
  \end{aligned}
\]
Again, choosing $\epsilon$ small we can absorb the $|Rm|^{p-1}|\nabla Ric|^2$ term on the left hand side.  Combining this with the bound~\eqref{eq: Term0.5} we obtain
\begin{equation}\label{eq:TermIIIFinalBnd}
  \begin{aligned}
  \int_{M}|Rm|^{p-1}|\nabla Ric|^2&\leq C\int_{M}\left(|Rm|^{p}+ |Rm|^{p-1} +|Rm|^{p-2}+ |Rm|^{ p-3}\right)\\
  &\quad- C\frac{d}{dt}\int_{M} |Rm|^{p-1} -C\frac{d}{dt}\int_{M}|Rm|^{p-1}|Ric|^2
  \end{aligned}
\end{equation}
Finally, we can bound term $(I)$ by the inequality
\[
 \int_M |Rm|^{p-2}|\nabla Rm||\nabla Ric| \leq \int_{M}|Rm|^{p-3}|\nabla Rm|^2 +  \int_{M}|Rm|^{p-1}|\nabla Ric|^2
\]
and then apply~\eqref{eq: Term1},~\eqref{eq:TermIIIFinalBnd} to obtain
\[
\begin{aligned}
&\frac{d}{dt} \int_{M}|Rm|^{p} + C_1|Rm|^{p-1}|Ric|^2+ C_2|Rm|^{p-1}\\
& \leq C\int_{M}\left(|Rm|^{p}+ |Rm|^{p-1} +|Rm|^{p-2}+ |Rm|^{ p-3} + |Rm|^{p-5}\right)\\
& \leq C\int_{M}(|Rm|^{p} +1)\\
&\leq \int_{M}|Rm|^{p} + C_1|Rm|^{p-1}|Ric|^2+ C_2|Rm|^{p-1}+ 1
\end{aligned}
\]
Now since ${\rm Vol}(g(t))$ is bounded along the flow by Lemma~\ref{lem: easyVolBnd}, the result follows from ODE comparison.

In light of the discussion at the beginning of this section we obtain

\begin{thm}
Suppose that $|\nabla \psi|^2, |\nabla^2\psi|, |H|^2, |\nabla H|, |\nabla^2H|^{2/3}$ are bounded on $[0,\tau]$ by $A$ along either of the spinor flows ~\eqref{flow} or ~\eqref{flow1}.  Suppose in addition that $\|\phi\|_{L^{\infty}}$ is bounded by $\hat{A}$ on on $[0,\tau]$.  Then there exists a constant $C$, depending only $A, \hat{A}, \dim M$, an upper bound for $\tau$ and 
\[
\kappa:= \int_{M} |Rm(0)|^{p}\sqrt{g(0)}+ {\rm Vol}(M, g(0)),
\]
such that $|Rm(t)|_{g(t)} \leq C$ uniformly on $[0,\tau)$.
\end{thm}

\appendix
\section{Notations and Conventions}\label{sec: conv}
\subsection{Notation}
We briefly list the notation used in this paper.  For more details on spinor bundles and the gamma matrices $\gamma^a$ we refer the reader to~\ref{subsec: CliffConv} below.
\begin{itemize}
\item For a tensor $T_{i_1\ldots i_k}\,^{j_1\ldots j_s}$ indices $p, \ell$ we denote by $\{p\ell\}$ the symmetrization of $T$ in the $p, \ell$ slots, possibly after lowering an index using the metric $g$.  That is, in an orthonormal frame for $g$
\[
T_{i_1 \ldots \{p \ldots i_k}\\,^{j_1\ldots \ell\} \ldots j_s} = \frac{1}{2}\left(T_{\{i_1 \ldots p \ldots i_k}\\,^{j_1\ldots \ell \ldots j_s} + T_{\{i_1 \ldots \ell \ldots i_k}\\,^{j_1\ldots p \ldots j_s} \right)
\]
\item Similarly, for a tensor $T_{i_1\ldots i_k}\,^{j_1\ldots j_s}$ indices $p, \ell$ we denote by $[p\ell]$ the anti-symmetrization of $T$ in the $p, \ell$ slots, possibly after lowering an index using the metric $g$.  For example, in an orthonormal frame for $g$
\[
T_{i_1 \ldots [p \ldots i_k}\\,^{j_1\ldots \ell] \ldots j_s} = \frac{1}{2}\left(T_{i_1 \ldots p \ldots i_k}\\,^{j_1\ldots \ell \ldots j_s} - T_{i_1 \ldots \ell \ldots i_k}\\,^{j_1\ldots p \ldots j_s} \right)
\]
\item For $\gamma$ matrices $\gamma^{b_1}, \ldots \gamma^{b_{N}}$ we denote by $\gamma^{b_1\cdots b_{N}}$ the anti-symmetrized product
\[
\frac{1}{N!}\sum_{\sigma \in \mathfrak{S}_n} (-1)^{|\sigma|}\gamma^{b_{\sigma(1)}}\cdots \gamma^{b_{\sigma(N)}}
\]
where $\mathfrak{S}_n$ denotes the symmetric group on $n$-symbols.
\item $\star$ denotes the Hodge star operator, while $*$ denotes any tensorial contraction involving the metric, multiplication by fixed constants, Clifford multiplication by unit vectors, and the metric on the spinor bundle
\item Given a $k$-form $H$ and a vector $\lambda= (\lambda_1, \lambda_2) \in \mathbb{R}^2$ we denote by
\[
(\lambda|H)_a := (\lambda_1H_{ab_1\cdots b_{k-1}}\gamma^{b_1\cdots b_{k-1}}+\lambda_2H_{b_1\cdots b_k}\gamma^{ab_1\cdots b_k})
\]
then ${\rm End}(S)$ valued $1$-form constructed from $H$.  We also denote by $h_a$ the hermitian part of $(\lambda|H)_{a}$.  Explicitly, subject to our conventions, we have
\[
h_{a} =\begin{cases} \lambda_1H_{ab_1\cdots b_{k-1}}\gamma^{b_1\cdots b_{k-1}} & \text{ if } k=1,2 \mod 4\\
 \lambda_2H_{b_1\cdots b_k}\gamma^{ab_1\cdots b_k} & \text{ if } k=0,3 \mod 4
 \end{cases}
\]

\end{itemize}

\subsection{Conventions for the Clifford algebra}\label{subsec: CliffConv}
We briefly recall the basics of spin geometry in order to fix notation.  For a thorough introduction to spin geometry we refer the reader to \cite{F, LM, Harv}. Let $(M,g)$ be a Riemannian manifold of dimension $n$.  Throughout this paper we take the Clifford algebra to be the algebra obtained by taking the quotient of the tensor algebra $\oplus_{r=0}^{n} TM^{\otimes r}$ by the ideal generated by elements
\[
v\otimes w + w\otimes v - g(v,w)
\]
where $v, w \in TM$.  Note that this is the {\em opposite} of the conventions used in \cite{AWW, HeWang, BG}.  Since $M$ is spin it admits principal ${\rm Spin}(n)$ bundle with a $2\rightarrow 1$ covering of the bundle $P_{SO(n)}$ of orthonormal frames.  A spinor bundle $S= P_{{\rm Spin}(n)} \times_{\rho} V$ is a vector bundle associated to the spinor representation $\rho: Cl(R^n, g_{e}) \rightarrow V$ where $g_e$ is the Euclidean metric.  Concretely, this means that for any choice of local orthonormal frame on $M$ given by $\{e_1,\ldots, e_n\}$ there is a map $\rho: Cl(TM, g) \rightarrow {\rm End}(S)$ given by
\[
\rho(e_a) = \gamma^a
\]
where $\gamma^a$ are local endomorphisms of $S$ satisfying the Clifford algebra relation
\bea
\gamma^a\gamma^b+\gamma^b\gamma^a= 2\delta_{ab}
\eea
With this notation, there is a map $\rho: \Lambda^kT^*M \rightarrow {\rm End}(S)$ given by first identifying $\Lambda^kT^*M \sim \Lambda^kTM$ using $g$, and then including $\Lambda^kTM$ into $ \oplus_{r=0}^{n} TM^{\otimes r}$ by the map
\[
e_{i_1}\wedge \cdots \wedge e_{i_p} \mapsto \sum_{\sigma \in \mathfrak{S}_p} (-1)^{|\sigma|} e_{i_{\sigma(1)}}\otimes \cdots \otimes e_{i_{\sigma(p)}}.
\]
With this convention it is straight forward to check that a $p$-form $\omega$ acts by
\[
\omega \cdot \psi = \sum_{(i_1,\ldots, i_p)} \omega(e_{i_1}, \ldots,e_{i_{p}})\gamma^{i_1\cdots i_p}\psi
\]
where the sum is over all $p$-tuples.

$S$ comes with a connection induced by the Levi-Civita connection given explicitly in an orthonormal frame $\{e_a\}$ by
\[
\nabla_{e_a} \psi = e_{a}\psi - \frac{1}{4}\Gamma_{abc}\gamma^{bc}\psi
\]
With respect to this connection the map $\rho: Cl(TM, g) \rightarrow {\rm End}(S)$ is parallel.  Furthermore, $S$ comes equipped with an inner product $\langle \cdot, \cdot \rangle$ satisying
\[
\begin{aligned}
\nabla_{e_a} \langle \psi_1, \psi_2 \rangle &= \langle \nabla_{e_a}\psi_1, \psi_2 \rangle + \langle \psi_1, \nabla_{e_a}\psi_2\rangle\\
\langle \gamma^a \psi_1, \psi_2 \rangle &= \langle \psi_1, \gamma^a \psi_2 \rangle.
\end{aligned}
\]
In particular, with our convention for the Clifford algebra the $\gamma^a$ are hermitian endomorphisms of $S$.

\subsection{Conventions for the curvature and Bochner formulas}

We fix here the conventions for the curvature.  The Riemann curvature will be given by the convention
\[
[\nabla_j, \nabla_i]\del_{\ell} = R_{ij}\,^{k}\,_{\ell} \del_{k}.
\]
The Ricci tensor is given by
\[
R_{j\ell} = R_{ij}\,^{i}_{\ell}.
\]
Throughout the paper we have used the following well-known Bochner formula.  We provide a proof for completeness. 

\begin{lem}\label{lem: BochnerApp}
For any smooth spinor $\psi$ we have the following Bochner formula
\[
\slashed{D}\nabla_p\psi  - \nabla_p\slashed{D}\psi = \frac{1}{2}R_{pk}\gamma^k \psi.
\]
\end{lem}
\begin{proof}
We work at a point $x \in M$, and in an orthonormal frame $\{e_a\}$ such that $\nabla_{e_a}e_{b}(x)=0$.  In this case we have
\[
\slashed{D}\nabla_p\psi  - \nabla_p\slashed{D}\psi  = \gamma^{k}(\nabla_k\nabla_p - \nabla_p\nabla_k)\psi
\]
Now, since the Levi-Civita connection on the spinor bundle is given by $\nabla_p = \del_p - \frac{1}{4}\Gamma_{pab}\gamma^{ab}$we have
\[
\slashed{D}\nabla_p\psi  - \nabla_p\slashed{D}\psi=-\frac{1}{4}R_{pkab}\gamma^k\gamma^{ab}\psi
\]
Consider the sum
\[
R_{pkab}\gamma^k\gamma^{ab} = \sum_{k<a}R_{pkab}\gamma^k\gamma^{ab} + \sum_{k>a}R_{pkab}\gamma^k\gamma^{ab} + \sum_{a}R_{paab}\gamma^k\gamma^{ab}
\]
Clearly we have $\sum_{a}R_{paab}\gamma^k\gamma^{ab} = -R_{pb}\gamma^b$.  The first sum can be rewritten as
\[
\begin{aligned}
 \sum_{k<a}R_{pkab}\gamma^k\gamma^{ab} &= \sum_{k<a}R_{pkab}\gamma^k\gamma^{a}\gamma^{b}\\
 &=  \sum_{k>a}R_{pakb}\gamma^{a}\gamma^{k}\gamma^{b}\\
 &= \sum_{k>a}R_{pabk}\gamma^{k}\gamma^{a}\gamma^{b}
\end{aligned}
\]
where in the first line we reindexed $a\leftrightarrow k$, and in the second line we used $a \ne k$ to commute the $\gamma$ matrices together with $R_{pakb}=-R_{pabk}$.  By the Bianchi identity we have
\[
R_{pabk}+R_{pkab}= -R_{pbka}
\]
and so we have
\[
\begin{aligned}
R_{pkab}\gamma^k\gamma^{ab} &=-R_{pb}\gamma^b - \sum_{k>a}R_{pbka}\gamma^{k}\gamma^{a}\gamma^{b}\\
&=-R_{pb}\gamma^b - \frac{1}{2}\sum_{k,a,b}R_{pbka}\gamma^{k}\gamma^{a}\gamma^{b}\\
&= -R_{pb}\gamma^b - \frac{1}{2}\sum_{k,a,b}R_{pbka}\gamma^{k}(-\gamma^{b}\gamma^{a}+2\delta_{ab})\\
&=-R_{pb}\gamma^b +\frac{1}{2}\sum_{k,a,b}R_{pbka}\gamma^{k}\gamma^{b}\gamma^{a}-\sum_{k,a}R_{paka}\gamma^k\\
&=-2R_{pb}\gamma^b+\frac{1}{2}\sum_{k,a,b}R_{pbka}\gamma^{k}\gamma^{b}\gamma^{a}\\
&=-2R_{pb}\gamma^b+\frac{1}{2}\sum_{k,a,b}R_{pbka}(-\gamma^b\gamma^k+2\delta_{bk})\gamma^{a}\\
&=-3R_{pb}\gamma^b - \frac{1}{2}\sum_{k,a,b}R_{pbka}\gamma^b\gamma^{ka}
\end{aligned}
\]
Reindexing on the right hand side and solving yields
\[
R_{pkab}\gamma^k\gamma^{ab}=-2R_{pb}\gamma^b
\]
and the result follows.
\end{proof}

\subsection{The Bourguignon-Gauduchon Connection}\label{sec: BGApp}$\,$\\

The Bourguignon-Gauduchon connection, introduced in \cite{BG} gives a way of inducing spinor bundles from a one-parameter family of metric $g(t)$ and a spinor bundle structure associated to $g(0)$.  Indeed, some thought will convince the reader that there is no intrinsic way to do this. If $g(t)$ is a one-parameter family of Riemannian metrics for $t\in(-\epsilon, \epsilon)$ and we write $\langle X, Y \rangle_{g(t)} = \langle w(t)X, w(t)Y \rangle_{g(0)}$ for some $w(t) \in {\rm End}(TM)$.  Then $w^{-1}\dot{w}\in{\rm End}(TM)$ has a decomposition as
\[
w^{-1}(t)\dot{w}(t) = A(t)^{s} + A(t)^{a.s.}
\]
defined by
\[
\langle A(t)^s X, Y \rangle_{g(t)} = \langle X, A(t)^sY \rangle_{g(t)}, \qquad \langle A(t)^{a.s.} X, Y \rangle_{g(t)} =- \langle X, A(t)^{a.s.}Y \rangle_{g(t)}.
\]
The Bourguignon-Gauduchon connection is defined on $TM$ by
\[
\nabla^{BG}_{t} v = \del_t v +A(t)^sv
\]
Note that if $\{e_a\}$ is a $g(0)$ orthonormal frame, and $\{e_a(t)\}$ denotes the frame obtained by parallel transport using $\nabla^{BG}$, then we have
\[
\begin{aligned}
\frac{d}{dt} \langle e_a(t), e_b(t) \rangle_{g(t)} &=  \langle (A^s+ A^{a.s.})e_a(t), e_b(t) \rangle_{g(t)} + \langle e_a(t), (A^s+ A^{a.s.})e_b(t) \rangle_{g(t)}\\
&-\langle A(t)^s e_a(t), e_b(t) \rangle_{g(t)} - \langle e_a(t),  A(t)^se_b(t) \rangle_{g(t)}\\
&=0
\end{aligned}
\]
and hence the parallel transport of a $g(0)$-orthonormal frame is $g(t)$-orthonormal.  Furthermore, $\nabla^{BG}$ preserves the ideal defining the Clifford algebra.  This allows us to define spin structures by defining the map 
\[
\rho(t): Cl(TM, g(t)) \rightarrow {\rm End}(S)
\]
by the formula
\[
\rho(t)(e_a(t)) =\rho(0)(e_a(0)) = \gamma^a.
\]
That is, $\rho(t)$ is obtained by parallel transport along $\nabla^{BG}$.  The key point is that the Bourguignon-Gauduchon connection provides a natural way of defining the action on the spinor bundle of orthonormal frames defined with respect to time varying metrics.

\end{document}